\documentclass[12pt]{amsart}
\usepackage{amsmath,amsthm,amsfonts,latexsym,amssymb,amscd,color}
\usepackage{chemarr}
\newtheorem{thm}{Theorem}[section]
\newtheorem{cor}[thm]{Corollary}
\newtheorem{lm}[thm]{Lemma}

\newtheorem{clm}[thm]{Claim}
\newtheorem*{clm*}{Claim}
\theoremstyle{definition}

\newtheorem{exmp}[thm]{Example}

\newtheorem{remrks}[thm]{Remarks}
\numberwithin{equation}{section}
\DeclareMathOperator{\atyp}{\mathbf{2}}
\DeclareMathOperator{\btyp}{\mathbf{3}}
\DeclareMathOperator{\ltyp}{\mathbf{4}}

\newcommand{\m}[1]{{\mathbf{\uppercase{#1}}}}
\newcommand{\al}[1]{\mathbf{#1}}
\newcommand{\rel}[1]{{\mathcal{\uppercase{#1}}}}

\DeclareMathOperator{\Con}{Con}
\DeclareMathOperator{\Aut}{Aut}
\DeclareMathOperator{\Hom}{Hom}
\DeclareMathOperator{\id}{\textrm{id}}     

\newcommand{\var}[1]{{\mathcal{\uppercase{#1}}}}
\newcommand{\Ovar}[1]{\var{#1}_{\mathcal{O}}}
\newcommand{\Su}{\mathsf{S}}
\newcommand{\Pd}{\mathsf{P}}

\newcommand{\lv}{{[\kern-1.8pt[}}  
\newcommand{\rv}{{]\kern-1.8pt]}}
\newcommand{\interval}[2]{\lv\, {#1},{#2}\,\rv}

\newcommand{\restr}{{\restriction}}
\newcommand{\proj}{\mathrm{proj}}

\newcommand{\crit}{{\rm crit}}

\newcommand{\abgr}[3]{\al{M}_{\al #1}(#2,#3)}

\newcommand{\Calpha}[1]{\abgr{C}{\alpha}{#1}}

\newcommand{\module}[2]{\abgr{#1}{#2}{\mathcal{O}}}
\newcommand{\MBalpha}{\module{B}{\alpha_{\al{B}}}}
\newcommand{\MBsatalpha}{\module{B[\alpha]}{\alpha_{\al{B[\alpha]}}}}
\newcommand{\MBialpha}{\module{B_i}{\alpha_i}}
\newcommand{\MBbaralpha}{\al{M}_{\bar{\al B}}(\bar{\alpha},\mathcal{O})}
\newcommand{\MCalpha}{\module{C}{\alpha}}
\newcommand{\MCCalpha}{\module{C}{\alpha_{\al{C}}}}
\newcommand{\MUalpha}{\module{U}{\alpha_{\al{U}}}}
\newcommand{\MValpha}{\module{V}{\alpha_{\al{V}}}}
\newcommand{\MEalpha}{\module{E}{\alpha_{\al{E}}}}
\newcommand{\MTalpha}{\module{T}{\alpha}}
\newcommand{\dentails}{\models_{\text{\rm d}}}
\newcommand{\Mod}{\text{\rm Mod}}
\newcommand{\setMBalpha}{M_{\al B}(\alpha_{\al B},\mathcal{O})}

\newcommand{\wec}[1]{{\mathbf{#1}}}  
\newcommand{\tupl}[2]{\lceil{#1}\rceil^{#2}}
\newcommand{\otupl}[1]{\tupl{#1}{o}}

\newcommand{\aaa}{\mathsf{a}}
\newcommand{\ccc}{\mathsf{c}}
\newcommand{\ee}{\mathsf{e}}
\newcommand{\hh}{\mathsf{h}}
\newcommand{\ii}{\mathsf{i}}
\newcommand{\pp}{\mathsf{p}}
\newcommand{\sss}{\mathsf{s}}

\let\phi=\varphi
\let\epsilon=\varepsilon
\let\bar=\overline
\let\hat=\widehat
\let\tilde=\widetilde

\begin{document}

\title[Dualizable algebras]{Dualizable algebras with parallelogram terms}

\author{Keith A. Kearnes}
\address[Keith Kearnes]{Department of Mathematics\\
University of Colorado\\
Boulder, CO 80309-0395\\
USA}
\email{Keith.Kearnes@Colorado.EDU}

\author{\'Agnes Szendrei}
\address[\'Agnes Szendrei]{Department of Mathematics\\
University of Colorado\\
Boulder, CO 80309-0395\\
USA}
\email{Agnes.Szendrei@Colorado.EDU}

\thanks{This material is based upon work supported by
the Hungarian National Foundation for Scientific Research (OTKA)
grant no.\ K83219 and K104251.
}

\begin{abstract}
We prove that if $\m a$ is a finite algebra
with a parallelogram term that satisfies 
the split centralizer condition,
then $\m a$ is dualizable. This yields
yet another proof of the dualizability
of any finite algebra with a near unanimity term,
but more importantly proves that every finite 
module, group or ring in a residually small
variety is dualizable.
\end{abstract}

\maketitle

\section{Introduction}\label{intro}

Natural Duality Theory investigates categorical
dualities that are mediated by finite 
algebras.
One of the main goals of the theory is to
identify those finite algebras that can serve
as character algebras for a duality. Such algebras are
called \emph{dualizable}. Although general criteria
for dualizability 
are known, the problem of identifying the finite
algebras satisfying the criteria is still difficult.
The broadest natural class of algebras
where the problem has been solved is the class of
finite algebras $\m a$ such that the prevariety
$\Su\Pd(\m a)$ has tame congruence-theoretic
typeset 
contained in $\{\btyp, \ltyp\}$.
According to the Big NU Obstacle Theorem
from~\cite{davey-heindorf-mckenzie}
a finite algebra satisfying this assumption is 
dualizable if and only if it has a near
unanimity term operation.

This paper probes the broader class of 
finite algebras $\m a$ such that the prevariety
$\Su\Pd(\m a)$ has tame congruence-theoretic
typeset contained in $\{\atyp, \btyp, \ltyp\}$.
There is a natural analogue of a near unanimity term
operation that is likely to be involved in this context,
called a ``parallelogram term operation''. Recently 
M.~Moore \cite{moore} has announced an extension of one direction
of the Big NU Obstacle Theorem, namely,
under the assumption that $\m a$
is finite and 
$\Su\Pd(\m a)$ has tame congruence-theoretic
typeset contained in $\{\atyp, \btyp, \ltyp\}$,
it is the case that if $\m a$ is dualizable then it 
must have a parallelogram term operation. Known examples
show that an unmodified converse
to this statement cannot hold. Thus, we expect that 
there is some condition $\varepsilon$ such that,
if $\m a$
is finite and 
$\Su\Pd(\m a)$ has tame congruence-theoretic
typeset contained in $\{\atyp, \btyp, \ltyp\}$, then 
$\m a$ is dualizable if and only if $\m a$
has a parallelogram term operation and condition $\varepsilon$ holds.

In this paper we produce a candidate for condition
$\varepsilon$ under the additional assumption that
$\m a$ lies in a residually small variety.
We call the condition the ``split centralizer condition''.
Specifically, we operate under the global assumptions
that $\m a$
is a finite algebra from a residually small variety and 
$\Su\Pd(\m a)$ has tame congruence-theoretic
typeset contained in $\{\atyp, \btyp, \ltyp\}$.
We conjecture that, under the global assumptions, 
$\m a$ will be dualizable if and only if 
$\m a$ has a parallelogram term 
operation and satisfies the split centralizer condition.
What we prove in this paper is the ``if'' part of this conjecture.

Let us reformulate the statement that we prove.
It is a fact that a finite algebra with a parallelogram term 
automatically has the property that
$\Su\Pd(\m a)$ has tame congruence-theoretic
typeset contained in $\{\atyp, \btyp, \ltyp\}$.
Thus, our result is that 
if $\m a$ is finite, lies in a residually small variety,
has a parallelogram term, and satisfies the split centralizer
condition, then $\m a$ is dualizable.
Interestingly, our result is strong enough to
prove the dualizability of the dualizable algebras
in residually small varieties with a parallelogram term
in all the known cases,
and in many new cases, and yet neatly avoids
an obstacle to dualizability identified by Idziak in 1994.

Let us state the split centralizer condition.
Let $\m b$ be a finite algebra and let $\mathcal Q$
be a quasivariety containing $\m b$. A \emph{$\mathcal Q$-congruence}
on $\m b$ is a congruence $\kappa$ of $\m b$
such that $\m b/\kappa\in \mathcal Q$.
Let $\delta$ be a meet irreducible congruence on $\m b$
with upper cover $\theta$, and let $\nu = (\delta:\theta)$
be the centralizer of $\theta$ modulo $\delta$.
We will say that the triple of congruences $(\delta,\theta,\nu)$
which arises in this way
is \emph{split} (relative to $\mathcal Q$)
by a triple $(\alpha,\beta,\kappa)$ of congruences of $\m b$ if
\begin{enumerate}
\item[(i)] $\kappa$ is a $\mathcal Q$-congruence on $\m b$,
\item[(ii)] $\beta\leq \delta$,
\item[(iii)] $\alpha\wedge\beta = \kappa$,
\item[(iv)] $\alpha\vee\beta = \nu$, and
\item[(v)] 
$\alpha/\kappa$ is abelian (or equivalently, $[\alpha,\alpha]\leq\kappa$).
\end{enumerate}
Now, if 
$\theta/\delta$
is nonabelian, then 
$\nu=(\delta:\theta) = \delta$, so 
$(\delta,\theta,\nu)$ will be split by 
$(\alpha,\beta,\kappa):=(0,\delta,0)$. Therefore 
splitting is only in question when 
$\theta/\delta$
is abelian. Henceforth we shall focus only on this case
and call a triple $(\delta,\theta,\nu)$ 
of congruences on $\m b$ \emph{relevant}
if 
\begin{enumerate}
\item[(a)] $\delta$ is completely meet irreducible, 
\item[(b)] $\theta$ is the upper cover of $\delta$,
\item[(c)] $\nu = (\delta:\theta)$, and 
\item[(d)] 
$\theta/\delta$ is abelian.
\end{enumerate}

The 
\emph{split centralizer condition} 
for a finite algebra
$\m a$ is the condition that, for $\mathcal Q := \Su\Pd(\m a)$
and for any subalgebra $\m b\leq\m a$, each
relevant triple $(\delta,\theta,\nu)$ of $\m b$
is split (relative to $\mathcal Q$) by
some triple $(\alpha,\beta,\kappa)$.
The relationships between 
these congruences of $\al B$
is 
depicted in Figure~1.

\begin{figure}
\setlength{\unitlength}{1mm}
\begin{picture}(70,55)
\thicklines
\put(15,0){%
\bezier{200}(20,0)(0,0)(0,20)
\bezier{200}(0,20)(0,40)(20,40)
\bezier{200}(20,40)(40,40)(40,20)
\bezier{200}(40,20)(40,0)(20,0)

\put(20,0){\circle*{1.3}} 
\put(20,40){\circle*{1.3}} 
\put(19,-4){$0$} 
\put(19,41){$1$} 
\put(-13,35){$\Con(\al B)$}

\put(8,23){\circle*{1.3}}
\put(11,26){\circle*{1.3}}
\bezier{200}(8,23)(9.5,24.5)(11,26)

\bezier{200}(11,26)(16,40)(20,40)
\bezier{200}(11,26)(24,30)(24,34)
\bezier{200}(20,40)(24,40)(24,34)

\put(6,24){{\small $\delta$}}
\put(9,27){{\small $\theta$}}

\put(50,22){{\small $\beta\leq\delta$}}

\put(20,35){\circle*{1.3}}
\bezier{200}(11,26)(13,33)(20,35)
\bezier{200}(11,26)(18,28)(20,35)
\put(19,36){{\small $\nu$}}

\put(50,32){{\small $\nu=(\delta:\theta)$}}

\put(21,4){{\small $\kappa$}}
\put(20,5){\circle{1.5}}

\put(2,20){{\small $\beta$}}
\put(5,20){\circle*{1.3}}
\put(35,20){\circle*{1.3}}
\bezier{200}(5,20)(6.5,21.5)(8,23)
\put(5,20){\line(1,-1){14.5}}
\put(20,35){\line(1,-1){15}}
\put(20.5,5.5){\line(1,1){14.5}}
\put(36,20){{\small $\alpha$}}

\put(50,17){{\small $\alpha\wedge\beta=\kappa$}}
\put(50,12){{\small $\alpha\vee\beta=\nu$}}
\put(50,7){{\small $[\alpha,\alpha]\leq\kappa$}}
}
\end{picture}                             
\caption{}
\end{figure}

It is not hard to show that if a finite algebra 
with a parallelogram term satisfies
the split centralizer condition, then it generates a residually 
small variety (combine Corollary~\ref{cor-ssc-rs} and 
Theorem~\ref{thm-parterm-cm}).
Therefore our main result can be restated as follows.

\begin{thm}
\label{thm-main}
If $\al A$ is a finite algebra with a parallelogram term and $\al A$
satisfies the split centralizer condition, then $\al A$ is dualizable.
\end{thm}

Section~2 summarizes preliminaries.
Sections~3 through 5 are devoted to proving
Theorem~\ref{thm-main}.
In Section~6, we apply Theorem~\ref{thm-main} to prove that
if $\var{V}$ is a variety with a parallelogram term
in which every finite subdirectly irreducible algebra is either 
abelian or neutral or almost neutral, then every finite algebra in 
$\var{V}$ is dualizable (Theorem~\ref{maincor}).
This generalizes the following known results:
\begin{enumerate}
\item[(1)] every
finite algebra with a near unanimity term
is dualizable~\cite{davey-werner},
\item[(2)] every paraprimal algebra is dualizable~\cite{davey-quackenbush}, 
and 
\item[(3)]
every finite commutative ring $R$ with unity whose Jacobson radical 
squares to zero is dualizable~\cite{commutative_rings}.
\end{enumerate}
Theorem~\ref{maincor}
also implies the new results that
\begin{enumerate}
\item[(4)] every finite module is dualizable,
and 
\item[(5)] every finite ring (commutative or not, unital or not) 
in a residually small variety is dualizable.
\end{enumerate}
We also show how to apply Theorem~\ref{thm-main}
to deduce that 
every finite group with abelian Sylow subgroups is 
dualizable \cite{nickodemus}. 
This last theorem is not a consequence of 
Theorem~\ref{maincor}.

\section{Preliminaries}\label{prelim}

Algebras will be denoted by boldface letters, their universes by the 
same letters in italics. For arbitrary algebras $\al A$ and $\al B$,
$\Aut(\al A)$ denotes the automorphism group of $\al A$,
$\Con(\al A)$ the congruence lattice of $\al A$, and $\Hom(\al A,\al B)$
the set of all homomorphisms $\al A\to\al B$.
The top and bottom elements of $\Con(\al A)$ are denoted $1$ and $0$, 
respectively, and the identity map on any set $A$ is denoted $\id_A$.
The variety generated by an algebra $\al A$ is denoted by $\var{V}(\al A)$.
We will write $\al B\le\al A$ to indicate that $\al B$ is a subalgebra of
$\al A$. By a \emph{section} of $\al A$ we mean a quotient of a subalgebra 
of $\al A$.

Let $\vartheta$ be an equivalence relation on a set $A$.
The $\vartheta$-class of an element $a\in A$ is denoted by $a/\vartheta$,
and the number of 
equivalence classes of $\vartheta$ may be referred to as 
\emph{the index of $\vartheta$}. 
We will often write $a\equiv_{\vartheta} b$ instead of $(a,b)\in\vartheta$.
For $B\subseteq A$ the restriction of 
$\vartheta$ to
$B$ is denoted by $\vartheta_B$. If $\al B\le\al A$ and 
$\vartheta$ is a congruence
of $\al A$, we may write $\vartheta_{\al B}$ for $\vartheta_B$, which is
a congruence of $\al B$.

Let $\vartheta$ be a congruence of an algebra $\al A$, and let
$\al B$ be a subalgebra of $\al A$.
We say that $\al B$ is \emph{saturated with respect
to $\vartheta$}, or $\al B$ is a \emph{$\vartheta$-saturated subalgebra}
of $\al A$, if $b\in\al B$ and $b\equiv_{\vartheta}a$ imply 
$a\in\al B$ for all $a\in\al A$.
In other words, $\al B$ is $\vartheta$-saturated if and only if
its universe is a union of $\vartheta$-classes of $\al A$.
For arbitrary subalgebra $\al B$ of $\al A$ there exists a smallest
$\vartheta$-saturated subalgebra of $\al A$ that contains $\al B$, which we
denote by $\al B[\vartheta]$; the universe of $\al B[\vartheta]$
is $B[\vartheta]:=\bigcup_{b\in B}b/\vartheta$. By the second 
isomorphism theorem the map 
$\al B/\vartheta_{\al B}\to\al B[\vartheta]/\vartheta_{\al B[\vartheta]}$,
$b/\vartheta_{\al B}\mapsto b/\vartheta_{\al B[\vartheta]}(=b/\vartheta)$
is an isomorphism, so the indices of $\vartheta_{\al B}$ and
$\vartheta_{\al B[\vartheta]}$ are equal.

For every natural number $m$ we will use the notation $[m]$ for the set
$\{1,2,\dots,m\}$.

\subsection{The Modular Commutator and Residual Smallness}
For the definition and basic properties of the commutator operation 
$[\phantom{n},\phantom{n}]$ on congruence lattices of algebras in
congruence modular varieties the reader is referred to \cite{freese-mckenzie}.
To avoid confusion, intervals in congruence lattices will be denoted by 
$\interval{\phantom{n}}{\phantom{n}}$.
Recall that a congruence $\alpha\in\Con(\al A)$ of an algebra $\al A$ is called
\emph{abelian} if $[\alpha,\alpha]=0$, and an interval
$\interval{\beta}{\alpha}$ in $\Con(\al A)$
is called abelian if $[\alpha,\alpha]\le\beta$,
or equivalently, if the congruence $\alpha/\beta\in\Con(\al A/\beta)$
is abelian.
For two congruences $\beta\le\alpha$ in $\Con(\al A)$ \emph{the centralizer of
$\alpha$ modulo $\beta$}, denoted $(\beta:\alpha)$, is the largest congruence
$\gamma\in\Con(\al A)$ such that $[\alpha,\gamma]\le\beta$.
The \emph{center} of $\al A$ is the congruence $\zeta:=(0:1)$.

For a cardinal $c$, 
a variety $\var{V}$ is called \emph{residually less than $c$} 
if every subdirectly irreducible algebra in $\var{V}$ has cardinality
$<c$; $\var{V}$ is called \emph{residually small} if it is residually 
less than some cardinal.
It is proved in \cite[Theorem~10.14]{freese-mckenzie} that for a congruence
modular variety $\var{V}$ to be residually small, it is necessary that 
the congruence lattice of every algebra $\al A\in\var{V}$ satisfy
the commutator identity
\begin{align}
[x\wedge y,y]&= x\wedge[y,y],\tag{C1}
\end{align}
which can also be expressed by the implication $x\le[y,y]\to x=[x,y]$ 
(see \cite[Theorem~8.1]{freese-mckenzie}).
For finitely generated varieties the converse is also true, as the following
theorem states.

\begin{thm}[{{}From \cite[Theorem~10.15]{freese-mckenzie}}]
\label{thm-rs}
Let $\al A$ be a finite algebra that generates a congruence modular variety
$\var{V}(\al A)$. Then the following conditions are equivalent:
\begin{enumerate}
\item[{\rm(a)}]
$\var{V}(\al A)$ is residually small,
\item[{\rm(b)}]
$\var{V}(\al A)$ is residually ${}<q$ for some natural number $q$,
\item[{\rm(c)}]
{\rm(C1)} holds in the congruence lattice of every subalgebra of $\al A$.
\end{enumerate}
\end{thm}

If $(\delta,\theta,\nu)$ is a relevant triple of an algebra $\al B$
such that (C1) holds in $\Con(\al B)$, 
then $[\theta\wedge\nu,\nu]=[\theta,\nu]\le\delta$ implies that 
$\theta\wedge[\nu,\nu]\le\delta$, and
since $\theta>\delta$ and $\delta$ is 
completely meet irreducible, it must be the case that
$[\nu,\nu]\le\delta$.
Conversely, it is not hard to show (see the first paragraph of the proof of
\cite[Theorem~10.14]{freese-mckenzie}) that if (C1) fails in $\Con(\al B)$,
then  there is a failure (with $x=\alpha$, $y=\beta$) 
of the following form:
\[
[\alpha,\beta]\le\delta<\theta\le\alpha\le[\beta,\beta]\,(\le\beta)
\]
where $\delta$ is completely meet irreducible and $\theta$ is its upper cover.
It follows that 
$[\theta,\theta]\le[\alpha,\beta]\le\delta$, so $\theta/\delta$ is abelian.
Moreover, $\nu:=(\delta:\theta)\ge\beta$ since 
$[\theta,\beta]\le[\alpha,\beta]\le\delta$, which implies that
$[\nu,\nu]\ge[\beta,\beta]\ge\theta$. 
Consequently, $(\delta,\theta,\nu)$ is a relevant triple such that 
$[\nu,\nu]\not\le\delta$.
This shows that conditions
(a)--(c) in Theorem~\ref{thm-rs} are also equivalent to the condition
\begin{enumerate}
\item[(d)]
$[\nu,\nu]\le\delta$
for every relevant triple $(\delta,\theta,\nu)$ of a subalgebra of $\al A$.
\end{enumerate} 

\begin{cor}
\label{cor-ssc-rs}
Let $\al A$ be a finite algebra that generates a congruence modular 
variety $\var{V}(\al A)$.
If $\al A$ satisfies the split centralizer condition, 
then {\rm(C1)} holds in 
the congruence lattice of every subalgebra of $\al A$, so
$\var{V}(\al A)$ is residually small.
\end{cor}

\begin{proof}
If $\al A$ satisfies the split centralizer condition, and 
$(\delta,\theta,\nu)$ is a relevant triple of a subalgebra $\al B$ of $\al A$
which is split by the triple $(\alpha,\beta,\kappa)$, then
$\interval{\delta}{\nu}\subseteq\interval{\beta}{\nu}
\searrow\interval{\kappa}{\alpha}$ with the last interval abelian, 
therefore the other two intervals are also abelian.
This proves that condition (d) holds for $\al A$.
Our statement now follows from the equivalence of (d) to conditions
(c) and (b) in Theorem~\ref{thm-rs}.
\end{proof}

\subsection{Compatible Relations, Entailment, and Dualizability}

Let $A$ be a set. By a \emph{relation} on $A$ we mean a subset $\rho$ of $A^n$
for some positive integer $n$, which we call the \emph{arity} of $\rho$.
For any nonempty subset $I$ of $[n]$, 
\emph{projection onto the coordinates in $I$} is the map
$\proj_I\colon A^n\to A^I$, $(a_i)_{i\in[n]}\mapsto(a_i)_{i\in I}$. 

For any algebra $\al A$, an $n$-ary relation $\rho$ on $A$ is called
a \emph{compatible relation of $\al A$} if $\rho$ is the universe 
of a subalgebra of $\al A^n$
(or equivalently, $\rho$ is a nonempty subuniverse of $\al A^n$). 
The set of all compatible relations of
$\al A$ of arity $\le n$ will be denoted by $\rel{R}_n(\al A)$, 
and $\rel{R}(\al A)$
will stand for the set $\bigcup_{n>0}\rel{R}_n(\al A)$
of all compatible relations of $\al A$.

It is straightforward to check that if $R$ is a set of compatible relations
of an algebra $\al A$, then so is every relation that can be obtained,
in finitely many steps, from
relations in $R\cup\{{=}\}$ by the following constructs:
\begin{itemize}
\item
nonempty 
intersection of relations of the same arity,
\item
direct product of two relations,
\item
permutation of coordinates of a relation, and
\item
projection of a relation onto a nonempty subset of its coordinates.
\end{itemize}
The relations that can be obtained in this way are exactly the 
nonempty 
relations that are definable by primitive positive formulas (with $=$) 
using the relations in $R$.
The relations that can be obtained, in finitely many steps, from
relations in $R\cup\{{=}\}$ by the first three types of constructs
are exactly the 
nonempty
relations that are definable by 
quantifier free primitive positive formulas (with $=$) 
using the relations in $R$.

A \emph{critical relation} of an algebra $\al A$ is a compatible relation
$\rho$ of $\al A$ that is 
\begin{enumerate}
\item[(i)] 
completely $\cap$-irreducible in 
the 
lattice of subuniverses
of $\al A^n$ where $n$ is the arity of $\rho$, and
\item[(ii)]
directly indecomposable as a relation, that is, $[n]$ cannot be partitioned 
into two nonempty sets $I$ and $J$ such that $\rho$ and 
$\proj_I(\rho)\times\proj_J(\rho)$ differ only by a permutation of 
coordinates.
\end{enumerate}
More informally, a compatible relation of $\al A$ is critical if it cannot be
obtained in a nontrivial way from other compatible relations 
using only the first three of the four types of constructs above.

In the theory of natural dualities there is an entailment concept, 
which we will denote by $\dentails$ ($d$ stands for 
`$\underline{\textrm{d}}$uality').
We refer the reader interested in the definition to Definition~4.1
in \cite[Chapter~2]{clark-davey}.
For the purposes of this paper an algebraic characterization of $\dentails$,
which we state in Theorem~\ref{thm-dentails} below, will be sufficient.
We need some definitions beforehand.

Let $B$ be a compatible relation of an algebra $\al A$, say $B$ is $n$-ary,
and let $B':=\proj_I(B)$ for some nonempty $I\subseteq[n]$.
Then the projection map $\proj_I\colon\al B\to\al B'$ is a surjective
homomorphism
between the algebras $\al B\,(\le\al A^n)$ and $
\al B'\,(\le\al A^I)$ with universes $B$ and $B'$.
Following \cite{clark-davey},
we call $\proj_I\colon\al B\to\al B'$ a \emph{retractive projection}
if it is a retraction, that is, if there exists a homomorphism
$\phi\colon\al B'\to\al B$ such that $\proj_I\circ\phi=\id_{B'}$. 
An important special case is a \emph{bijective projection}
$\proj_I\colon\al B\to\al B'$, when the retractive projection
is bijective.
Accordingly, we say that $B'$ is obtained from $B$ by
retractive projection (or bijective projection) onto its coordinates in $I$.

From now on we will restrict to finite algebras.
The following theorem is a consequence of Theorems~2.2 and 2.6
in \cite[Chapter~9]{clark-davey}.

\begin{thm}[{}From \cite{clark-davey}]
\label{thm-dentails}
Let $\al A$ be a finite algebra, and let $R\subseteq\rel{R}(\al A)$.
For a 
nonempty
relation $\rho$ on $A$, we have
$R\dentails\rho$ if and only if
$\rho$ can be obtained, in finitely many steps, from
relations in $R\cup\{{=}\}$ by the following constructs:
\begin{itemize}
\item
nonempty
intersection of relations of the same arity,
\item
product of two relations,
\item
permutation of coordinates of a relation, and
\item
retractive projection of a relation onto a nonempty subset of its coordinates.
\end{itemize}
\end{thm}

The four types of constructs in this theorem will be referred to as 
\emph{$\dentails$-constructs}. 
Notice that the only difference between the list
of $\dentails$-constructs and the earlier list of 
constructs for compatible relations is that among the $\dentails$-constructs
projections are restricted to retractive projections.
Two immediate consequences are worth mentioning: 
(i)~Since every $\dentails$-construct occurs on the earlier list, it follows
that if $R\subseteq\rel{R}(\al A)$ and
$R\dentails\rho$, then $\rho\in\rel{R}(\al A)$.
(ii)~Since the first three constructs on the two lists are the same,
our earlier remark implies that if $R\subseteq\rel{R}(\al A)$
and $\rho$ 
is a nonempty relation on $A$ that 
is definable by a quantifier-free primitive positive formula
(with ${=}$) using relations in $R$, then $R\dentails\rho$.

Observation (i) above allows us to extend $\dentails$ to subsets of
$\rel{R}(\al A)$ as follows: for $R,R'\subseteq\rel{R}(\al A)$
let $R\dentails R'$ mean that $R\dentails\rho$ for every $\rho\in R'$.
It is easy to check that the relation $\dentails$ is transitive on subsets
of $\rel{R}(\al A)$, that is, if $R,R',R''\subseteq\rel{R}(\al A)$
satisfy $R\dentails R'$ and $R'\dentails R''$, then $R\dentails R''$.

The most powerful general criterion for dualizability is the following theorem
of Willard and Z\'adori.

\begin{thm}[See \cite{willard},\cite{zadori}]
\label{thm-WZ}
Let $\al A$ be a finite algebra.
If $\rel{R}_n(\al A)\dentails\rel{R}(\al A)$
for some integer $n\ge1$, then $\al A$ is dualizable.
\end{thm}

In fact, it is shown in \cite{zadori} that
if $\rel{R}_n(\al A)\dentails\rel{R}(\al A)$
holds for some $n\ge1$, then 
it also holds with the restriction that
retractive projections among the $\dentails$-constructs are limited to
bijective projections.

\subsection{Algebras with Parallelogram Terms}

Let $m$ and $n$ be positive integers
and let $k= m+n$. The concept of an
{\em $(m,n)$-parallelogram term} 
(or {\em $k$-parallelogram term})
for a variety $\var{V}$ was introduced in
\cite{parallelogram} to mean 
a term 
${\sf P}$
such that the identities represented by the rows of the following 
matrix equation hold in $\var{V}$:
\begin{equation}\label{p}
{\sf P}
\left(
\begin{array}{ccc|}
x&x&y\\
x&x&y\\
&\vdots&\\
x&x&y\\
\hline
y&x&x\\
&\vdots&\\
y&x&x\\
y&x&x\\
\end{array}
\begin{array}{cccccccc}
z&y&\cdots&y&y&\cdots&y&y\\
y&z&&y&y&&y&y\\
\vdots&&\ddots& &&&&\vdots\\
y&y&&z&y&&y&y\\
y&y&&y&z&&y&y\\
\vdots&&&&&\ddots&&\vdots\\
y&y&&y&y&&z&y\\
y&y&\cdots&y&y&\cdots&y&z\\
\end{array}
\right) = 
\left(\begin{matrix}
y\\
y\\
\vdots\\
y\\
y\\
\vdots\\
y\\
y
\end{matrix}
\right).
\end{equation}
Here $\sf P$ is $(k+3)$-ary, the rightmost block
of variables is a $k\times k$ array, the upper left
block is $m\times 3$ and the lower left block is $n\times 3$.
An $(m,n)$-parallelogram term (or $k$-parallelogram term) for an algebra
$\al A$ is defined to be 
an $(m,n)$-parallelogram term (or $k$-parallelogram term)
for the variety $\var{V}(\al A)$ it generates.

It is easy to see from these definitions that a $k$-parallelogram term
that is independent of its last $k$ variables is a Maltsev term, and
a $k$-parallelogram term that is independent of its first $3$ variables 
is a $k$-ary near unanimity term.

It was proved in \cite[Theorem~3.5]{parallelogram} that
if $m,n,m',n'$ are positive integers such that $m+n=k=m'+n'$, then a variety
has an $(m,n)$-parallelogram term if and only if it has an 
$(m',n')$-parallelogram term; this justifies referring to them as 
$k$-parallelogram terms. In addition, \cite[Theorem~3.5]{parallelogram}
also shows that a variety has a $k$-parallelogram term if and only if it has
a term, called a $k$-cube term, introduced in \cite[Definition~2.4]{bimmvw}.
It follows from \cite[Theorem~2.7]{bimmvw} 
that every variety with
a $k$-cube term is congruence modular. 
So, by combining all these results, or by a direct argument using
\cite[Theorem~3.2]{dent-kearnes-szendrei} we get the following conclusion.

\begin{thm}[See \cite{bimmvw},\cite{parallelogram};\cite{dent-kearnes-szendrei}]
\label{thm-parterm-cm}
Every variety with a parallelogram term is congruence modular.
\end{thm}

Our starting point for the proof of Theorem~\ref{thm-main}
will be the structure theorem in \cite{parallelogram}
(see \cite[Theorems~4.1 and 2.5]{parallelogram})
on the critical relations
of a finite algebra with a parallelogram term. 
If $C$ is an $n$-ary critical relation of $\al A$, let
$\al C$ denote the subalgebra of $\al A^n$ with universe
$C$, and let $\al A_i:=\proj_i(\al C)$ for each $i\in[n]$. 
Furthermore, let $\delta_i\in\Con(\al A_i)$ ($i\in[n]$)
be such that 
$\delta:=\prod_{i=1}^n\delta_i$
is the largest product congruence of 
$\prod_{i=1}^n\al A_i$ with the property that $\al C$ is a $\delta$-saturated
subalgebra of $\prod_{i=1}^n\al A_i$.
Then $\bar{\al C}:=\al C/\delta_{\al C}$ is a subdirect product
of the algebras $\bar{\al A}_i:=\al A_i/\delta_i$ $(i\in[n])$, which we call
\emph{the reduced representation of $\al C$}.

The next theorem contains
those parts of the structure theorem on critical relations
that we will need later on; we
retain the numbering from \cite[Theorem~2.5]{parallelogram}.

\begin{thm}[From \cite{parallelogram}]
\label{thm-paralg}
Let $C$ be an $n$-ary critical relation of a finite algebra
$\al A$ with a $k$-parallelogram term,
and let $\bar{\al C}$ be its reduced representation. 
If $n\ge k\,(>1)$, then the following hold.
\begin{enumerate}
\item[{\rm(1)}]
$\bar{\al C}\le\prod_{i=1}^n\bar{\al A}_i$ is a representation of
$\bar{\al C}$ as a subdirect product of subdirectly irreducible
algebras $\bar{\al A}_i$.
\item[{\rm(7)}]
If $n>2$, then each $\bar{\al A}_i$ has abelian monolith $\mu_i$
$(i\in[n])$.
\item[{\rm(8)}]
For the centralizers $\rho_\ell:=(0:\mu_\ell)$ of the monoliths $\mu_\ell$
$(\ell\in[n])$, the image of the composite map
\[
\bar{\al C}\hookrightarrow\prod_{\ell=1}^n\bar{\al A}_\ell
\twoheadrightarrow\bar{\al A}_i/\rho_i\times\bar{\al A}_j/\rho_j
\]
is the graph of an isomorphism 
for any $i,j\in[n]$.
\end{enumerate}
\end{thm}

Since $\bar{\al A}_i=\al A_i/\delta_i$ $(i\in[n])$ is subdirectly irreducible,
$\delta_i\in\Con(\al A_i)$ is completely meet irreducible,
and the monolith of $\bar{\al A}_i$ is
$\mu_i=\theta_i/\delta_i$ where $\theta_i\in\Con(\al A_i)$
is the unique cover of $\delta_i$.
Furthermore, $(0:\mu_i)=\rho_i=\nu_i/\delta_i$ where
$\nu_i=(\delta_i:\theta_i)$, and hence $\bar{\al A}_i/\rho_i\cong\al A_i/\nu_i$.
Thus, with our current terminology of a relevant triple, 
Theorem~\ref{thm-paralg} (restricted to the case $n>2$)
can be restated as follows.

\begin{cor}
\label{cor-paralg}
Let $C$ be an $n$-ary critical relation of a finite algebra
$\al A$ with a $k$-parallelogram term,
and let $\al C$ be the subalgebra of $\al A^n$ with universe $C$.
If $n\ge\max(3,k)$, then the following hold.
\begin{itemize}
\item
$\al C$ is a subdirect product of the subalgebras
$\al A_i:=\proj_i(\al C)$ $(i\in[n])$ of $\al A$.
\item
If $\delta=\prod_{i=1}^n\delta_i$ $\bigl(\delta_i\in\Con(\al A_i)\bigr)$ 
is the largest product congruence 
of $\prod_{i=1}^n\al A_i$ for which $\al C$ is $\delta$-saturated, 
then each $\delta_i$ $(i\in[n])$ is the first component
of a relevant triple $(\delta_i,\theta_i,\nu_i)$ of $\al A_i$.
\item
The assignment
\[
\iota_{ij}\colon\al A_i/\nu_i\to\al A_j/\nu_j,\ \ 
c_i/\nu_i\mapsto c_j/\nu_j\ \ 
\text{whenever}\ \ (c_1,\dots,c_n)\in C
\]
is an isomorphism for any $i,j\in[n]$.
\end{itemize} 
\end{cor}

\section{Reduction to Abelian Congruences of Bounded Index}\label{reduc}

Our main tool for proving Theorem~\ref{thm-main} will be 
Theorem~\ref{thm-WZ}. 
If $\al A$ is a finite algebra satisfying the assumptions of 
Theorem~\ref{thm-main}, then Corollary~\ref{cor-ssc-rs} and 
Theorem~\ref{thm-parterm-cm} imply that (C1) holds in the congruence lattice
of every subalgebra of $\al A$.
Hence $\nu'/\delta'$ is abelian for every relevant triple
$(\delta',\theta',\nu')$ of a subalgebra $\al A'$ of $\al A$.
Therefore,
if $C$ is an $n$-ary critical relation of $\al A$
such that, with the notation of Corollary~\ref{cor-paralg},
$\delta_i=0$ for every $i\in[n]$, then $\nu_i$ is abelian for every
$i\in[n]$, and hence the product congruence $\nu:=\prod_{i=1}^n\nu_i$
of $\prod_{i=1}^n\al A_i$ restricts to $\al C$ as an abelian congruence 
$\nu_{\al C}$.
Moreover, by the last item in Corollary~\ref{cor-paralg},
$\nu_{\al C}$ has index $<|A|$.

In this situation we can use modules associated to abelian congruences, 
as explained in Sections~\ref{modules}--\ref{mainthm-sec}, to show that
$C$ is entailed by compatible relations of bounded arity.
However, this argument does not work if, instead of an abelian interval
$\interval{0}{\nu_{\al C}}$, we have an abelian interval 
$\interval{\delta_{\al C}}{\nu_{\al C}}$ with $\delta_{\al C}\not=0$.
In the general case when
$\delta_{\al C}$ may be a nontrivial congruence
of $\al C$,
we will use the assumption that $\al A$ satisfies 
the split centralizer condition to replace $C$ by another compatible relation
$B$, which is not weaker than $C$ with respect to entailment, but 
has an abelian interval 
$\interval{0}{\tilde{\alpha}_{\al C}}$
corresponding to $\interval{\delta_{\al C}}{\nu_{\al C}}$
at the bottom of the congruence lattice, moreover, 
the index of $\tilde{\alpha}_{\al C}$
remains bounded by a number independent of $n$
(though might be much bigger than the index of $\nu_{\al C}$, 
which is $<|A|$).

The purpose of this section is to construct such a relation $B$ for every
critical relation $C$ of $\al A$. 

Given a finite algebra $\al A$, we define several constants associated to 
$\al A$ as follows.
\begin{itemize}
\item
Let $\aaa$ be the maximum of all 
$|\Aut(\al S/\nu)|$ where $\al S\le\al A$ and $(\delta,\theta,\nu)$
is a relevant triple of $\al S$.
($\aaa$ stands for `$\underline{\text{a}}$utomorphisms'.)
\item
Let $\sss$ be the number of distinct pairs $(\al S,\delta)$
such that $\al S\le\al A$ and $(\delta,\theta,\nu)$
is a relevant triple of $\al S$.
($\sss$ stands for `$\underline{\text{s}}$ubdirectly irreducible
$\underline{\text{s}}$ections $\al S/\delta$'.)
\item
Let $\ii:=|A|^{\aaa\sss}$.
($\ii$ stands for `$\underline{\text{i}}$ndex'; 
see Theorem~\ref{thm-reduction} below.)
\item
Let $\pp$ be the least positive integer with the property that
for every subalgebra $\al S\le\al A$ every relevant triple of 
$\al S$ is split by a triple $(\alpha,\beta,\kappa)$ such that 
$\al S/\kappa$ embeds into $\al A^p$ for some $p\le\pp$.
($\pp$ stands for `$\underline{\text{p}}$ower'.)
\end{itemize}

\begin{thm}
\label{thm-reduction}
Let $\al A$ be a finite algebra with a $k$-parallelogram term
such that $\al A$ satisfies the split centralizer condition.
For every $n\ge \max(3,k)$ and for every $n$-ary critical relation
$C$ of $\al A$, there exists a compatible relation $B$ of $\al A$
which has the following two properties:
\begin{enumerate}
\item[$(*)$]
There exist 
\begin{enumerate}
\item[{\rm(I)}]
subalgebras $\al B_i\le\al A^{p_i}$ with $p_i\le\pp$
for each $i\in[n]$
such that $\al B_i$ is isomorphic to a section of $\al A$, and  
\item[{\rm(II)}]
nontrivial abelian congruences 
$\tilde\alpha_i\in\Con(\al B_i)$ $(i\in[n])$
\end{enumerate}
such that
\begin{enumerate}
\item[{\rm(III)}]
$B$ is the universe of
a subdirect product $\al B$ of 
$\al B_1,\dots,\al B_n$, and
\item[{\rm(IV)}]
the product congruence
$\tilde\alpha:=\prod_{i=1}^n \tilde\alpha_i$ of $\prod_{i=1}^n\al B_i$ 
restricts to $\al B$ as a congruence 
$\tilde\alpha_{\al B}$ 
of index $\le\ii$.
\end{enumerate}
\item[$(**)$]
$\rel{R}_{1+\pp}(\al A)\cup\{B\}\dentails C$.
\end{enumerate}
\end{thm}

\begin{proof}
Corollary~\ref{cor-paralg} describes the structure of $C$. 
Using the notation of Corollary~\ref{cor-paralg}, we first
define two equivalence relations, $\approx$ and $\sim$, 
on $[n]$
as follows: 
\begin{itemize}
\item
$i\approx j$ iff $\al A_i=\al A_j$ and 
$(\delta_i,\theta_i,\nu_i)=(\delta_j,\theta_j,\nu_j)$;
\item
$i\sim j$ iff $i\approx j$ and  
$\iota_{ij}$ is the identity isomorphism.
\end{itemize}
To give an upper bound for the indices of $\approx$ and $\sim$,
notice that if $i,j\in[n]$ are such that 
$\al A_i=\al A_j$ and $\delta_i=\delta_j$, then it follows that
$(\delta_i,\theta_i,\nu_i)=(\delta_j,\theta_j,\nu_j)$, so $i\approx j$. 
This implies that $\approx$ has at most $\sss$ classes, that is,
$\approx$ has index $|[n]/{\approx}|\le\sss$.
The definition of $\sim$ shows that
every class of ${\approx}/{\sim}$ has size $\le\aaa$.
Hence $\sim$ has at most $\aaa\sss$ classes, that is,
$\sim$ has index $|[n]/{\sim}|\le\aaa\sss$.

Our assumption that $\al A$ satisfies the split centralizer condition
ensures that for each $i\in[n]$ the relevant triple
$(\delta_i,\theta_i,\nu_i)$ of $\al A_i\ (\le\al A)$ is split
by a triple $(\alpha_i,\beta_i,\kappa_i)$.
These choices could differ coordinate by coordinate, but we can 
choose a transversal $T_\approx$ for the classes of $\approx$ and 
redefine the triples in coordinates $i\notin T_\approx$ to arrange that
$(\alpha_i,\beta_i,\kappa_i)=(\alpha_j,\beta_j,\kappa_j)$
whenever $i\approx j$ ($i,j\in[n]$).

Let $\alpha$, $\beta$, $\delta$, and $\kappa$ 
denote the product congruences $\prod_{i=1}^n \alpha_i$, 
$\prod_{i=1}^n \beta_i$, $\prod_{i=1}^n \delta_i$, 
and $\prod_{i=1}^n \kappa_i$ of 
$\prod_{i=1}^n\al A_i$. We have 
$\kappa\le\alpha$,
$\kappa\le\beta\le\delta$, and
$\kappa=\alpha\wedge\beta$,
because 
$\kappa_i\le\alpha_i$,
$\kappa_i\le\beta_i\le\delta_i$, and
$\kappa_i=\alpha_i\wedge\beta_i$
for all $i$.
By Corollary~\ref{cor-paralg},
$\al C$ is $\delta$-saturated, and hence also 
$\beta$- and $\kappa$-saturated.

Now let $D$ denote the set
of all tuples $(d_1,\dots,d_n)\in C$ 
such that 
\begin{enumerate}
\item[$(\dagger)$]
$d_i\equiv_{\alpha_j} d_j$ whenever $i\sim j$. 
\end{enumerate}
Note that $i\sim j$ implies that $\alpha_i=\alpha_j$,
so $\alpha_i$ and $\alpha_j$ are interchangeable in $(\dagger)$.

\begin{clm}
\label{clm-c-d}
Let $T$ be a transversal for the classes of $\sim$.
For every $(c_1,\dots,c_n)\in C$ 
\begin{enumerate}
\item[{\rm(1)}]
there exists 
$(d_1,\dots,d_n)\in D$ such that 
\begin{enumerate}
\item[$(\ddagger)$]
$(d_1,\dots,d_n)\equiv_\beta(c_1,\dots,c_n)$ and $d_t=c_t$ for all $t\in T$;
\end{enumerate}
\item[{\rm(2)}]
the tuples $(d_1,\dots,d_n)\in D$ satisfying $(\ddagger)$
are uniquely determined modulo $\kappa$.
\end{enumerate}
\end{clm}

\noindent
{\it Proof of Claim~\ref{clm-c-d}.}
Choose $(c_1,\dots,c_n)\in C$.

For (1), 
define $d_t:=c_t$ for $t\in T$. 
Our aim is to show that 
\begin{multline}
\label{eq-dtuple}
\text{\qquad for each $t\in T$ and for all $i\sim t$ with $i\not=t$}\\
\text{there exist $d_i\in A_i$ such that 
$c_i\equiv_{\beta_i} d_i\equiv_{\alpha_t} d_t$.\qquad}
\end{multline}
This will complete the proof of (1) for the following reason.
Let $(d_1,\dots,d_n)$ be a tuple obtained in this way.
By its construction, $(d_1,\dots,d_n)$ will satisfy condition $(\ddagger)$. 
It will also satisfy condition $(\dagger)$, because
whenever $i\sim j$, we have a $t\in T$ with $i,j\sim t$,
so $d_i,d_j\equiv_{\alpha_t} d_t$; since $\alpha_j=\alpha_t$, we get that
$d_i\equiv_{\alpha_j} d_j$.
Now, since $C$ is $\beta$-saturated, the fact that $(\ddagger)$ 
holds for $(d_1,\dots,d_n)$
ensures that $(d_1,\dots,d_n)\in C$. 
Hence 
the fact that $(\dagger)$ also holds for $(d_1,\dots,d_n)$
implies that $(d_1,\dots,d_n)\in D$.

To prove \eqref{eq-dtuple}, 
choose $t\in T$ and $i\sim t$ with $i\not=t$.
Then $i\sim t$ implies that $\al A_i=\al A_t$, $\beta_i=\beta_t$,
$\nu_i=\nu_t$, and $c_i/\nu_i=c_t/\nu_t$. Hence,
$c_i\equiv_{\nu_t} c_t$. 
By our assumptions, $(\alpha_t,\beta_t,\kappa_t)$
splits $(\delta_t,\theta_t,\nu_t)$, so 
$\alpha_t\vee\beta_t=\nu_t$ and 
$\alpha_t/\kappa_t$ is an abelian congruence of 
$\al A_t/\kappa_t$.
Since $\al A_t/\kappa_t$ lies in a congruence modular variety,
we have that $\alpha_t/\kappa_t$ permutes with all congruences of 
$\al A_t/\kappa_t$ (see \cite[Theorem~6.2]{freese-mckenzie}). 
Hence, $\nu_t=\alpha_t\vee\beta_t=\beta_t\circ\alpha_t$. 
Thus, the fact that
$c_i\equiv_{\nu_t} c_t$ holds implies that
there exists $d_i\in A_t=A_i$ such that 
$c_i\equiv_{\beta_t} d_i\equiv_{\alpha_t} c_t=d_t$. Since $\beta_i=\beta_t$,
this $d_i$ is the element we wanted to find.

For (2),
assume that $(d_1,\dots,d_n),\,(d_1',\dots,d_n')\in D$ 
both satisfy condition $(\ddagger)$.
We want to show that 
$(d_1,\dots,d_n)\equiv_{\kappa}(d_1',\dots,d_n')$.
Since $\kappa=\alpha\wedge\beta$, this is equivalent to showing that
$(d_1,\dots,d_n)\equiv_{\alpha}(d_1',\dots,d_n')$ and
$(d_1,\dots,d_n)\equiv_{\beta}(d_1',\dots,d_n')$. 
The latter follows from the assumption that 
both tuples satisfy condition $(\ddagger)$.
To prove the former, consider any $i\in[n]$, and let
$t\in T$ be such that $i\sim t$.
Combining the second part of condition $(\ddagger)$ with
condition $(\dagger)$ in the definition of $D$, we get that
$c_t=d_t\equiv_{\alpha_i} d_i$, and similarly, 
$c_t=d_t'\equiv_{\alpha_i} d_i'$, hence
$d_i\equiv_{\alpha_i} d_i'$.
Thus, $(d_1,\dots,d_n)\equiv_{\alpha}(d_1',\dots,d_n')$, completing the
proof of (2).
\hfill$\diamond$

\begin{clm}
\label{clm-saturation}
$D$ is the universe of a subalgebra $\al D$ of
$\al C$ with the following properties:
\begin{enumerate}
\item[{\rm(1)}]
$\al D$ is a subdirect product of $\al A_1,\dots,\al A_n$,
\item[{\rm(2)}]
$\proj_{T}(\al D)=\proj_{T}(\al C)$ for every transversal $T$ for 
the classes of $\sim$,
\item[{\rm(3)}] 
$\al D[\beta]=\al C$, and
\item[{\rm(4)}]
$\al D[\kappa]=\al D$.
\end{enumerate}
\end{clm}

\noindent
{\it Proof of Claim~\ref{clm-saturation}.}
Consider a
transversal $T$ for the classes of $\sim$.
Then $D\subseteq C$ implies that 
$\proj_{T}(D)\subseteq\proj_{T}(C)$, and
Claim~\ref{clm-c-d}~(1) shows that 
$\proj_{T}(D)\supseteq\proj_{T}(C)$.
Thus, $\proj_{T}(D)=\proj_{T}(C)$.
This implies that
$\proj_t(D)=\proj_t(C)=A_t$ for every $t\in T$.
Since every element of $[n]$ occurs in some transversal $T$,
we get that $\proj_i(D)=A_i$ for all $i\in[n]$.
In particular, we see that $D\not=\emptyset$.
The definition of $D$ shows that $D\subseteq C$ and $D$ 
is definable by a primitive positive 
formula, using the compatible relations 
$\alpha_1,\dots,\alpha_n$ and $C$ of $\al A$.
Hence, $D$ is a compatible relation of $\al A$.
These arguments prove that  $D$ is the universe of a subalgebra 
$\al D$ of $\al C$ satisfying (1)--(2).

For (3), observe that
since $\al D\le\al C$ and $\al C$ is $\beta$-saturated, it follows that
$\al D[\beta]\le \al C[\beta]$ and $\al C[\beta]=\al C$. 
Thus, $\al D[\beta]\le\al C$.
However, we get from Claim~\ref{clm-c-d}~(1) that 
$C\subseteq D[\beta]$. 
Hence, $\al D[\beta]=\al C$, as claimed.

Finally, we prove (4).
Clearly, $\al D\le\al D[\kappa]$.
The reverse inclusion $D[\kappa]\subseteq D$ can be established as follows:
\[
D[\kappa]
\subseteq D[\alpha]\cap D[\beta]
= D[\alpha]\cap C
\subseteq D,
\]
where the last inclusion
is an immediate consequence of the definition of $D$. 
This completes the proof of 
Claim~\ref{clm-saturation}.
\phantom{m}
\hfill$\diamond$

\begin{clm}
\label{clm-index-of-alpha}
For every transversal $T$
for the classes of $\sim$,
\[
|\al D/\alpha_{\al D}|\le|\proj_{T}(\al D)|\le\ii.
\]
\end{clm}

\noindent
{\it Proof of Claim~\ref{clm-index-of-alpha}.}
Let $T$ be a transversal for the classes of $\sim$.
From the inequality $|[n]/{\sim}|\le\aaa\sss$ proved earlier
we get that $|T|\le\aaa\sss$.
Thus,
$|\proj_{T}(\al D)|
\le|A|^{|T|}
\le|A|^{\aaa\sss}=\ii$, which establishes the second 
inequality.

It follows from the definition of $D$ that
the kernel of the projection map 
$\proj_T\colon\al D\to\proj_T(\al D)$
is contained in $\alpha$. This implies the 
first inequality in Claim~\ref{clm-index-of-alpha}.
\hfill$\diamond$

\medskip

Now we are ready to define the algebras $\al B_1,\dots,\al B_n$, $\al B$,
and the congruences $\tilde{\alpha}_i\in\Con(\al B_i)$ for which
the conclusions $(*)$--$(**)$ of Theorem~\ref{thm-reduction} hold.

Since each $\kappa_i$ is a $\var{Q}$-congruence of $\al A_i$
for $\var{Q}=\Su\Pd(\al A)$, we have that
there exist 
subalgebras $\al B_i\le\al A^{p_i}$ with $p_i\le\pp$
for each $i$, and surjective homomorphisms $\phi_i\colon\al A_i\to\al B_i$
with kernels $\kappa_i$ $(i\in[n])$ such that
$\phi_i$ induces an isomorphism 
$\bar\phi_i\colon\al A_i/\kappa_i\to\al B_i$.
Recall that we chose the splitting 
triples so that $(\alpha_i,\beta_i,\kappa_i)=(\alpha_j,\beta_j,\kappa_j)$
whenever $i\approx j$ (and hence $\al A_i=\al A_j$).
With the same reasoning, we can arrange that 
the algebras $\al B_i$ and the
homomorphisms $\phi_i$ are selected so that
$\al B_i=\al B_j$ and $\phi_i=\phi_j$ whenever 
$i\approx j$. 
For every $i\in[n]$, 
define $\tilde{\alpha}_i\in\Con(\al B_i)$ to be the image of 
$\alpha_i\in\Con(\al A_i)$ under the homomorphism $\phi_i$,
which is the same as the image of 
$\alpha_i/\kappa_i\in\Con(\al A_i/\kappa_i)$ under the isomorphism
$\bar{\phi}_i$; so $\tilde{\alpha}_i$ is indeed a congruence of 
$\al B_i$.
Furthermore, define $\al B$ to be the image of $\al D$ under 
the product homomorphism
$\phi:=\prod_{i=1}^n\phi_i\colon\prod_{i=1}^n\al A_i\to\prod_{i=1}^n\al B_i$.

\begin{clm}
\label{clm-Bs-alphas}
Conditions {\rm(I)--(IV)} hold for 
the algebras $\al B_1,\dots,\al B_n$ and $\al B$, and 
for the congruences $\tilde{\alpha}_1,\dots,\tilde{\alpha}_n$
defined above.
Moreover, $\al D$ is the full inverse image of $\al B$ under the
homomorphism $\phi$.
\end{clm}

\noindent
{\it Proof of Claim~\ref{clm-Bs-alphas}.}
(I) follows from the choice of the $\al B_i$'s.

(II) holds, because for every $i\in[n]$,
$\tilde{\alpha}_i=\bar{\phi}_i[\alpha_i/\kappa_i]$
where $\bar{\phi}_i\colon\al A_i/\kappa_i\to\al B_i$
is an isomorphism and
$\alpha_i/\kappa_i$ is a nontrivial abelian congruence of $\al A_i$.
(The latter follows from the fact that
$(\alpha_i,\beta_i,\kappa_i)$ splits the relevant
triple $(\delta_i,\theta_i,\nu_i)$ of $\al A_i$.)
 
For (III), recall that
$\al D$ is a subdirect product of $\al A_1,\dots,\al A_n$
by Claim~\ref{clm-saturation}(1), and the homomorphisms
$\phi_i\colon\al A_i\to\al B_i$ ($i\in[n]$) are onto.
Since $\phi=\prod_{i=1}^n\phi_i$, it follows that
$\al B=\phi[\al D]$ 
is a subdirect product of $\al B_1,\dots,\al B_n$.

To verify (IV) notice that 
$\phi:=\prod_{i=1}^n\phi_i\colon\prod_{i=1}^n\al A_i\to\prod_{i=1}^n\al B_i$
is a surjective homomorphism with kernel
$\kappa=\prod_{i=1}^n\kappa_i$, therefore
$\phi$ decomposes as shown by the first line of the array below:
\begin{equation}
\label{eq-phi}
\begin{matrix}
\phi\colon  & \prod_{i=1}^n\al A_i &
\stackrel{\textrm{nat}}{\to} & \bigl(\prod_{i=1}^n\al A_i\bigr)/\kappa &
\stackrel{\cong}{\to} & \prod_{i=1}^n(\al A_i/\kappa_i) &
\stackrel{\bar{\phi}}{\to} & \prod_{i=1}^n\al B_i & \phantom{m}\\[6pt]
 & \alpha &
\mapsto & \alpha/\kappa &
\mapsto & \prod_{i=1}^n\alpha_i/\kappa_i &
\mapsto & \tilde{\alpha}\\[6pt]
\phi\restr_{\al D}\colon  & \al D &
\stackrel{\textrm{nat}}{\to} & \al D/\kappa_{\al D} &
            & \stackrel{\bar{\phi}\circ\cong}{\longrightarrow}  &
            & \al B\\[6pt]
 & \alpha_{\al D} &
\mapsto & \alpha_{\al D}/\kappa_{\al D} &
            & \mapsto & 
            &  \tilde{\alpha}_{\al B}\\
\end{matrix}
\end{equation}
The leftmost factor of $\phi$ is the natural map, 
the middle factor is the natural isomorphism, and
the rightmost factor is the isomorphism $\bar{\phi}:=\prod_{i=1}^n\bar{\phi}_i$.
Since $\tilde{\alpha_i}=\phi[\alpha_i]$ for every $i\in[n]$, $\phi$ maps
the congruence $\alpha=\prod_{i=1}^n\alpha_i$ of $\prod_{i=1}^n\al A_i$
onto the congruence
$\tilde{\alpha}=\prod_{i=1}^n\tilde{\alpha}_i$ of $\prod_{i=1}^n\al B_i$, via
this factorization, as indicated by the second line of
\eqref{eq-phi}.
Therefore, combining the two isomorphisms among the factors of $\phi$ and
restricting $\phi$ to $\al D$ yields that $\phi\restr_{\al D}$ factors
and acts on $\alpha_{\al D}$ as the third and fourth lines
of \eqref{eq-phi} show.
Thus,
\[
\al D/\alpha_{\al D}
\cong (\al D/\kappa_{\al D})/(\alpha_{\al D}/\kappa_{\al D})
\cong \al B/\tilde\alpha_{\al B}.
\]
This implies that the index of $\tilde{\alpha}_{\al B}$ in $\al B$
is $|\al B/\tilde\alpha_{\al B}|=|\al D/\alpha_{\al D}|$, which is $\le\ii$
by Claim~\ref{clm-index-of-alpha}.

For the last statement of the claim, $B=\phi[D]$ implies that
$\phi^{-1}[B]\supseteq D$. Since the kernel of $\phi$ is $\kappa$,
the equality $B=\phi[D]$ also implies that $\phi^{-1}[B]\subseteq D[\kappa]$.
But we know from Claim~\ref{clm-saturation} that 
$D[\kappa]=D$, so the equality $\phi^{-1}[B]=D$ we wanted to prove follows.
\hfill$\diamond$

\begin{clm}
\label{clm-C-entailed}
For every $i\in[n]$, $A_i$, $B_i$,
the (graph of the) homomorphism $\phi_i$, and 
\[
\beta_i\circ\phi_i:=\{(x,y)\in A_i\times B_i: x\equiv_{\beta_i} z
\ \text{and}\ y=\phi_i(z)\ \text{for some $z\in A_i$}
\}
\]
are compatible relations of $\al A$. Moreover,
\begin{equation}
\label{eq-C-entailed}
\{A_i,\,B_i,\,\phi_i,\,\,\beta_i\circ\phi_i:i\in[n]\}\cup\{B\}
\dentails C.
\end{equation}
\end{clm}

\noindent
{\it Proof of Claim~\ref{clm-C-entailed}.}
Choose any $i\in[n]$. Since $\al A_i\le\al A$ and $\al B_i\le\al A^{p_i}$,
we have that $A_i\in\rel{R}_1(\al A)$ and
$B_i\in\rel{R}_{p_i}(\al A)$.
The fact that $\phi_i$ is a homomorphism $\al A_i\to\al B_i$
implies that
its graph is the universe of a subalgebra of 
$\al A_i\times\al B_i\le\al A^{1+p_i}$. 
Thus, $\phi_i\in\rel{R}_{1+p_i}(\al A)$.
The definition of $\beta_i\circ\phi_i$ shows that
$\beta_i\circ\phi_i\not=\emptyset$ and
$\beta_i\circ\phi_i$ is definable by a primitive positive 
formula, using the compatible relations 
$A_i, B_i, \beta_i, \phi_i\in\rel{R}(\al A)$, 
therefore $\beta_i\circ\phi_i\in\rel{R}(\al A)$.
In fact, $\beta_i\circ\phi_i$ is also the universe
of a subalgebra of $\al A_i\times\al B_i\le\al A^{1+p_i}$, so
$\beta_i\circ\phi_i\in\rel{R}_{1+p_i}(\al A)$. 
This proves the first statement of the claim.

To prove the second statement, let us fix a transversal $T$ for the 
classes of $\sim$, and define $X$ to be the
set of all tuples
\[(c_1,\dots,c_n,b_1,\dots,b_n)\ 
\in \prod_{i=1}^n A_i\times\prod_{i=1}^n B_i\,\Bigl(\subseteq A^{n+\sum p_i}\Bigr)
\]
such that 
\begin{itemize}
\item
$(b_1,\dots,b_n)\in B$, 
\item
$(c_i,b_i)\in\beta_i\circ\phi_i$ for all $i\in[n]$, and
\item
$b_t=\phi_t(c_t)$ for all $t\in T$.
\end{itemize}
Our main goal is to show that
\begin{enumerate}
\item[(i)]
$\proj_{[n]}(X)=C$ and
\item[(ii)]
the projection map $\proj_{[n]}\colon X\to C$ is one-to-one,
\end{enumerate}
because we can deduce \eqref{eq-C-entailed} from 
statements (i)--(ii) as follows.
We have $X\not=\emptyset$ by (i), and
$X$ is definable by a quantifier-free primitive positive formula
using the relations $A_i$, $B_i$, $B$, $\phi_t$, and $\beta_i\circ\phi_i$, 
therefore
\[
\{A_i,\,B_i,\,\phi_i,\,\,\beta_i\circ\phi_i:i\in[n]\}\cup\{B\}
\dentails X.
\]
Furthermore, by (i) and (ii), 
$C$ can be obtained from $X$ by bijective projection, which is an
$\dentails$-construct, so $\{X\}\dentails C$. 
Hence \eqref{eq-C-entailed} follows by the transitivity of
$\dentails$.

It remains to prove (i) and (ii).
For the inclusion $\supseteq$ in (i), let $(c_1,\dots,c_n)\in C$.
By Claim~\ref{clm-c-d}, there exists $(d_1,\dots,d_n)\in D$ such that
$(d_1,\dots,d_n)\equiv_\beta(c_1,\dots,c_n)$ and $d_t=c_t$ for all $t\in T$.
Let $b_i=\phi_i(d_i)$ for each $i\in[n]$.
By the definition of $B$, this choice ensures that
$(b_1,\dots,b_n)\in B$.
Furthermore, $c_i\equiv_{\beta_i} d_i$ and $b_i=\phi_i(d_i)$ imply
$(c_i,b_i)\in\beta_i\circ\phi_i$ for every $i\in[n]$, while
$d_t=c_t$ and $b_t=\phi_t(d_t)$ imply
$b_t=\phi_t(c_t)$ for every $t\in T$.
Hence, $(c_1,\dots,c_n,b_1,\dots,b_n)\in X$, so
$(c_1,\dots,c_n)\in\proj_{[n]}(X)$.

The inclusion $\subseteq$ in (i)
and the claim in (ii) will follow if we prove the following statement:
\begin{enumerate}
\item[{}]
for every $(c_1,\dots,c_n)\in\proj_{[n]}(X)$ we have that 
\begin{enumerate}
\item[--]
$(c_1,\dots,c_n)\in C$, and 
\item[--]
there is a unique tuple $(b_1,\dots,b_n)$ 
such that $(c_1,\dots,c_n,b_1,\dots,b_n)\in X$.
\end{enumerate}
\end{enumerate}
So, let $(c_1,\dots,c_n)\in\proj_{[n]}(X)$.
Then $(c_1,\dots,c_n,b_1,\dots,b_n)\in X$ for at least one 
tuple $(b_1,\dots,b_n)$. By the definition of $X$
it must be the case that $(b_1,\dots,b_n)\in B$, 
$(c_i,b_i)\in\beta_i\circ\phi_i$ for all $i\in[n]$, and
$b_t=\phi_t(c_t)$ for all $t\in T$.
For any $t\in T$, define $d_t:=c_t$; so $c_t\equiv_{\beta_t}d_t$ and
$b_t=\phi_t(d_t)$.
For $i\in[n]\setminus T$,
use the definition of $\beta_i\circ\phi_i$ to get 
$d_i$ such that $c_i\equiv_{\beta_i} d_i$ and $b_i=\phi_i(d_i)$.
Thus, $(c_1,\dots,c_n)\equiv_{\beta}(d_1,\dots,d_n)$ and 
$(d_1,\dots,d_n)\in\phi^{-1}[(b_1,\dots,b_n)]$.
We established in Claim~\ref{clm-Bs-alphas} that
$D=\phi^{-1}[B]$, so 
$(d_1,\dots,d_n)\in\phi^{-1}[(b_1,\dots,b_n)]$
implies that $(d_1,\dots,d_n)\in D$.
Since $D\subseteq C$ and $C$ is $\beta$-saturated,
$(c_1,\dots,c_n)\equiv_{\beta}(d_1,\dots,d_n)$ yields that
$(c_1,\dots,c_n)\in C$.

For the uniqueness of $(b_1,\dots,b_n)$ observe that 
our argument in the preceding paragraph shows the following:
if $(b_1,\dots,b_n)$ is such that
$(c_1,\dots,c_n,b_1,\dots,b_n)\in X$, then there exists
$(d_1,\dots,d_n)\in D$ with the properties 
\[
(c_1,\dots,c_n)\equiv_\beta(d_1,\dots,d_n)\in\phi^{-1}[(b_1,\dots,b_n)]
\quad\text{and}\quad
\text{$c_t=d_t$ for all $t\in T$.}
\]
So, if $(b_1',\dots,b_n')$ is another tuple with
$(c_1,\dots,c_n,b_1',\dots,b_n')\in X$, then there exists
$(d_1',\dots,d_n')\in D$ such that 
\[
(c_1,\dots,c_n)\equiv_\beta(d_1',\dots,d_n')\in\phi^{-1}[(b_1',\dots,b_n')]
\quad\text{and}\quad
\text{$c_t=d_t'$ for all $t\in T$.}
\]
It follows from Claim~\ref{clm-c-d}(2) that 
$(d_1,\dots,d_n)\equiv_{\kappa_\al D}(d_1',\dots,d_n')$.
Since $\kappa_{\al D}$ is the kernel of $\phi$ we get that
\[
(b_1,\dots,b_n)
=\phi\bigl((d_1,\dots,d_n)\bigr)
=\phi\bigl((d_1',\dots,d_n')\bigr)
=(b_1',\dots,b_n'),
\] 
proving the uniqueness of $(b_1,\dots,b_n)$.
\hfill$\diamond$

\medskip
Statement $(*)$ of Theorem~\ref{thm-reduction} was proved in 
Claim~\ref{clm-Bs-alphas}. Statement $(**)$ of Theorem~\ref{thm-reduction}
follows from Claim~\ref{clm-C-entailed} and the fact
(established in the proof of Claim~\ref{clm-C-entailed}) that
each one of the relations $A_i,B_i,\phi_i,\beta_i\circ\phi_i$
($i\in[n]$)
is a member of $\rel{R}_m(\al A)$ for some $m\le 1+\pp$.

This completes the proof of Theorem~\ref{thm-reduction}.
\end{proof}

\section{Abelian Congruences and Modules}\label{modules}

In \cite[Chapter~9]{freese-mckenzie} Freese and McKenzie
describe how matrix rings of any size can be associated  
to a congruence modular variety $\var{V}$, and then
modules over these rings 
to any abelian congruence $\alpha$ of an algebra $\al C$ 
in $\var{V}$.
They prove in \cite[Theorem~9.9]{freese-mckenzie}
that there is a strong connection between
the pair $(\al C,\alpha)$ and the associated modules, namely,
if all $\alpha$-classes are represented in the module, then
the interval $\interval{0}{\alpha}$ in $\Con(\al C)$
is isomorphic to the lattice of submodules of 
the associated module.

Our goal in this section is to prove an analogous result for subalgebras
in place of congruences.
We will start by recalling the relevant definitions and basic facts
from \cite[Chapter~9]{freese-mckenzie}, but we will 
slightly change the notation.

Throughout this section we will work under the following global assumptions:
\begin{itemize}
\item
$\var{V}$ is a congruence modular variety,
\item
$d$ is a fixed difference term for $\var{V}$, and
\item
$\mathcal{O}$ is a fixed nonempty set of constant symbols not occurring 
in the language of $\var{V}$.
\end{itemize}
Recall that every 
congruence modular variety has a difference term 
(see \cite[Theorem~5.5]{freese-mckenzie}), which is a ternary term $d$
satisfying the following conditions
for every algebra $\al C\in\var{V}$ and congruence $\vartheta$ of $\al C$:
\begin{align}
d^{\al C}(x,x,y) & {}= y 
\quad \text{for all $x,y\in\al C$, and} 
\label{eq-diffterm1}\\
d^{\al C}(y,x,x) & {}\equiv_{[\vartheta,\vartheta]} y
\quad \text{for all $x,y\in\al C$ such that $x\equiv_\vartheta y$.}
\label{eq-diffterm2}
\end{align}
In addition, $d$ also satisfies
the following condition (see \cite[Proposition~5.7]{freese-mckenzie}): for 
every term $t=t(x_1,\dots,x_k)$ and for arbitrary abelian congruence $\alpha$
of an algebra $\al C\in\var{V}$, 
\begin{multline}
\label{eq-diffterm3}
d^{\al C}(t^{\al C}(u_1,\dots,u_k),t^{\al C}(v_1,\dots,v_k),
                                          t^{\al C}(w_1,\dots,w_k))\\
=t^{\al C}(d^{\al C}(u_1,v_1,w_1),\dots,d^{\al C}(u_k,v_k,w_k))\\
\text{whenever
$u_i,v_i,w_i\in\al C$ are such that
$u_i\equiv_\alpha v_i\equiv_\alpha w_i$ for all $i$.}
\end{multline}

We will use the set $\mathcal{O}$ of new constant symbols
to expand the language of $\var{V}$. 
Terms in the expanded language will
be called \emph{$\mathcal{O}$-terms}, and the algebras obtained 
from the members of $\var{V}$ by interpreting all constant symbols $o\in O$
will be called \emph{$\mathcal{O}$-algebras}.  
The class of algebras obtained in this way will
be denoted by $\Ovar{V}$. We will restrict the use of the phrases 
`subalgebra', `homomorphism', and `product' to the algebras in 
$\var{V}$ (i.e., to algebras in the original language), 
and use the phrases `$\mathcal{O}$-subalgebra', 
`$\mathcal{O}$-homomorphism', and
`$\mathcal{O}$-product' for $\mathcal{O}$-algebras.
There will be no need for using the phrase `$\mathcal{O}$-congruence',
because every $\mathcal{O}$-algebra has the same congruence lattice,
with the same commutator operation, as its reduct to the language of
$\var{V}$.

If we apply Corollary~5.8 (along with Lemma 5.6) 
from \cite{freese-mckenzie}
to $\mathcal{O}$-algebras in $\Ovar{V}$, we get the following.

\begin{lm}[From \cite{freese-mckenzie}]
\label{lm-group}
If $\alpha$ is an abelian congruence of 
an $\mathcal{O}$-algebra $\al C$ in $\Ovar{V}$, then
for every $o\in\mathcal{O}$, the $\alpha$-class containing $o^{\al C}$ 
is an abelian group 
\[
\Calpha{o}:=(o^{\al C}/\alpha;+_o,-_o,o^{\al C})
\]
with zero element $o^{\al C}$ for the operations
$+_o$ and $-_o$ defined by
\[
u +_o v := d^{\al C}(u,o^{\al C},v)
\quad\text{and}\quad
-_o u := d^{\al C}(o^{\al C},u,o^{\al C})
\quad\text{for all $u,v\in o^{\al C}/\alpha$.}
\]
Furthermore,
\begin{enumerate}
\item[{\rm(i)}]
$d^{\al C}(u,v,w)=u -_o v +_o w$ for all $u,v,w\in o^{\al C}/\alpha$, and
\item[{\rm(ii)}]
the $\mathcal{O}$-term operations of $\al C$ 
are linear between the $\alpha$-classes; more precisely,
if $r(x_1,\dots,x_k)$ is an $\mathcal{O}$-term and
$o_1,\dots,o_k,o\in\mathcal{O}$ are such that
$r(o_1,\dots,o_k)=o$, then 
\[
r^{\al C}\restr_{(o_1^{\al C}/\alpha)\times\dots\times(o_k^{\al C}/\alpha)}\colon
\Calpha{o_1}\times\dots\times\Calpha{o_k}
       \to\Calpha{o}
\]
is a group homomorphism. 
Consequently, for the unary $\mathcal{O}$-terms 
$r_i(x):=r(o_1,\dots,o_{i-1},x,o_{i+1},\dots,o_k)$ $(1\le i\le k)$
the maps
$r_i^{\al C}\restr_{o_i^{\al C}/\alpha}\colon
    \Calpha{o_i}\to \Calpha{o}$
are also group homomorphisms, and 
\begin{multline}
\label{eq-linDecomp}
\qquad\qquad
r^{\al C}(u_1,\dots,u_k)=r_1^{\al C}(u_1) +_o \dots +_o r_k^{\al C}(u_k)\\ 
\quad\text{for all $u_i\in o_i^{\al C}/\alpha$ $(i=1,\dots,k)$.}
\qquad
\end{multline}
\end{enumerate}
\end{lm}

In Section~5 the following straightforward consequence of 
Lemma~\ref{lm-group}(i) will be useful.

\begin{cor}
\label{cor-dMaltsev}
Under the same assumptions as in Lemma~\ref{lm-group},
if $e$ is an integer that is a multiple of the exponent of
the group $\Calpha{o}$, then
\begin{multline}
\label{eq-dMaltsev}
d^{\al C}(d^{\al C}(\dots d^{\al C}(d^{\al C}(u_1,u_{e+1},u_2),u_{e+1},u_3)
                          \dots),u_{e+1},u_e)
=u_1 +_o \dots +_o u_{e+1}\\
\quad\text{for all $u_1,\dots,u_{e+1}\in o^{\al C}/\alpha$.}
\end{multline}
\end{cor}

\begin{proof}
By Lemma~\ref{lm-group}(i) the left hand side of \eqref{eq-dMaltsev}
equals $u_1 +_o u_2 +_o u_3 +_o \dots +_o u_e -_o (e-1)u_{e+1}$,
which equals the right hand side of \eqref{eq-dMaltsev}, because
$e$ is a multiple of the exponent of the group $\Calpha{o}$.
\end{proof}

Next we define the matrix ring $\al R(\var{V},\mathcal{O})$
of the variety $\var{V}$ (where the size of the matrices is 
$|\mathcal{O}|$). Let $\al F$ be the free algebra in 
$\var{V}$ with free generating set $\{x\}\cup\mathcal{O}$.
In other words, $\al F$ is the algebra of all $\mathcal{O}$-terms 
of $\Ovar{V}$ in one variable $x$ (modulo the identities of $\Ovar{V}$).
For each $o\in\mathcal{O}$, let $\epsilon_o$ denote the unique 
endomorphism of $\al F$ that maps $x$ to $o$ and fixes all elements of 
$\mathcal{O}$; that is, $\epsilon_o$ maps every $\mathcal{O}$-term
$r=r(x)$ to $r(o)$. 
Furthermore, let $\gamma_o:=[\theta_o,\theta_o]$ where 
$\theta_o:=\ker(\epsilon_o)$. 
Now for any $o,o'\in\mathcal{O}$, let
\[
\bar H_{o,o'}:=\{r/\gamma_{o}:r=r(x)\in\al F,\ r(o)=o'\}.
\]
Then each $\bar H_{o,o'}$ has a natural abelian group structure with zero element
$o'/\gamma_{o}$ and addition and inversion defined by
\begin{align*}
r/\gamma_{o}+s/\gamma_{o}&{}:=d(r,o',s)/\gamma_{o}
\quad\text{and}\\
-(r/\gamma_{o})&{}:=d(o',r,o')/\gamma_{o}
\quad
\text{for all $r/\gamma_{o},s/\gamma_{o}\in \bar H_{o,o'}$.}
\end{align*}
Moreover, composition of 
$\mathcal{O}$-terms yields a multiplication that is defined by
\[
(t/\gamma_{o'})(r/\gamma_{o}):=t(r(x))/\gamma_{o}
\quad
\text{for all $t/\gamma_{o'}\in \bar H_{o',o''}$ and
$r/\gamma_{o}\in \bar H_{o,o'}$.}
\]
It can be also shown that multiplication distributes over addition.
Thus, the
$\mathcal{O}\times\mathcal{O}$
matrices $(m_{o',o})$ with the properties that 
\begin{itemize}
\item
the entry $m_{o',o}$ in position $(o',o)$
(i.e., row $o'$, column $o$) satisfies the condition
$m_{o',o}\in \bar H_{o,o'}$
for every $o,o'\in\mathcal{O}$, and 
\item
each column contains only finitely 
many nonzero entries 
\end{itemize}
form a ring for the ordinary matrix operations. 
This is the matrix ring $\al R(\var{V},\mathcal{O})$
associated to $\var{V}$. 

The proofs of all claims made throughout the definition 
above and the lemma below can be found at the beginning of 
Chapter~9 (before Theorem~9.4) in \cite{freese-mckenzie}.

\begin{lm}[From \cite{freese-mckenzie}]
\label{lm-module}
If $\alpha$ is an abelian congruence of 
an $\mathcal{O}$-algebra $\al C$ in $\Ovar{V}$, then
the direct sum
\begin{align*}
\qquad
\MCalpha
:={}&\bigoplus_{o\in\mathcal{O}}\Calpha{o}\\
={}&\{(a_o)_{o\in\mathcal{O}}\in\prod_{o\in\mathcal{O}}\Calpha{o}:
\text{$a_o=o^{\al C}$ for all but finitely many $o$'s}\}
\end{align*}
of the abelian groups $\Calpha{o}$ $(o\in O)$ is 
an $\al R(\var{V},\mathcal{O})$-module 
for the ordinary matrix-vector multiplication if
multiplication for the entries is defined as follows:
for $a\in\Calpha{o}$ and $m_{o',o}=r/\gamma_{o}\in\bar H_{o,o'}$
$(\text{with } r=r(x)\in\al F,\ r(o)=o')$, 
\begin{equation}
\label{eq-moduleMult}
m_{o',o}a=(r/\gamma_{o}) a:=r^{\al C}(a).
\end{equation}
\end{lm}

Now we are ready to state our first theorem which, for every pair
$(\al C,\alpha)$ consisting of an $\mathcal{O}$-algebra $\al C$ in $\Ovar{V}$
and an abelian congruence $\alpha$ of $\al C$, 
relates the subalgebras of $\al C$ to the submodules of the associated
$\al R(\var{V},\mathcal{O})$-module $\MCalpha$.
Recall our convention that the restriction of $\alpha$ to any subalgebra
$\al U$ (subset $U$) of $\al C$ is denoted by $\alpha_{\al U}$ 
($\alpha_U$, respectively).

\begin{thm}
\label{thm-subalg-submod}
Let $\alpha$ be an abelian congruence of an $\mathcal{O}$-algebra 
$\al C$ in $\Ovar{V}$ such that the constants $o^{\al C}$ $(o\in\mathcal{O})$ 
represent all $\alpha$-classes of $\al C$, and let 
$\al E:=\langle\mathcal{O}^{\al C}\rangle$ be the least $\mathcal{O}$-subalgebra 
of $\al C$.
\begin{enumerate}
\item[{\rm(1)}]
If $\al U$ is an $\mathcal{O}$-subalgebra of $\al C$, then 
$\MUalpha$ is an 
$\al R(\var{V},\mathcal{O})$-submodule 
of $\MCalpha$ that contains $\MEalpha$ as a submodule.
\item[{\rm(2)}]
The following conditions on any subset $U$ of $\al C$ 
are equivalent:
\begin{enumerate}
\item[{\rm(a)}]
$U$ is the universe of a subalgebra of $\al C$ that contains $\al E$;
\item[{\rm(b)}]
$U$ is the universe of an $\mathcal{O}$-subalgebra of $\al C$;
\item[{\rm(c)}]
$U$ is closed under all functions 
$t^{\al C}
(x_1,\dots,x_k)\restr_{(o_1^{\al C}/\alpha)\times\dots\times (o_k^{\al C}/\alpha)}$
where $t$ is an $\mathcal{O}$-term and $o_1,\dots,o_k\in\mathcal{O}$;
\item[{\rm(d)}]
the set
\[
\qquad\quad
\bigoplus_{o\in\mathcal{O}} (o^{\al C}/\alpha_U):=
\{(u_o)_{o\in\mathcal{O}}\in\prod_{o\in\mathcal{O}}(o^{\al C}/\alpha_U):
\text{$u_o=o^{\al C}$ for all but finitely many $o$'s}\}
\]
is the universe of an $\al R(\var{V},\mathcal{O})$-submodule 
of $\MCalpha$ that contains \break
$\MEalpha$.
\end{enumerate}
\item[{\rm(3)}]
If $U$ is the universe of an $\mathcal{O}$-subalgebra 
$\al U$ of $\al C$, then 
the $\al R(\var{V},\mathcal{O})$-submodule of
$\MCalpha$ described in {\rm(d)}
is $\MUalpha$.
\end{enumerate}
\end{thm}

\begin{proof}{}
For (1),
assume that $\al U$ is an $\mathcal{O}$-subalgebra
of $\al C$; in particular, $o^{\al U}=o^{\al C}$ for all $o\in\mathcal{O}$.
Then $\alpha_{\al U}$ 
is an abelian congruence of $\al U$, and 
the universe of $\MUalpha$ is 
a subset of the universe of 
$\MCalpha$.
The definition of the modules 
$\MUalpha$ and $\MCalpha$
shows that they have the same zero element 
$(o^{\al U})_{o\in\mathcal{O}}=(o^{\al C})_{o\in\mathcal{O}}$,
the group operations in the two modules
are determined, in each coordinate, by the 
same $\mathcal{O}$-terms $d(x,o,y)$, $d(o,x,o)$ 
(see Lemma~\ref{lm-group}),
and for every matrix in $\al R(\var{V},\mathcal{O})$,
multiplication for the entries in the two modules are determined by the same
$\mathcal{O}$-terms (see Lemma~\ref{lm-module}).
Since $\al U$ is an $\mathcal{O}$-subalgebra of $\al C$,
this implies that $\MUalpha$ is an 
$\al R(\var{V},\mathcal{O})$-submodule of $\MCalpha$.

$\al E$ is an $\mathcal{O}$-subalgebra of $\al U$,
therefore we can repeat the argument in the preceding paragraph
for $\al U$ and $\al E$ (in place of $\al C$ and $\al U$) to
conclude that $\MEalpha$ is an
$\al R(\var{V},\mathcal{O})$-submodule of
$\MUalpha$.

To verify (2) and (3),
we will prove the equivalence of the first three conditions in (2) by showing
that (a) $\Leftrightarrow$ (b) and (b) $\Leftrightarrow$ (c). To establish that
the fourth condition is also equivalent to them we will prove that 
(b) $\Rightarrow$ (d) and (d) $\Rightarrow$ (c). Statement (3) will 
be verified along with the proof of the implication (b) $\Rightarrow$ (d).

If $U$ is the universe of a subalgebra $\al U$ of $\al C$
containing $\al E=\langle\mathcal{O}^{\al C}\rangle$, then
we can define $o^{\al U}:=o^{\al C}$ for every 
$o\in\mathcal{O}$ to make $\al U$ an $\mathcal{O}$-subalgebra of $\al C$.
This proves (a) $\Rightarrow$ (b). The converse is a tautology, so 
(a) $\Leftrightarrow$ (b) is proved.

$U$ is the universe of an $\mathcal{O}$-subalgebra of $\al C$ if and only if
$U$ is closed under all functions $t^{\al C}$ induced by $\mathcal{O}$-terms $t$.
Since our assumption is that 
the constants $o^{\al C}$ $(o\in\mathcal{O})$ 
represent all $\alpha$-classes of $\al C$, it follows that for every 
$\mathcal{O}$-term $t$, $U$ is closed under the term function $t^{\al C}$
if and only if it is closed under all its restrictions described in (c).
This proves (b) $\Leftrightarrow$ (c).

Now assume that $U$ is the universe of an $\mathcal{O}$-subalgebra $\al U$
of $\al C$; 
in particular, $o^{\al U}=o^{\al C}$ for all $o\in\mathcal{O}$.
Then $\alpha_U=\alpha_{\al U}$  
is an abelian congruence of $\al U$. The definition of the
$\al R(\var{V},\mathcal{O})$-module 
$\MUalpha$
shows that its universe is exactly
$\bigoplus_{o\in\mathcal{O}} (o^{\al C}/\alpha_U)$.
Therefore statement (1) proves that 
$\bigoplus_{o\in\mathcal{O}} (o^{\al C}/\alpha_U)$ is indeed 
the underlying set of an $\al R(\var{V},\mathcal{O})$-submodule of
$\MCalpha$ that contains $\MEalpha$ as a submodule, namely 
$\MUalpha$.
This finishes the proof of both (b) $\Rightarrow$ (d) and (3).

Finally, we want to argue that
(d) $\Rightarrow$ (c).
We will start by proving a claim which is independent of $U$.

\begin{clm}
\label{clm-moduleops}
Let $\al C$ and $\alpha$ be as in the theorem.
For arbitrary $\mathcal{O}$-term $t=t(x_1,\dots,x_k)$ and constant symbols 
$o_1,\dots,o_k\in\mathcal{O}$,
\begin{enumerate}
\item[{\rm(1)}]
$t^{\al C}(o_1^{\al C},\dots,o_k^{\al C})\in o^{\al C}/\alpha$ 
for some $o\in\mathcal{O}$, and
\item[{\rm(2)}]
for every such $o\in\mathcal{O}$ there exist
$m_i=r_i/\gamma_{o_i}\in\bar H_{o_i,o}$ 
$($with $r_i=r_i(x)\in\al F$, $r_i(o_i)=o${}$)$ for 
$i=1,\dots,k$ such that
\begin{multline}
\label{eq-moduleops}
t^{\al C}(u_1,\dots,u_k)=m_1u_1 +_o \dots +_o m_ku_k +_o 
t^{\al C}(o_1^{\al C},\dots,o_k^{\al C})\\
\text{for all $(u_1,\dots,u_k)\in 
(o_1^{\al C}/\alpha)\times\dots\times(o_k^{\al C}/\alpha)$.}
\end{multline}
\end{enumerate}
\end{clm}

\noindent
{\it Proof of Claim~\ref{clm-moduleops}.}
To simplify notation, let $\bar{x}:=(x_1,\dots,x_k)$ and 
$\bar{o}:=(o_1,\dots,o_k)$.

(1) holds, because
our assumption that the constants $o^{\al C}$ $(o\in\mathcal{O})$ 
represent all $\alpha$-classes of $\al C$ implies that the element
$t^{\al C}(\bar{o}^{\al C})$ of $\al C$ is in $o^{\al C}/\alpha$
for some $o\in\mathcal{O}$. 

To prove (2)
let us fix such an $o\in\mathcal{O}$ for the rest of the argument.
Since $t=t(\bar{x})$ is an $\mathcal{O}$-term, so is 
\begin{equation}
\label{eq-translation}
r(\bar{x}):=d(t(\bar{x}),t(\bar{o}),o). 
\end{equation}
As in Lemma~\ref{lm-group}(ii), consider the unary $\mathcal{O}$-terms
$r_i(x):=r(o_1,\dots,o_{i-1},x,o_{i+1},\dots,o_k)$ for each $i$ ($1\le i\le k$).
Since $d$ satisfies the identity $d(x,x,y)=y$ (see \eqref{eq-diffterm1}),
we get that $r(\bar{o})=o$, and hence
$r_i(o_i)=o$ for every $i$ ($1\le i\le k$).
This shows that $r$ satisfies the assumptions of Lemma~\ref{lm-group}(ii),
and $m_i:=r_i/\gamma_{o_i}$ is a member of $\bar H_{o_i,o}$
for every $i$ ($1\le i\le k$).
Thus, first using equality \eqref{eq-linDecomp} 
from Lemma~\ref{lm-group}(ii), and then
equality \eqref{eq-moduleMult} from Lemma~\ref{lm-module} we obtain that
\begin{multline}
\label{eq-moduleops-almost}
r^{\al C}(u_1,\dots,u_k)
=r_1^{\al C}(u_1) +_o \dots +_o  r_k^{\al C}(u_k)
=m_1u_1 +_o \dots +_o m_ku_k\\
\text{for all $(u_1,\dots,u_k)\in 
(o_1^{\al C}/\alpha)\times\dots\times(o_k^{\al C}/\alpha)$.}
\end{multline}
On the other hand, since 
$r^{\al C}(u_1,\dots,u_k)\in o^{\al C}/\alpha$
for all $(u_1,\dots,u_k)\in 
(o_1^{\al C}/\alpha)\times\dots\times(o_k^{\al C}/\alpha)$
(cf.\ Lemma~\ref{lm-group}(ii)), we can apply 
first \eqref{eq-translation}, and then the equality in Lemma~\ref{lm-group}(i) 
to get that
\begin{multline*}
r^{\al C}(u_1,\dots,u_k)
=d^{\al C}(t^{\al C}(u_1,\dots,u_k),t^{\al C}(\bar{o}^{\al C}),o^{\al C})
=t^{\al C}(u_1,\dots,u_k) -_o t^{\al C}(\bar{o}^{\al C})\\
\text{for all $(u_1,\dots,u_k)\in 
(o_1^{\al C}/\alpha)\times\dots\times(o_k^{\al C}/\alpha)$.}
\end{multline*}
The last displayed equality shows that by adding $t^{\al C}(\bar{o}^{\al C})$
to both sides of \eqref{eq-moduleops-almost} 
in the group $\Calpha{o}$ we get the
equality \eqref{eq-moduleops} we wanted to prove.
\hfill$\diamond$

\medskip

To prove the implication (d) $\Rightarrow$ (c), assume that 
(d) holds for $U$, that is, the subset 
$S_U:=\bigoplus_{o\in\mathcal{O}} ({o}^{\al C}/\alpha_U)$
of $\MCalpha$ is the universe of an
$\al R(\var{V},\mathcal{O})$-submodule of 
$\MCalpha$, and contains 
all elements of $\MEalpha$. 
In particular, the zero element $(o^{\al C})_{o\in\mathcal{O}}$ of 
$\MCalpha$ belongs to $S_U$, which implies that
$\mathcal{O}^{\al C}\subseteq U$.

Let $t=t(x_1,\dots,x_k)$ be an arbitrary $\mathcal{O}$-term, 
and let $o_1,\dots,o_k\in\mathcal{O}$.
We want to show that $U$ is closed under the function
$t^{\al C}
(x_1,\dots,x_k)\restr_{(o_1^{\al C}/\alpha)\times\dots\times (o_k^{\al C}/\alpha)}$.
By Claim~\ref{clm-moduleops}
there exist $o\in\mathcal{O}$ 
and $m_i=r_i/\gamma_{o_i}\in\bar H_{o_i,o}$ 
$($with $r_i=r_i(x)\in\al F$, $r_i(o_i)=o${}$)$ for $i=1,\dots,k$ such that
$t^{\al C}(o_1^{\al C},\dots,o_k^{\al C})\in o^{\al C}/\alpha$ and
\eqref{eq-moduleops} holds.
Let us fix an arbitrary tuple 
$(u_1,\dots,u_k)\in
(o_1^{\al C}/\alpha)\times\dots\times (o_k^{\al C}/\alpha)$
that belongs to $U$; that is,
$(u_1,\dots,u_k)\in
(o_1^{\al C}/\alpha_U)\times\dots\times (o_k^{\al C}/\alpha_U)$.
Furthermore, let
$a:=t^{\al C}(o_1^{\al C},\dots,o_k^{\al C})$.
We will be done if we show that the element
\[
b:=t^{\al C}(u_1,\dots,u_k)=m_1u_1 +_o \dots +_o m_ku_k +_o a 
\]
belongs to $U$.

Note that $b=t^{\al C}(u_1,\dots,u_k)\equiv_\alpha 
t^{\al C}(o_1^{\al C},\dots,o_k^{\al C})=a$, therefore
$a\in o^{\al C}/\alpha$ implies that 
$b\in o^{\al C}/\alpha$. Moreover, since
$a=t^{\al C}(o_1^{\al C},\dots,o_k^{\al C})$ is an element of the least
$\mathcal{O}$-subalgebra $\al E$ of $\al C$ (in which
$o^{\al E}=o^{\al C}$), we get from $a\in o^{\al C}/\alpha$ that 
$a\in o^{\al E}/\alpha_{\al E}$.

Now, for each $i$ ($1\le i\le k$), 
let $\tilde{m}_i$ denote the 
$\mathcal{O}\times\mathcal{O}$ matrix which has $m_i$ in position $(o,o_i)$
and zeros (i.e., $o''/\gamma_{o'}$) in all other positions $(o'',o')$,
and let $\tilde{u}_i$ denote the $\mathcal{O}$-tuple
which has $u_i$ in its $o_i$-th position
and zeros (i.e., ${o'}^{\al C}$) in all other positions $o'\in\mathcal{O}$.
Furthermore, let $\tilde{a}$ and $\tilde{b}$ denote the $\mathcal{O}$-tuples
which have 
$a$ and $b$, respectively, in their
$o$-th positions, and zeros in all other positions.
Since $m_i\in\bar H_{o_i,o}$ for every $i$ ($1\le i\le k$), 
the definition of $\al R(\var{V},\mathcal{O})$ shows that
$\tilde{m}_1,\dots,\tilde{m}_k\in\al R(\var{V},\mathcal{O})$.
Since $u_i\in o_i^{\al C}/\alpha_U$ for every $i$ ($1\le i\le k$) and 
$\mathcal{O}^{\al C}\subseteq U$, 
we get that 
$\tilde{u}_1,\dots,\tilde{u}_k
\in S_U
\subseteq\MCalpha$.
Similarly, $a\in o^{\al E}/\alpha_{\al E}$ implies that
$\tilde{a}
\in \MEalpha$. Therefore, our assumption that $S_U$
contains the elements of $\MEalpha$ yields that
$\tilde{a}\in S_U
\subseteq\MCalpha$.
Finally, we have 
$\tilde{b}
\in\MCalpha$, because
$b\in o^{\al C}/\alpha$.
Since $b=m_1u_1 +_o \dots +_o m_ku_k +_o a$, 
the construction of the tuples $\tilde{u}_i$, $\tilde{a}$,
$\tilde{b}$ and the matrices $\tilde{m}_i$ makes sure that
the equality 
$\tilde{b}
=\tilde{m}_1\tilde{u}_1 + \dots + \tilde{m}_k\tilde{u}_k + \tilde{a}$
holds in the $\al R(\var{V},\mathcal{O})$-module 
$\MCalpha$.
Since 
$\tilde{u}_1,\dots,\tilde{u}_k,\tilde{a}
\in S_U$
and, by our assumption, $S_U$
is closed under the 
$\al R(\var{V},\mathcal{O})$-module operations of 
$\MCalpha$, we get that
$\tilde{b}\in S_U$.
Hence, $b\in o^{\al C}/\alpha_U\subseteq U$.
This completes the proof of (d) $\Rightarrow$ (c), and also the proof of
Theorem~\ref{thm-subalg-submod}
\end{proof}

Our second theorem is the analog of Theorem~9.9 in \cite{freese-mckenzie}
mentioned at the beginning of this section. 
Given a pair $(\al C,\alpha)$ consisting of an $\mathcal{O}$-algebra in
$\Ovar{V}$ and an abelian congruence $\alpha$ of $\al C$, the
theorem establishes
an isomorphism between the lattice of $\mathcal{O}$-subalgebras of 
$\al C$ and an interval
in the submodule lattice of the associated 
$\al R(\var{V},\mathcal{O})$-module $\MCalpha$,
provided all $\alpha$-classes are represented in the module.

\begin{thm}
\label{thm-lattice-iso}
Let $\al C$ be an $\mathcal{O}$-algebra 
in $\Ovar{V}$, and let $\al E:=\langle\mathcal{O}^{\al C}\rangle$ be
its least $\mathcal{O}$-subalgebra.
If $\alpha$ is an abelian congruence of 
$\al C$ such that the constants $o^{\al C}$ $(o\in\mathcal{O})$ 
represent all $\alpha$-classes of $\al C$, then the mapping 
\begin{equation}
\label{eq-lattice-iso}
\al U\mapsto\MUalpha
\end{equation}
is an isomorphism between the 
lattice of all $\mathcal{O}$-subalgebras of $\al C$ 
and the lattice of all $\al R(\var{V},\mathcal{O})$-submodules of 
$\MCalpha$ that contain
$\MEalpha$.
\end{thm}

\begin{proof}
It follows from Theorem~\ref{thm-subalg-submod}(1)
that for every $\mathcal{O}$-subalgebra $\al U$ of $\al C$,
$\MUalpha$ is an  
$\al R(\var{V},\mathcal{O})$-submodule of 
$\MCalpha$ that contains
$\MEalpha$.
Let $\al U$ and $\al V$ be arbitrary $\mathcal{O}$-subalgebras of
$\al C$. 
If $\al U\le\al V$, then
for every $o\in\mathcal{O}$, the congruence class 
$o^{\al U}/\alpha_{\al U}=o^{\al C}/\alpha_{\al U}$
in $\al U$ is contained in the corresponding congruence class
$o^{\al V}/\alpha_{\al V}=o^{\al C}/\alpha_{\al V}$
of $\al V$.
This implies that
$\MUalpha
\le\MValpha$,
showing that the map
\eqref{eq-lattice-iso} is order-preserving.
Now suppose that $\al U\not\le\al V$. 
The assumption that the interpretations of the 
constant symbols $o\in\mathcal{O}$ 
represent all $\alpha$-classes of $\al C$,
implies the existence of $o^{\al C}=o^{\al U}=o^{\al V}$
such that $o^{\al U}/\alpha_{\al U}\not\subseteq o^{\al V}/\alpha_{\al V}$.
Thus, 
$\MUalpha
\not\le\MValpha$.
This shows that the map \eqref{eq-lattice-iso} 
preserves and reflects $\le$; hence, in particular, it
is one-to-one.

The surjectivity of the map \eqref{eq-lattice-iso} is proved by
the next claim.

\begin{clm}
\label{clm-allSubmods}
Let $\al C$, $\al E$, and $\alpha$ be as in 
Theorem~\ref{thm-lattice-iso}.
If $\al N$ is an $\al R(\var{V},\mathcal{O})$-submodule of 
$\MCalpha$ that contains $\MEalpha$, then
\begin{enumerate}
\item[{\rm(1)}]
as an abelian group, $\al N=\bigoplus_{o\in\mathcal{O}}\al N_o$
where for every $o\in\mathcal{O}$, $\al N_o$ is the subgroup of 
$\Calpha{o}$ that
consists of the $o$-components of all tuples in $\al N$;
\item[{\rm(2)}]
$N_o=N_{o'}$ whenever $o^{\al C}/\alpha={o'}^{\al C}/\alpha$ 
$(o,o'\in\mathcal{O})$; 
\item[{\rm(3)}]
the union $U:=\bigcup_{o\in\mathcal{O}}N_o$ 
of the universes of the groups $\al N_o$
is the universe of an $\mathcal{O}$-subalgebra $\al U$ of $\al C$, and
$\al N=\MUalpha$.
\end{enumerate}
\end{clm}

\noindent
{\it Proof of Claim~\ref{clm-allSubmods}.}{}
For (1) notice first that
since the group operations in $\MCalpha$  
are performed coordinatewise, it follows immediately
from the definition of the $\al N_o$'s
that $\al N_o$ is a subgroup of 
$\Calpha{o}$ (in particular, $o^{\al C}\in\al N_o$) 
for each $o\in\mathcal{O}$, and 
$\al N\subseteq\bigoplus_{o\in\mathcal{O}}\al N_o$.

To prove that equality holds, consider
an arbitrary element 
$(a_o)_{o\in\mathcal{O}}\in\bigoplus_{o\in\mathcal{O}}\al N_o
\bigl(\subseteq \MCalpha\bigr)$.
Then (i)~$a_o\in N_o$ 
for every $o\in\mathcal{O}$, and
(ii)~$a_o=o^{\al C}$ for all but finitely many $o$'s.
Condition~(i) implies that for every $o\in\mathcal{O}$
there exists a tuple $\bar{a}_o\in N$ such that the $o$-th
coordinate of $\bar{a}_o$ is $a_o$. Let $e_o$ be the matrix in
$\al R(\var{V},\mathcal{O})$ which has $x/\gamma_o$ in position
$(o,o)$, and zeros (that is, $o''/\gamma_{o'}$) in all other positions
$(o'',o')$. Since $\al N$ is an $\al R(\var{V},\mathcal{O})$-submodule
of $\MCalpha$,
each tuple $e_o\bar{a}_o$ is an element of $\al N$. 
It follows from the construction of $e_o$ that the tuple $e_o\bar{a}_o$
has $a_o$ in the $o$-th coordinate and zeros elsewhere. 
By condition~(ii),
only finitely many of the tuples $e_o\bar{a}_o$ ($o\in\mathcal{O}$)
is not the zero element of $\al N$, therefore 
their sum in $\MCalpha$, and hence also in $\al N$,
is $(a_o)_{o\in\mathcal{O}}$. Thus $(a_o)_{o\in\mathcal{O}}\in\al N$,
which finishes the proof that $\al N=\bigoplus_{o\in\mathcal{O}}\al N_o$.

To prove (2) 
assume that $o,o'\in\mathcal{O}$ are such that 
$o^{\al C}/\alpha= {o'}^{\al C}/\alpha$.
Since $o^{\al E}=o^{\al C}$ and ${o'}^{\al E}={o'}^{\al C}$, we also have that
$o^{\al E}/\alpha_{\al E}={o'}^{\al E}/\alpha_{\al E}$. 
Therefore
our assumption that $\al N$ is an $\al R(\var{V},\mathcal{O})$-submodule of 
$\MCalpha$ that contains $\MEalpha$ implies that
\[
o^{\al E}/\alpha_{\al E}={o'}^{\al E}/\alpha_{\al E}\subseteq N_o\cap N_{o'}
\subseteq N_o\cup N_{o'}\subseteq o^{\al C}/\alpha= {o'}^{\al C}/\alpha.
\]
In particular, we have that $o^{\al E},{o'}^{\al E}\in N_o\cap N_{o'}$,
or equivalently, $o^{\al C},{o'}^{\al C}\in N_o\cap N_{o'}$.

For any element $a\in N_o$ let $\tilde{a}$ denote the $\mathcal{O}$-tuple
which has $a$ in its $o$-th position, and zeros 
in all other positions (i.e., ${o''}^{\al C}$ in the $o''$-th position for
$o''\not=o$). 
It follows from statement (1) that $\tilde{a}\in\al N$ for all $a\in N_o$.
Now let $r$ denote the $\mathcal{O}\times\mathcal{O}$
matrix which has $d(x,o,o')/\gamma_o$ in position
$(o',o)$ and zeros (i.e., $o'''/\gamma_{o''}$) in all other positions
$(o''',o'')$. Since $d(o,o,o')=o'$, we have that 
$d(x,o,o')/\gamma_o\in\bar{H}_{o,o'}$, and hence 
$r\in\al R(\var{V},\mathcal{O})$.
As $\al N$ is an $\al R(\var{V},\mathcal{O})$-submodule of 
$\MCalpha$, we get that
$r\tilde{a}\in\al N$ for all $a\in N_o$. 
Thus
the $o'$-th component of $r\tilde{a}$ is in $N_{o'}$; that is,
\begin{equation}
\label{eq-transl}
d^{\al C}(a,o^{\al C},{o'}^{\al C})=\bigl(d(x,o,o')/\gamma_o\bigr)a\in N_{o'}
\quad\text{for all $a\in N_o$.}
\end{equation}
By Lemma~\ref{lm-group}(i),
$d^{\al C}(u,o^{\al C},{o'}^{\al C})=u+_o{o'}^{\al C}$ holds in
$\Calpha{o}$ for all $u\in o^{\al C}/\alpha$.
Therefore \eqref{eq-transl} can be restated as follows:
the image of $N_o$ under the translation 
$u\mapsto u+_o{o'}^{\al C}$ is contained in $N_{o'}$.
Because of ${o'}^{\al C}\in N_o$ this translation 
is a bijection of the abelian group $\al N_o$ onto itself, 
so we get that $N_o\subseteq N_{o'}$. 
A similar argument, with the roles of $o$ and $o'$ switched, shows that 
$N_{o'}\subseteq N_o$ also holds, and hence finishes the proof of 
statement (2).

Finally, to establish (3),
recall that $U:=\bigcup_{o\in\mathcal{O}} N_o$. 
Since 
$(o^{\al C}\in)N_o\subseteq o^{\al C}/\alpha(\subseteq C)$
for every $o\in\mathcal{O}$, we get that $U\subseteq C$ and
that the sets $N_o$ and $N_{o'}$ are disjoint
whenever $o^{\al C}\not\equiv_\alpha {o'}^{\al C}$.
By statement (2), 
$N_o=N_{o'}$
whenever $o^{\al C}\equiv_\alpha {o'}^{\al C}$.
Therefore $\{N_o:o\in\mathcal{O}\}$
is the partition of $U$ corresponding to the equivalence relation
$\alpha_U$.
It follows now from statement (1) that the universe of $\al N$ is the set
$\bigoplus_{o\in\mathcal{O}} N_o=\bigoplus_{o\in\mathcal{O}} (o^{\al C}/\alpha_U)$.
Hence, the equivalence of conditions (d) and (b) in 
Theorem~\ref{thm-subalg-submod}(2) implies that $U$ is the universe
of an $\mathcal{O}$-subalgebra $\al U$ of $\al C$.
Moreover, by Theorem~\ref{thm-subalg-submod}(3), we have that
$\al N=\MUalpha$.
This finishes the proof of Claim~\ref{clm-allSubmods}.
\hfill$\diamond$

\medskip

The proof of Theorem~\ref{thm-lattice-iso} is complete.
\end{proof}

We conclude this section by an auxiliary result concerning the modules
associated to $\alpha$-saturated subalgebras of direct products 
where $\alpha$ is an abelian product congruence.

\begin{lm}
\label{lm-dir-pr}
Let $\al C$ be an $\mathcal{O}$-subalgebra of an $\mathcal{O}$-product
$\prod_{i=1}^n\al B_i$ of $\mathcal{O}$-algebras
$\al B_1,\dots,\al B_n$ in $\Ovar{V}$, and let $\alpha=\prod_{i=1}^n\alpha_i$
be the product congruence of $\prod_{i=1}^n\al B_i$
where for each $i$, $\alpha_i$ is an abelian congruence of $\al B_i$.
If $\al C$ is $\alpha$-saturated in $\prod_{i=1}^n\al B_i$, 
then 
\begin{enumerate}
\item[{\rm(1)}]
$\alpha_{\al C}$ is an abelian congruence of $\al C$ and
$o^{\al C}/\alpha_{\al C}=
\prod_{i=1}^n(o^{\al B_i}/\alpha_i)$ holds for all
$o\in\mathcal{O}$; moreover,
\item[{\rm(2)}]
the $\al R(\var{V},\mathcal{O})$-modules 
$\MCCalpha$ and 
$\prod_{i=1}^n\MBialpha$
are naturally isomorphic via the map defined by
\begin{equation}
\label{eq-saturated}
\begin{aligned}
\MCCalpha
& \to 
\prod_{i=1}^n\MBialpha  \\
(a_o)_{o\in\mathcal{O}}=\bigl((a_{o1},\dots,a_{on})\bigr)_{o\in\mathcal{O}}
& \mapsto 
\bigl((a_{o1})_{o\in\mathcal{O}},\dots,(a_{on})_{o\in\mathcal{O}}\bigr)
\end{aligned}
\end{equation}
for all 
$a_o=(a_{o1},\dots,a_{on})\in o^{\al C}/\alpha_{\al C}=
\prod_{i=1}^n(o^{\al B_i}/\alpha_i)$ $(o\in\mathcal{O})$.
\end{enumerate}
\end{lm}

\begin{proof}{}
First we prove (1).
Since each $\alpha_i$ is an abelian congruence of $\al B_i$, their
product, $\alpha$, is an abelian congruence of $\prod_{i=1}^n\al B_i$,
and hence $\alpha_{\al C}$ is an abelian congruence of $\al C$.
Therefore the module $\MCCalpha$
exists. 
By our assumption, $\prod_{i=1}^n\al B_i$ is an $\mathcal{O}$-product
and $\al C$ is an $\mathcal{O}$-subalgebra of
$\prod_{i=1}^n\al B_i$, so $o^{\al C}=(o^{\al B_1},\dots,o^{\al B_n})$
for every $o\in\mathcal{O}$. This implies that
$o^{\al C}/\alpha_{\al C}\subseteq\prod_{i=1}^n (o^{\al B_i}/\alpha_i)$
for every $o\in\mathcal{O}$.
In fact, $=$ holds here, because $\al C$ is $\alpha$-saturated in 
$\prod_{i=1}^n\al B_i$. 

Now we prove (2).
It follows from the equalities in (1) that
the elements of $\MCCalpha$
are exactly the tuples $(a_o)_{o\in\mathcal{O}}$ such that 
(i)~$a_o\in o^{\al C}/\alpha_{\al C}=\prod_{i=1}^n(o^{\al B_i}/\alpha_i)$, that is,
$a_o=(a_{o1},\dots,a_{on})$ with $a_{oi}\in o^{\al B_i}/\alpha_i$
for every $o$ and $i$, and 
(ii)~$a_o=o^{\al C}=(o^{\al B_1},\dots,o^{\al B_n})$ for all but finitely many 
$o$'s.
The image of each such tuple $(a_o)_{o\in\mathcal{O}}$ under the map
\eqref{eq-saturated} clearly belongs to the module 
$\prod_{i=1}^n\MBialpha$. 
The map \eqref{eq-saturated}
just regroups coordinates, so it is one-to-one. To see that it is also onto,
let $\bigl((b_{o1})_{o\in\mathcal{O}},\dots,(b_{on})_{o\in\mathcal{O}}\bigr)\in
\prod_{i=1}^n\MBialpha$.
Then, for every $i$, we have that  
(i)$'$~$b_{oi}\in o^{\al B_i}/\alpha_i$ for all $o\in\mathcal{O}$, and
(ii)$'$~$b_{oi}=o^{\al B_i}$ for all but finitely many $o$'s.
Therefore the tuples $b_o:=(b_{o1},\dots,b_{on})$ satisfy the conditions 
that 
(i)~$b_o\in \prod_{i=1}^n(o^{\al B_i}/\alpha_i)=o^{\al C}/\alpha_{\al C}$
for all $o\in\mathcal{O}$, and
(ii)~$b_o=(o^{\al B_1},\dots,o^{\al B_n})=o^{\al C}$ for all but finitely many $o$'s.
Hence $(b_o)_{o\in\mathcal{O}}\in\MCCalpha$, and
$\bigl((b_{o1})_{o\in\mathcal{O}},\dots,(b_{on})_{o\in\mathcal{O}}\bigr)$
is its image under the map \eqref{eq-saturated}. This proves that
\eqref{eq-saturated} is a bijection.

To see that \eqref{eq-saturated} is a group isomorphism, recall from 
Lemma~\ref{lm-module} that 
in each one of the modules $\MCCalpha$ and
$\MBialpha$ ($1\le i\le n$), $+$ 
is defined to be the coordinatewise operation that is $+_o$ in 
coordinate $o$ for each $o\in\mathcal{O}$.
Therefore, in the module 
$\prod_{i=1}^n\MBialpha$, $+$ is the 
coordinatewise operation that is the operation
$+_o$ on $o^{\al B_i}/\alpha_i\,(\subseteq\al B_i)$ in coordinate $(o,i)$.
In the module $\MCCalpha$,
the operation $+_o$ on $o^{\al C}/\alpha_{\al C}=\prod_{i=1}^n(o^{\al B_i}/\alpha_i)$
also acts coordinatewise, because $+_o$ is an $\mathcal{O}$-term
and $\al C$ is an $\mathcal{O}$-subalgebra of $\prod_{i=1}^n\al B_i$.
Therefore, in $\MCCalpha$, too,
$+$ is the 
coordinatewise operation that is the operation
$+_o$ on $o^{\al B_i}/\alpha_i\,(\subseteq\al B_i)$ in coordinate $(o,i)$.
This shows that \eqref{eq-saturated} is a group isomorphism.

Finally, we will argue that \eqref{eq-saturated} is 
an $\al R(\var{V},\mathcal{O})$-module isomorphism.
Recall again from Lemma~\ref{lm-module} and the discussion 
preceding it that 
$\al R(\var{V},\mathcal{O})$ is a ring of matrices, and multiplication
by elements of $\al R(\var{V},\mathcal{O})$ in each one of the modules
$\MCCalpha$ and
$\MBialpha$ ($1\le i\le n$) is defined by 
matrix multiplication. Furthermore, in the module 
$\prod_{i=1}^n\MBialpha$, multiplication by 
ring elements is performed coordinatewise.
Therefore we will be done if we show that 
\begin{enumerate}
\item[$(*)$]
multiplication by ring elements acts coordinatewise
on the components $a_o=(a_{o1},\dots,a_{on})$ of the elements
$(a_o)_{o\in\mathcal{O}}$ of $\MCCalpha$.
\end{enumerate}
Since  \eqref{eq-saturated} is a group isomorphism, it suffices to prove
$(*)$ for the multiplication of entries. 
Following the definition of multiplication of entries in Lemma~\ref{lm-module},
let $m_{o',o}=r/\gamma_o\in\bar{H}_{o,o'}$ (where $r=r(x)\in\al F$ and
$r(o)=o'$), and let 
$a_o=(a_{o1},\dots,a_{on})\in 
       o^{\al C}/\alpha_{\al C}=\prod_{i=1}^n(o^{\al B_i}/\alpha_i)$.
By definition, $m_{o',o}a_o=r^{\al C}(a_o)$.
Since $r$ is an $\mathcal{O}$-term and $\al C$ is an $\mathcal{O}$-subalgebra
of $\prod_{i=1}^n\al B_i$, we get that 
$r^{\al C}(a_o)=
\bigl(r^{\al B_i}(a_{oi})\bigr)_{i=1}^n
=(m_{o',o}a_{oi})_{i=1}^n$. 
This proves that multiplication by $m_{o',o}$ acts coordinatewise
on $a_o$, and hence finishes the proof of the lemma.
\end{proof}

The isomorphism \eqref{eq-saturated} shows that under the assumptions of
Lemma~\ref{lm-dir-pr} the only difference between the modules
$\MCCalpha$ and
$\prod_{i=1}^n\MBialpha$ is how we group their
coordinates. Therefore we may identify 
$\MCCalpha$ and
$\prod_{i=1}^n\MBialpha$ via the isomorphism
\eqref{eq-saturated}, and view the submodules of 
$\MCCalpha$ as 
submodules of the direct product 
$\prod_{i=1}^n\MBialpha$.

\section{Proof of Theorem~\ref{thm-main}}\label{mainthm-sec}

Throughout this section $\al A$ will be 
a finite algebra with a $k$-parallelogram term, and 
$\var{V}$ the variety generated by $\al A$.
Our goal is to combine the results of Sections~\ref{reduc} and \ref{modules} 
to prove that if $\al A$ satisfies the
split centralizer condition, then $\al A$ is dualizable.
In view of Theorem~\ref{thm-WZ}, we will be done if we can show 
that there exists a constant $\ccc$,
depending only on $\al A$, such that 
\begin{equation}
\label{eq-ToProve}
\rel{R}_\ccc({\al A})\dentails\rel{R}(\al A).
\end{equation}
Recall that $\rel{R}(\al A)$ stands for 
the set of all (finitary) compatible relations of
$\al A$, and for every positive integer $n$, 
$\rel{R}_n({\al A})$ denotes the set of all compatible relations 
of $\al A$ of arity $\le n$.
The relation $\dentails$ is described in Theorem~\ref{thm-dentails}.

In addition to the natural numbers $\aaa$, $\sss$, $\ii$, 
and $\pp$ introduced in Section~\ref{reduc}, we will need a few other
parameters related to $\al A$, which we introduce now. 
The list includes the constant $\ccc$ that we will use in the proof of
\eqref{eq-ToProve}.
Before defining the new parameters, let us fix a set $\mathcal{O}$ of
constant symbols not occurring in the language of $\var{V}$ such that
$|\mathcal{O}|=\ii$. 
Furthermore, let $\Mod(\al A,\pp)$ denote the
set of all $\al R(\var{V},\mathcal{O})$-modules $\MTalpha$
where $\al T$ is a subalgebra of $\al A^p$
for some $1\le p\le\pp$ that is isomorphic to a section of $\al A$, 
$\alpha$ is a nontrivial
abelian congruence of $\al T$, and $\al T$ is made into an 
$\mathcal{O}$-algebra by interpreting
the new constant symbols $o\in\mathcal{O}$
in $\al T$ in such a way that the
elements $o^{\al T}$ ($o\in\mathcal{O}$) represent every $\alpha$-class.
\begin{itemize}
\item
For every 
$\al R(\var{V},\mathcal{O})$-module
$\al Q$, let 
\[
\hh_{\al Q}:=\sum\{|\Hom(\al M,\al Q)|:\al M\in\Mod(\al A,\pp)\},
\] 
and
let $\hh$ be the maximum of the numbers $\hh_{\al Q}$ as
$\al Q$ runs over all subdirectly irreducible 
$\al R(\var{V},\mathcal{O})$-modules.
($\hh$ stands for `$\underline{\text{h}}$omomorphisms'.)
\item
Let $\ee$ be the least common multiple of the exponents of the additive 
groups of all modules $\al M\in\Mod(\al A,\pp)$. 
($\ee$ stands for `$\underline{\text{e}}$xponent'.)
\item
Let 
$\ccc_0:=\max(1+\pp,k-1)$, and let
$\ccc:=\max(\ccc_0,\pp\hh(\ee+1))$.
($\ccc$ stands for `$\underline{\text{c}}$onstant'.)
\end{itemize}
Note that the ring $\al R(\var{V},\mathcal{O})$ is finite, because
it consists of $\mathcal{O}\times\mathcal{O}$ matrices where each entry
is determined by an element of the free algebra $\al F$ in
$\var{V}$ with free generating set $\{x\}\cup\mathcal{O}$, 
and $\al F$ is finite, since $\var{V}$ is generated by $\al A$
(a finite algebra).
The set $\Mod(\al A,\pp)$ is also finite, since $\al A$, $\pp$, 
and $\mathcal{O}$ are finite. 
It follows from the main result of \cite{kearnes-module} that
$|\al Q|\le|\al R(\var{V},\mathcal{O})|$ holds 
for every subdirectly irreducible $\al R(\var{V},\mathcal{O})$-module
$\al Q$. Therefore there are only finitely many 
subdirectly irreducible $\al R(\var{V},\mathcal{O})$-modules, 
up to isomorphism, and all are finite, so $\hh$ is a natural number.
The finiteness of $\Mod(\al A,\pp)$ implies also that $\ee$ is a natural number,
and hence so are $\ccc_0$ and $\ccc$.

\begin{remrks}
\label{rem-about-rhqc}
(a)
The definition of $\Mod(\al A,\pp)$ shows that
$\Mod(\al A,\pp)\not=\emptyset$ if and only if
some power $\al A^p$ of $\al A$ with $1\le p\le\pp$ has a subalgebra
that is isomorphic to a section of $\al A$
with a nontrivial abelian congruence.

(b)
If $\Mod(\al A,\pp)=\emptyset$, then 
$\pp=1$,
$\ee=1$, $\hh=0$, 
and
$\ccc=\ccc_0=\max(2,k-1)$,
while if $\Mod(\al A,\pp)\not=\emptyset$, then 
$\pp\ge1$,
$\ee\ge2$, and
$\hh\ge|\Mod(\al A,\pp)|\ge 1$, because 
the summands in the definition of $\hh_{\al Q}$
satisfy $|\Hom(\al M,\al Q)|\ge 1$
for all $\al M\in\Mod(\al A,\pp)$ and all  
$\al R(\var{V},\mathcal{O})$-modules $\al Q$.

(c)
In particular, for the trivial $\al R(\var{V},\mathcal{O})$-module
$\al Q_0$ we have
$\hh_{\al Q_0}=|\Mod(\al A,\pp)|$, therefore the inequality in (b)
implies that $\hh\ge\hh_{\al Q_0}$.
\end{remrks}

We start the proof of Theorem~\ref{thm-main} by considering 
compatible relations $B$ of $\al A$ that are constructed in
Theorem~\ref{thm-reduction}.
We will use the same notation for these relations as in
Theorem~\ref{thm-reduction}, except that the 
$\tilde{\phantom{n}}$'s from the notation of
the congruences $\tilde\alpha_i$ will be omitted. 
So, the set of relations 
we will consider, and will denote by $\rel{R}^*(\al A)$,
consists of all compatible relations $B$ of $\al A$
that satisfy the following condition $(*)$ from Theorem~\ref{thm-reduction},
for some $n$:
\begin{enumerate}
\item[$(*)$]
There exist 
\begin{enumerate}
\item[{\rm(I)}]
subalgebras $\al B_i\le\al A^{p_i}$ with $p_i\le\pp$
for each $i\in[n]$
such that $\al B_i$ is isomorphic to a section of $\al A$, and  
\item[{\rm(II)}]
nontrivial abelian congruences 
$\alpha_i\in\Con(\al B_i)$ $(i\in[n])$
\end{enumerate}
such that
\begin{enumerate}
\item[{\rm(III)}]
$B$ is the universe of a
subdirect product $\al B$ of 
$\al B_1,\dots,\al B_n$, and
\item[{\rm(IV)}]
the product congruence
$\alpha:=\prod_{i=1}^n \alpha_i$ of $\prod_{i=1}^n\al B_i$ 
restricts to $\al B$ as a congruence 
$\alpha_{\al B}$ 
of index $\le\ii$.
\end{enumerate}
\end{enumerate}
We will refer to $n$ as the \emph{$*$-arity of $B$}.
Since 
$\al B\le\prod_{i=1}^n\al B_i\le\prod_{i=1}^n\al A^{p_i}$,
the arity of $B$, as a compatible 
relation of $\al A$, is $\sum_{i=1}^n p_i$.

\begin{lm}
\label{lm-main}
Let $B\in\rel{R}^*(\al A)$ be a compatible relation of $\al A$
of $*$-arity $n$, and let $\al B_i$, $\alpha_i$, $\al B$, and $\alpha$
be as in $(*)$.
If 
$n>\hh(\ee+1)$ 
and $\al B$ is 
a $\cap$-irreducible subalgebra
of its $\alpha$-saturation $\al B[\alpha]$ in $\prod_{i=1}^n\al B_i$,
then there exist 
\begin{enumerate}
\item[{\rm(1)}]
a compatible relation $B'\in\rel{R}^*(\al A)$ of $\al A$ of $*$-arity 
$\le n-\ee$ 
and
\item[{\rm(2)}]
subalgebras $\al D\le\al B_i^{\ee+1}$ and 
$\al D'\le\al B_i^2$
for some $i\in[n]$ 
\end{enumerate}
such that
\begin{enumerate}
\item[{\rm(3)}]
$\{B_1,\dots,B_n,B',D,D'\}\dentails B$.
\end{enumerate}
\end{lm}

\begin{proof}
Let $B$ satisfy the hypotheses of the lemma,
including the assumptions that 
$n>\hh(\ee+1)$ 
and 
$\al B$ is a $\cap$-irreducible subalgebra of $\al B[\alpha]$.
Since
$\al A$ has a parallelogram term, we know from Theorem~\ref{thm-parterm-cm}
that the variety $\var{V}$ 
it generates is congruence modular.
As in Section~\ref{modules}, we 
expand the language of $\var{V}$ by $\mathcal{O}$.

Let us fix interpretations $o^{\al B}$ for each $o\in\mathcal{O}$ in $\al B$ 
in such a way that the elements $o^{\al B}$ ($o\in\mathcal{O}$) represent all
$\alpha_{\al B}$-classes of $\al B$; this is possible, since one of our 
assumptions is that $\alpha_{\al B}$ has index $\le\ii$. 
So, $\al B$ becomes an $\mathcal{O}$-algebra in $\Ovar{V}$. 
Now we fix interpretations $o^{\al B_i}$
for every constant symbol $o\in\mathcal{O}$ 
in each $\al B_i$ ($i\in[n]$) such that
$o^{\al B}=(o^{\al B_1},\dots,o^{\al B_n})$. Thus, each $\al B_i$ will become
an $\mathcal{O}$-algebra in $\Ovar{V}$.
(Note that $\al B_i$ and $\al B_j$ might become different
$\mathcal{O}$-algebras even if $\al B_i$ and $\al B_j$ are the same 
as algebras in the original language of $\var{V}$.)
This choice makes sure that $\al B$ is a subdirect
$\mathcal{O}$-subalgebra of the $\mathcal{O}$-product
$\prod_{i=1}^n \al B_i$
of the $\mathcal{O}$-algebras $\al B_1,\dots,\al B_n$.
Since the elements $o^{\al B}$ ($o\in\mathcal{O}$)
represent all $\alpha_{\al B}$-classes of $\al B$, it follows that
for every $i\in[n]$, the elements $o^{\al B_i}$ ($o\in\mathcal{O}$)
represent all $\alpha_i$-classes of $\al B_i$.
Hence $\MBialpha\in\Mod(\al A,\pp)$ for every $i\in[n]$.

To make $\al B[\alpha]$ into an $\mathcal{O}$-algebra in $\Ovar{V}$,
let $o^{\al B[\alpha]}:=o^{\al B}$ for every $o\in\mathcal{O}$.
Then $\al B[\alpha]$ is also a subdirect
$\mathcal{O}$-subalgebra of the $\mathcal{O}$-product
$\prod_{i=1}^n \al B_i$
of the $\mathcal{O}$-algebras $\al B_1,\dots,\al B_n$,
and $\al B$ is an $\mathcal{O}$-subalgebra of
$\al B[\alpha]$. 
Moreover, since
every $\alpha_{\al B[\alpha]}$-class contains an $\alpha_{\al B}$-class,
the elements $o^{\al B[\alpha]}=o^{\al B}$ ($o\in\mathcal{O}$) represent all 
$\alpha_{\al B[\alpha]}$-classes of $\al B[\alpha]$.

Now let us consider the $\al R(\var{V},\mathcal{O})$-modules
$\MBalpha$ and $\MBsatalpha$.
The fact that $\al B$ is an $\mathcal{O}$-subalgebra of $\al B[\alpha]$
implies, by Theorem~\ref{thm-lattice-iso}, 
that $\MBalpha$ is an $\al R(\var{V},\mathcal{O})$-submodule of $\MBsatalpha$.  
Our additional assumption that $\al B$ is 
a $\cap$-irreducible subalgebra 
of $\al B[\alpha]$ implies 
that $\al B$ is also $\cap$-irreducible 
in the lattice of $\mathcal{O}$-subalgebras of $\al B[\alpha]$,
hence, again by Theorem~\ref{thm-lattice-iso},
$\MBalpha$ is a $\cap$-irreducible $\al R(\var{V},\mathcal{O})$-submodule
of $\MBsatalpha$.
Recall also that by Lemma~\ref{lm-dir-pr} (applied to $\al C=\al B[\alpha]$), 
the $\al R(\var{V},\mathcal{O})$-modules
$\MBsatalpha$ and $\prod_{i=1}^n\MBialpha$ are
naturally isomorphic.
Therefore we can identify the modules 
$\MBsatalpha$ and $\prod_{i=1}^n\MBialpha$
via this isomorphism, and get that 
$\MBalpha$ is a $\cap$-irreducible $\al R(\var{V},\mathcal{O})$-submodule
of $\prod_{i=1}^n\MBialpha$.
Thus, the quotient
\[
\al Q:=\Bigl(\prod_{i=1}^n\MBialpha\Bigr)\Big/\MBalpha
\]
is a subdirectly irreducible 
or trivial
$\al R(\var{V},\mathcal{O})$-module
(according to whether $\al B<\al B[\alpha]$ or 
$\al B=\al B[\alpha]$).

Let $\phi\colon\prod_{i=1}^n\MBialpha\to\al Q$
be the natural homomorphism.
For each $i\in[n]$,
the mapping $\psi_i\colon\MBialpha\to\al Q$ defined by
$\psi_i(z)=\phi(0,\dots,0,z,0,\dots,0)$ (with $z$ in the $i$-th position)
for all $z\in\MBialpha$
is an $\al R(\var{V},\mathcal{O})$-module homomorphism, and 
\[
\phi(z_1,\dots,z_n)=\sum_{i=1}^n\psi_i(z_i)
\quad
\text{for all $(z_1,\dots,z_n)\in\prod_{i=1}^n\MBialpha$.}
\]
Thus, for every element $(z_1,\dots,z_n)$ of $\prod_{i=1}^n\MBialpha$, 
\begin{equation}
\label{eq-solset}
(z_1,\dots,z_n)\in \setMBalpha
\quad\text{if and only if}\quad 
\sum_{i=1}^n\psi_i(z_i)=0;
\end{equation}
in other words, $\setMBalpha$ is the solution set of the equation
$\sum_{i=1}^n\psi_i(z_i)=0$ in $\prod_{i=1}^n\MBialpha$.

Since the $\al R(\var{V},\mathcal{O})$-module 
$\al Q$ is subdirectly irreducible 
or trivial,
and 
the $\al R(\var{V},\mathcal{O})$-modules $\MBialpha$
all belong to $\Mod(\al A,\pp)$, the definition of $\hh$ 
(combined with Remarks~\ref{rem-about-rhqc}~(c))
makes sure that
there are at most $\hh$ distinct homomorphisms among the $\psi_i$'s 
($i\in [n]$).
Therefore our assumption 
$n>\hh(\ee+1)$ 
forces that 
at least 
$\ee+2$ 
of the $\psi_i$'s are equal.
By permuting coordinates we may assume that 
for $i=1,\dots,\ee+2$ 
the $\psi_i$' are equal;
hence for $i=1,\dots,\ee+2$
the $\MBialpha$'s are equal as $\al R(\var{V},\mathcal{O})$-modules, 
the $\al B_i$'s are equal as $\mathcal{O}$-algebras, and
the $\alpha_i$'s are equal.
Let 
\begin{align}
\bar{\psi} &{}:= \psi_1=\dots=\psi_{\ee+2},\label{eq-psi-s}\\  
\bar{\al B} &{}:=\al B_1=\dots=\al B_{\ee+2},\label{eq-B-s}\\ 
\bar{\alpha} &{}:=\alpha_1=\dots=\alpha_{\ee+2},\label{eq-alpha-s}
\end{align}
and 
\begin{equation}
\label{eq-o-s}
o^{\bar{\al B}}:=o^{\al B_1}=\dots=o^{\al B_{\ee+2}}\quad
\text{for every $o\in\mathcal{O}$}.
\end{equation} 

From now on we will 
partition the coordinates of $\prod_{i=1}^n\al B_1$ 
(and its subalgebras $\al B[\alpha]$, $\al B$, etc.)
into two blocks, the first block consisting of the first 
$\ee+1$ coordinates, and the second from coordinates $\ee+2,\dots,n$.
Accordingly, we will
use the notation $\bar{\wec{x}}=(x_1,\dots,x_{\ee+1})$
for the elements of
$\bar{\al B}^{\ee+1}=\prod_{i=1}^{\ee+1}\al B_i$ 
and the notation $\wec{x}=(x_{\ee+2},\dots,x_n)$
for the elements of $\prod_{i=\ee+2}^n\al B_i$.
Hence, an element of $\prod_{i=1}^n\al B_i$ (in particular, of
$\al B$ or $\al B[\alpha]$) will be written as $(\bar{\wec{x}},\wec{x})$.
For any tuple $(\bar{\wec{x}},\wec{x})\in\al B[\alpha]$ let
$\mathcal{O}_{(\bar{\wec{x}},\wec{x})}$ denote
the set of all $o\in\mathcal{O}$ such that 
$(\bar{\wec{x}},\wec{x})\in o^{\al B[\alpha]}/\alpha_{\al B[\alpha]}$.

\begin{clm}
\label{clm-alpha}
The following hold for every element $(\bar{\wec{x}},\wec{x})\in\al B[\alpha]$:
\begin{enumerate}
\item[{\rm(1)}]
$\mathcal{O}_{(\bar{\wec{x}},\wec{x})}\not=\emptyset$, and for any 
$o\in\mathcal{O}$,
\begin{align*}
o\in\mathcal{O}_{(\bar{\wec{x}},\wec{x})}
\quad &\Leftrightarrow\quad
x_i\in o^{\al B_i}/\alpha_i\ \ \text{for all\ \ $i\in[n]$}\\
\quad &\Leftrightarrow\quad
x_i\in o^{\bar{\al B}}/\bar{\alpha}\ \ \text{for all\ \ $i\in[\ee+1]$, and}\\
\quad &\phantom{{}\Leftrightarrow{}}\quad
x_i\in o^{\al B_i}/\alpha_i\ \ \text{for all\ \ $i\in[n]\setminus[\ee+1]$.}
\end{align*}
\item[{\rm(2)}]
In particular,
$x_1\equiv_{\bar{\alpha}}x_2\equiv_{\bar{\alpha}}\dots\equiv_{\bar{\alpha}}
x_{\ee+1}$.
\end{enumerate}
\end{clm}

\noindent
{\it Proof of Claim~\ref{clm-alpha}.} Let
$(\bar{\wec{x}},\wec{x})\in\al B[\alpha]$.

To prove (1) observe first that  
since the elements $o^{\al B[\alpha]}$ ($o\in\mathcal{O}$) represent all 
$\alpha_{\al B[\alpha]}$-classes, there is an $o\in\mathcal{O}$ such that
$o^{\al B[\alpha]}$ is in the $\alpha_{\al B[\alpha]}$-class of 
$(\bar{\wec{x}},\wec{x})$. 
Hence $o\in\mathcal{O}_{(\bar{\wec{x}},\wec{x})}$, proving that 
$\mathcal{O}_{(\bar{\wec{x}},\wec{x})}\not=\emptyset$.
For the second statement in (1), recall that 
$o^{\al B[\alpha]}=(o^{\al B_1},\dots,o^{\al B_n})$ for every
$o\in\mathcal{O}$, and that $\alpha_{\al B[\alpha]}$ is the restriction of
the product congruence $\alpha=\prod_{i=1}^n\alpha_i$ to $\al B[\alpha]$.
Thus, $(\bar{\wec{x}},\wec{x})\in o^{\al B[\alpha]}/\alpha_{\al B[\alpha]}$
if and only if 
$x_i\in o^{\al B_i}/\alpha_i$ for all $i\in[n]$.
This, together with the definition of $\mathcal{O}_{(\bar{\wec{x}},\wec{x})}$
implies the first $\Leftrightarrow$ in the displayed statement.
The second $\Leftrightarrow$ follows by \eqref{eq-B-s}--\eqref{eq-o-s}.

For (2), 
we get from (1) that if $o\in\mathcal{O}_{(\bar{\wec{x}},\wec{x})}$ then
$x_1,\dots,x_{\ee+1}\in o^{\bar{\al B}}/\bar{\alpha}$. Since 
$\mathcal{O}_{(\bar{\wec{x}},\wec{x})}\not=\emptyset$, this implies that
$x_1,\dots,x_{\ee+1}$ are in the same $\bar{\alpha}$-class.
\hfill$\diamond$

\medskip

\begin{clm}
\label{clm-y}
The following conditions on
an element $(\bar{\wec{x}},\wec{x})\in\al B[\alpha]$
are equivalent: 
\begin{enumerate}
\item[{\rm(a)}]
$(\bar{\wec{x}},\wec{x})\in\al B$;
\item[{\rm(b)}]
$(y,\dots,y,\wec{x})\in\al B$ (with $\ee+1$ occurrences of $y$)
for the element $y$ defined by
\begin{equation}
\label{eq-y}
y=d^{\bar{\al B}}(d^{\bar{\al B}}(\dots 
d^{\bar{\al B}}(d^{\bar{\al B}}(x_1,x_{\ee+1},x_2),x_{\ee+1},x_3)\dots),x_{\ee+1},x_\ee).
\end{equation}
\end{enumerate}
\end{clm}

\noindent
{\it Proof of Claim~\ref{clm-y}.}
Throughout the proof, we will work with a fixed (but arbitrary) tuple 
$(\bar{\wec{x}},\wec{x})\in\al B[\alpha]$, and $y$ will denote 
the element defined in \eqref{eq-y}.
By Claim~\ref{clm-alpha}(1), 
$\mathcal{O}_{(\bar{\wec{x}},\wec{x})}\not=\emptyset$ and
we have $x_i\in o^{\al B_i}/\alpha_i$ whenever
$o\in\mathcal{O}_{(\bar{\wec{x}},\wec{x})}$ and $i\in[n]$. 
Consequently, for each $o\in\mathcal{O}_{(\bar{\wec{x}},\wec{x})}$ and $i\in[n]$,
the $\mathcal{O}$-tuple
$\otupl{x_i}$ that has $x_i$
in position $o$ and zeros (i.e., ${o'}^{\al B_i}$) 
in all other positions $o'$ is an element of $\MBialpha$. 
Similarly, 
the $\mathcal{O}$-tuple
$\otupl{(\bar{\wec{x}},\wec{x})}$ 
that has $(\bar{\wec{x}},\wec{x})=(x_i)_{i\in[n]}$ in position $o$ and 
zeros (i.e., ${o'}^{\al B[\alpha]}=({o'}^{\al B_1},\dots,{o'}^{\al B_n})$)
in all other positions $o'$ is
an element of $\MBsatalpha$.
By inspecting the isomorphism that we use to identify
$\MBsatalpha$ with $\prod_{i=1}^n\MBialpha$ it is easy to see that
this identification yields that 
$\otupl{(\bar{\wec{x}},\wec{x})}=(\otupl{x_i})_{i\in[n]}$.

Therefore, if $(\bar{\wec{x}},\wec{x})\in\al B$,
then
$(\otupl{x_i})_{i\in[n]}=\otupl{(\bar{\wec{x}},\wec{x})}\in\setMBalpha$ for all 
$o\in\mathcal{O}_{(\bar{\wec{x}},\wec{x})}$, so 
using \eqref{eq-solset} we see that
(a) implies the following condition:
\begin{enumerate}
\item[(a)$'$]
$\sum_{i=1}^n \psi_i(\otupl{x_i})=0$
for all $o\in\mathcal{O}_{(\bar{\wec{x}},\wec{x})}$.
\end{enumerate}
Conversely, assume that (a)$'$ holds for $(\bar{\wec{x}},\wec{x})$.
Since $\mathcal{O}_{(\bar{\wec{x}},\wec{x})}\not=\emptyset$,
we can fix an $o\in\mathcal{O}_{(\bar{\wec{x}},\wec{x})}$
and use \eqref{eq-solset} to conclude that 
$\otupl{(\bar{\wec{x}},\wec{x})}=(\otupl{x_i})_{i\in[n]}\in\setMBalpha$.
By the definition of $\MBalpha$, this implies that
$(\bar{\wec{x}},\wec{x})\in\al B$.
Thus, (a) $\Leftrightarrow$ (a)$'$.

Now let us consider the element $y\in\bar{\al B}$.
By Claim~\ref{clm-alpha}(1), for every $o\in\mathcal{O}_{(\bar{\wec{x}},\wec{x})}$
we have that $x_1,\dots,x_{\ee+1}\in o^{\bar{\al B}}/\bar{\alpha}$,
so the idempotence of $d$ implies that
$y\in o^{\bar{\al B}}/\bar{\alpha}$.
Since $\mathcal{O}_{(\bar{\wec{x}},\wec{x})}\not=\emptyset$, 
it follows that $(\bar{\wec{x}},\wec{x})\equiv_{\alpha}(y,\dots,y,\wec{x})$.
Thus, $(y,\dots,y,\wec{x})\in\al B[\alpha]$ and
the tuples $(\bar{\wec{x}},\wec{x})$ and 
$(y,\dots,y,\wec{x})$ belong to the same $\alpha_{\al B[\alpha]}$-classes
$o^{\al B[\alpha]}/\alpha_{\al B[\alpha]}$ ($o\in\mathcal{O}$).
Consequently, 
$\mathcal{O}_{(\bar{\wec{x}},\wec{x})}=\mathcal{O}_{(y,\dots,y,\wec{x})}$. 

For every 
$o\in\mathcal{O}_{(y,\dots,y,\wec{x})}=\mathcal{O}_{(\bar{\wec{x}},\wec{x})}$ 
the $\mathcal{O}$-tuple 
$\otupl{y}$ that has $y$
in position $o$ and zeros (i.e., ${o'}^{\bar{\al B}}$) 
in all other positions $o'$ is an
element of $\MBbaralpha$, and 
the $\mathcal{O}$-tuple 
$\otupl{(y,\dots,y,\wec{x})}$ 
that has $(y,\dots,y,\wec{x})=(y,\dots,y,x_{\ee+2},\dots,x_n)$ 
in position $o$ and 
zeros (i.e., ${o'}^{\al B[\alpha]}=({o'}^{\al B_1},\dots,{o'}^{\al B_n})$)
in all other positions $o'$ is an
element of $\MBsatalpha$.
As before, the identification of $\MBsatalpha$
with $\prod_{i=1}^n\MBialpha$ yields that 
$\otupl{(y,\dots,y,\wec{x})}=
(\otupl{y},\dots,\otupl{y},\otupl{x_{\ee+2}},\dots,\otupl{x_n})$.
Hence, if we apply
the equivalence of conditions (a) and (a)$'$ to the tuple
$(y,\dots,y,\wec{x})\in\al B[\alpha]$, we obtain that 
(b) $\Leftrightarrow$ (b)$'$ for the condition
\begin{enumerate}
\item[(b)$'$]
$\sum_{i=1}^{\ee+1}\psi_i(\otupl{y})+\sum_{i=\ee+2}^n \psi_i(\otupl{x_i})=0$
for all $o\in\mathcal{O}_{(y,\dots,y,\wec{x})}=
\mathcal{O}_{(\bar{\wec{x}},\wec{x})}$.
\end{enumerate}
Thus, our claim (a) $\Leftrightarrow$ (b) will follow if we prove that
(a)$'$ $\Leftrightarrow$ (b)$'$.

The only difference between conditions (a)$'$ and (b)$'$ are in the first 
$\ee+1$ summands. Therefore we will be done if we show that
\begin{equation}
\label{eq-q-sum}
\sum_{i=1}^{\ee+1}\psi(\otupl{x_i})=\sum_{i=1}^{\ee+1}\psi_i(\otupl{y})
\quad
\text{for every 
$o\in\mathcal{O}_{(\bar{\wec{x}},\wec{x})}=\mathcal{O}_{(y,\dots,y,\wec{x})}$.}
\end{equation}
So, let 
$o\in\mathcal{O}_{(\bar{\wec{x}},\wec{x})}=\mathcal{O}_{(y,\dots,y,\wec{x})}$.
Then, by Claim~\ref{clm-alpha}(1) and by our earlier argument on $y$,
we have that
$x_1,\dots,x_{\ee+1},y\in o^{\bar{\al B}}/\bar{\alpha}$. 
By combining the definition of $y$ in \eqref{eq-y} with 
Corollary~\ref{cor-dMaltsev} we get that
$y=x_1 +_o x_2 +_o \dots +_o x_{\ee+1}$
holds in the abelian group $\al M_{\bar{\al B}}(\bar{\alpha},o)$.
Hence, by the definition of $+$ in the module $\MBbaralpha$,
\begin{equation}
\label{eq-yo}
\otupl{y}=\otupl{x_1} + \otupl{x_2} + \dots + \otupl{x_{\ee+1}}
\end{equation}
holds in $\MBbaralpha$.
So, using first \eqref{eq-psi-s}, then the fact that
$\bar{\psi}$ is a module homomorphism $\MBbaralpha\to\al Q$, 
and finally \eqref{eq-yo},
we get that
\[
\sum_{i=1}^{\ee+1}\psi_i(\otupl{x_i})
=\sum_{i=1}^{\ee+1}\bar{\psi}(\otupl{x_i})
=\bar{\psi}\Big(\sum_{i=1}^{\ee+1}\otupl{x_i}\Bigr)
=\bar{\psi}(\otupl{y}).
\]
Similarly,
\[
\sum_{i=1}^{\ee+1}\psi_i(\otupl{y})
=(\ee+1)\bar{\psi}(\otupl{y})
=\bar{\psi}((\ee+1)\otupl{y})
=\bar{\psi}(\otupl{y}),
\]
where the last equality is true, because the 
definition of $\ee$ and the fact that 
$\otupl{y}\in\MBbaralpha\in\Mod(\al A,\pp)$
ensure that $\ee$ is a multiple of the additive order of $\otupl{y}$.
This establishes \eqref{eq-q-sum}, and hence completes the proof of the claim.
\hfill$\diamond$

\medskip

Now we define the relations $B'$ and $D$ that we will use to prove
Lemma~\ref{lm-main}. Let
\begin{gather}
B':=\{(y,\wec{x})\in\bar{B}\times\prod_{i=\ee+2}^n B_i:
(y,\dots,y,\wec{x})\in B\},\notag\\
D:=\{(y,\bar{\wec{x}})\in\bar{B}^{\ee+2}:
\text{\eqref{eq-y} holds and \ }
x_1\equiv_{\bar{\alpha}}\dots\equiv_{\bar{\alpha}}x_{\ee+1}\}.
\notag
\end{gather}

\begin{clm}
\label{clm-B'}
$B'$ has the following properties.
\begin{enumerate}
\item[(1)]
$\bar{\al B}$ has an $\mathcal{O}$-subalgebra
$\bar{\al B}'$ such that 
$B'$ is the universe of a 
subdirect $\mathcal{O}$-subalgebra $\al B'$ of 
the $\mathcal{O}$-product $\bar{\al B}'\times\prod_{i=\ee+2}^n\al B_i$
of the $\mathcal{O}$-algebras $\bar{\al B}'$ and $\al B_i$
$(\ee+2\le i\le n)$.
\item[(2)]
All the congruences 
$\bar{\alpha}':=\bar{\alpha}\restr_{\bar{\al B}'}\in\Con(\bar{\al B}')$
and $\alpha_i\in\Con(\al B_i)$ $(\ee+2\le i\le n)$ are abelian, and all
but possibly $\bar{\alpha}'$ are nontrivial. 
\item[(3)]
For the product congruence 
$\alpha':=\bar{\alpha}'\times\prod_{i=\ee+2}^n\alpha_i$
of\/ $\bar{\al B}'\times\prod_{i=\ee+2}^n\al B_i$,
the elements $o^{\al B'}$ $(o\in\mathcal{O})$ represent all
$\alpha'_{\al B'}$-classes of $\al B'$, therefore $\alpha'$
restricts to $\al B'$ as a congruence $\alpha'_{\al B'}$ of index $\le\ii$.
\end{enumerate}
Consequently, $\al B'$ satisfies all conditions in (I)--(IV), 
with the possible exception that the abelian congruence 
$\bar{\alpha}'$ in the first factor might be trivial.
Hence, if $\bar{\alpha}'$ is a nontrivial
congruence of $\bar{\al B}'$, then 
$B'\in\rel{R}^*(\al A)$ is a compatible relation of $\al A$ of $*$-arity
$n-\ee$.
\end{clm}

\noindent
{\it Proof of Claim~\ref{clm-B'}}.
Notice 
first that 
$B'$ is a compatible relation of $\al A$, because 
$B'\not=\emptyset$ by Claim~\ref{clm-y}, and
$B'$ is definable
by a primitive positive formula using the compatible relations
$\bar B=B_1$, $B_i$ ($\ee+2\le i\le n$), and $B$. 
Thus the definition of $B'$ shows that
$B'$ is, in fact, the universe of a subalgebra $\al B'$ of 
$\bar{\al B}\times\prod_{i=\ee+2}^n \al B_i$.
Moreover, \eqref{eq-o-s} implies that all
tuples $(o^{\bar{\al B}},o^{\al B_{\ee+2}},\dots,o^{\al B_n})$ ($o\in\mathcal{O}$)
are in $\al B'$, therefore by defining
$o^{\al B'}:=(o^{\bar{\al B}},o^{\al B_{\ee+2}},\dots,o^{\al B_n})$ 
for every $o\in\mathcal{O}$, $\al B'$ becomes an $\mathcal{O}$-subalgebra of
$\bar{\al B}\times \prod_{i=\ee+2}^n \al B_i$.
Claim~\ref{clm-y} implies that $\al B$ and $\al B'$ have the same 
projections onto their last $n-(\ee+1)$ coordinates, so
letting $\bar{\al B}'$ 
denote the projection of $\al B'$ onto its 
first coordinate, we get an
$\mathcal{O}$-subalgebra 
$\bar{\al B}'$ of $\bar{\al B}$
such that $\al B'$ is a subdirect $\mathcal{O}$-subalgebra of 
the $\mathcal{O}$-product $\bar{\al B}'\times\prod_{i=\ee+2}^n\al B_i$
of the $\mathcal{O}$-algebras $\bar{\al B}'$ and $\al B_i$
($\ee+2\le i\le n$).
This proves (1).

Considering $\bar{\al B}'$ as 
an algebra in the original language of $\var{V}$, we have
$\bar{\al B}'\le\bar{\al B}=\al B_1\le\al A^{p_1}$
where $p_1\le\pp$. 
Furthermore, since $\bar{\al B}$ is
isomorphic to a section of $\al A$, so is its subalgebra
$\bar{\al B}'$.
The other algebras $\al B_{\ee+2},\dots\al B_n$ are unchanged, therefore
we get that conditions (I) and (III) hold for $B'$. 

Condition (II) for $\al B$ implies that
$\bar{\alpha}'$ is an abelian congruence of $\bar{\al B}'$ and
$\alpha_i$ is a nontrivial abelian congruence of $\al B_i$ for every $i$
($\ee+2\le i\le n$). This proves (2) and that (II) holds for $\al B'$
with the possible exception that $\bar{\alpha}'$ may be trivial.

Finally, since 
the elements 
$o^{\al B}=(o^{\bar{\al B}},\dots,o^{\bar{\al B}},o^{\al B_{\ee+2}},\dots,o^{\al B_n})$
($o\in\mathcal{O}$)
represent all $\alpha_{\al B}$-classes in $\al B$, the definition of
$\al B'$ implies that the elements
$o^{\al B'}=(o^{\bar{\al B}},o^{\al B_{\ee+2}},\dots,o^{\al B_n})$
($o\in\mathcal{O}$)
represent all 
$\alpha'_{\al B'}$-classes in $\al B'$.
Hence 
$\alpha'_{\al B'}$
has index $\le|\mathcal{O}|=\ii$.
This proves 
(3) and
that condition (IV) also hold for $B'$.

Thus 
we have $B'\in\rel{R}^*(\al A)$ 
if $\bar{\alpha}'$ is a nontrivial congruence of $\bar{\al B}'$.
It is clear from the construction of $B'$ that its $*$-arity is $n-\ee$,
completing the proof of 
Claim~\ref{clm-B'}.
\hfill$\diamond$

\smallskip

\begin{clm}
\label{clm-D}
$D$ is the universe of a subalgebra $\al D$ of $\bar{\al B}^{\ee+1}$.
\end{clm}

\noindent
{\it Proof of Claim~\ref{clm-D}}.
The following fact will be useful:
property \eqref{eq-diffterm3} of the difference term $d$
for $\bar{\al B}$ and its abelian congruence $\bar{\alpha}$ 
is equivalent to saying that
\[
D_1:=\bigl\{\bigl(d^{\bar{\al B}}(u,v,w),u,v,w\bigr)\in\bar{B}^4: 
u\equiv_{\bar{\alpha}}v\equiv_{\bar{\alpha}}w\bigr\}
\]
is the universe of a subalgebra of
$\bar{\al B}^4$. 
Hence $D_1$ is a compatible relation of $\al A$.
Observing that all four coordinates of the tuples in $D_1$ are 
$\bar{\alpha}$-related, one can easily check that 
$D$ is definable by a primitive positive formula using $D_1$.
Clearly, $D\not=\emptyset$.
Therefore it follows that $D$ is a compatible relation of $\al A$.
The construction of $D$ shows that $D$ is, in fact, 
the universe of a subalgebra
of $\bar{\al B}^{\ee+1}=\al B_1^{\ee+1}$.
\hfill$\diamond$

\smallskip

\begin{clm}
\label{clm-entailB}
$\{B_1,\dots,B_n,B',D\}\dentails B$.
\end{clm}

\noindent
{\it Proof of Claim~\ref{clm-entailB}}.
Let
\[
W:=\{(y,\bar{\wec{x}},\wec{x})\in\bar{B}^{\ee+2}\times\prod_{i=\ee+2}^n  B_i:
(y,\bar{\wec{x}})\in D\ \ \text{and}\ \ (y,\wec{x})\in B'\}.
\]
Our goal is to show that
\begin{enumerate}
\item[(i)]
The projection map $W\to \proj_{[n+1]\setminus\{1\}}(W)$ that omits 
the first coordinate of $W$ 
is one-to-one, and
\item[(ii)]
its image is
$B$, that is, 
\begin{equation}
\label{eq-ontoB}
\proj_{[n+1]\setminus\{1\}}(W)=B.
\end{equation}
\end{enumerate}
(i)--(ii) will imply the statement of Claim~\ref{clm-entailB}
for the following reason.
By (ii), $W$ is nonempty, and it
is easy to see from the definition of $W$ that 
$W$ is definable
by a quantifier-free primitive positive formula, using 
$B_1,\dots,B_n$, $B'$ and $D$. 
Hence 
$\{B_1,\dots,B_n,B',D\}\dentails W$.
By 
(i) and (ii),
$B$ is obtained from $W$ by bijective projection, so
$\{W\}\dentails B$.
By the transitivity of $\dentails$ we get that
$\{B_1,\dots,B_n,B',D\}\dentails B$, 
as claimed.

To prove (i) observe that the
projection map $W\to \proj_{[n+1]\setminus\{1\}}(W)$ is one-to-one,
because for every element $(y,\bar{\wec{x}},\wec{x})\in W$ we have
$(y,\bar{\wec{x}})\in D$, and hence $y$ is uniquely determined
by $\bar{\wec{x}}$, via \eqref{eq-y}.

To prove the inclusion $\supseteq$ in \eqref{eq-ontoB}, let
$(\bar{\wec{x}},\wec{x})\in B$ and define $y\in\bar{B}$ by \eqref{eq-y}.
Then Claim~\ref{clm-y} implies that $(y,\dots,y,\wec{x})\in B$,
so $(y,\wec{x})\in B'$.
On the other hand, by Claim~\ref{clm-alpha}(2) we have that 
$x_1\equiv_{\bar{\alpha}}\dots\equiv_{\bar{\alpha}}x_{\ee+1}$,
therefore $(y,\bar{\wec{x}})\in D$.
This shows that $(y,\bar{\wec{x}},\wec{x})\in W$, and hence
$(\bar{\wec{x}},\wec{x})\in \proj_{[n+1]\setminus\{1\}}(W)$.

For the inclusion $\subseteq$ in \eqref{eq-ontoB}, assume that 
$(\bar{\wec{x}},\wec{x})\in \proj_{[n+1]\setminus\{1\}}(W)$.
Then there exists $y\in\bar{B}$ such that 
$(y,\bar{\wec{x}},\wec{x})\in W$. Let us fix such a $y$.
By the definition of $W$ it follows that 
$(y,\bar{\wec{x}})\in D$ and $(y,\wec{x})\in B'$.
The latter implies that $(y,\dots,y,\wec{x})\in\al B$, while the 
former implies that 
$x_1\equiv_{\bar{\alpha}}\dots\equiv_{\bar{\alpha}}x_{\ee+1}$
and the equality in \eqref{eq-y} holds for $y$.
Since $d$ is idempotent, we get that 
$x_1\equiv_{\bar{\alpha}}\dots\equiv_{\bar{\alpha}}x_{\ee+1}
\equiv_{\bar{\alpha}}y$, so 
$(\bar{\wec{x}},\wec{x})\equiv_\alpha(y,\dots,y,\wec{x})$.
Since $(y,\dots,y,\wec{x})\in\al B$, this shows that
$(\bar{\wec{x}},\wec{x})\in B[\alpha]$.
Therefore Claim~\ref{clm-y} applies, and  we obtain that 
$(\bar{\wec{x}},\wec{x})\in B$.
\hfill$\diamond$

\medskip

Now we are ready to prove the statements (1)--(3) of Lemma~\ref{lm-main}.
We will use the notation and the conclusions of
Claims~\ref{clm-B'}--\ref{clm-entailB}.
There are two cases to consider. 

\smallskip

{\sc Case 1:} $\bar{\alpha}'\in\Con(\bar{\al B}')$ is nontrivial.

Then we know from Claim~\ref{clm-B'} that 
$B'\in\rel{R}^*(\al A)$ is a compatible relation of $\al A$ of $*$-arity
$n-\ee$, so (1) holds. For (2), choose $\al D'$ to be any subalgebra 
of $\bar{\al B}^2$, say $\al D'=\bar{\al B}^2$.
Since $\bar{\al B}=\al B_1$, Claim~\ref{clm-D} proves that 
(2) also holds. Finally, (3) follows from Claim~\ref{clm-entailB}
(the choice of $D'$ is irrelevant).

\smallskip

{\sc Case 2:} $\bar{\alpha}'\in\Con(\bar{\al B}')$ is trivial.

By Claim~\ref{clm-B'},
$\bar{\al B}'$ is an $\mathcal{O}$-subalgebra of $\bar{\al B}$,
therefore $o^{\bar{\al B}'}=o^{\bar{\al B}}$ for all $o\in\mathcal{O}$.
Since the elements $o^{\bar{\al B}}$ ($o\in\mathcal{O}$)
represent all $\bar{\alpha}$-classes of $\bar{\al B}$ and
$\bar{\alpha}'=\bar{\alpha}\restr_{\bar{\al B}'}$ is a trivial
congruence of $\bar{\al B}'$, we get that the underlying set
of $\bar{\al B}'$ is $\bar{B}'=\mathcal{O}^{\bar{\al B}}\,
(=\mathcal{O}^{\bar{\al B}'})$,
and $o_1^{\bar{\al B}}=o_2^{\bar{\al B}}$ 
whenever $o_1^{\bar{\al B}}\equiv_{\bar{\alpha}}o_2^{\bar{\al B}}$ 
in $\bar{\al B}$
($o_1,o_2\in\mathcal{O}$).
It follows that the assignment $o^{\bar{\al B}}\mapsto o^{\bar{\al B}}/\bar{\alpha}$
is an isomorphism $\bar{\al B}'\to\bar{\al B}/\bar{\alpha}$. Hence
the map $\nu\colon\bar{\al B}\to\bar{\al B}'$ that assigns to every
$\bar{b}\in\bar{B}$ the unique element 
$o^{\bar{\al B}}\in\mathcal{O}^{\bar{\al B}}=\bar{B}'$
such that $o^{\bar{\al B}}\equiv_{\bar{\alpha}}\bar{b}$
is a well-defined homomorphism with kernel $\bar{\alpha}$.

Now let $\al D'$ denote the subalgebra of $\bar{\al B}^2$ whose universe
\[
D'=\{(x,y)\in \bar{B}^2:y=\nu(x)\}
\] 
is the graph of $\nu$, and let 
$\al B''=\proj_{[n-\ee]\setminus\{1\}}(\al B')$ 
be the image of $\al B'$ under the 
projection homomorphism onto all coordinates except the first.
Since $\bar{\al B}=\al B_1$, 
Claim~\ref{clm-D} and the definition of $\al D'$
imply that condition (2) of Lemma~\ref{lm-main} holds 
for $\al D$ and $\al D'$.
It follows from the conclusions of Claim~\ref{clm-B'} that
$B''$ is a compatible relation of $\al A$ that belongs to 
$\mathcal{R}^*(\al A)$ and has $*$-arity $n-\ee-1$.
Thus, condition (1) of Lemma~\ref{lm-main} holds for $B''$ in place of $B'$.
It remains to establish that condition (3) of Lemma~\ref{lm-main}
also holds with $B''$ replacing $B'$, that is,
\begin{equation}
\label{eq-modified-entailB}
\{B_1,\dots,B_n,B'',D,D'\}\dentails B.
\end{equation}

First we prove that 
\begin{equation}
\label{eq-partial-dentails}
B'=\{(y,x_{\ee+2},\dots,x_n)\in \bar{B}\times B'': (x_{\ee+2},y)\in D'\}.
\end{equation}
To verify the inclusion $\subseteq$, 
let $(y,\wec{x})=(y,x_{\ee+2},\dots,x_n)$ 
be an arbitrary element of $\al B'$.
Clearly, $(y,\wec{x})\in \bar{B}\times B''$, so we need to
show that $(x_{\ee+2},y)\in D'$, that is, $y=\nu(x_{\ee+2})$.
By Claim~\ref{clm-B'}, $\al B'$ is a subdirect $\mathcal{O}$-subalgebra
of $\bar{\al B}'\times\al B_{\ee+2}\times\dots\times\al B_n$, and the
elements $o^{\al B'}=(o^{\bar{\al B}},o^{\al B_{\ee+2}},\dots,o^{\al B_{n}})$
represent all $\alpha'_{\al B'}$-classes of $\al B'$.
Also,
recall from \eqref{eq-B-s}--\eqref{eq-o-s}
that at the beginning of the proof of Lemma~\ref{lm-main}
we arranged that 
$\al B_{\ee+2}=\bar{\al B}$, $\alpha_{\ee+2}=\bar{\alpha}$, and
$o^{\al B_{\ee+2}}=o^{\bar{\al B}}$ for all $o\in\mathcal{O}$.
Therefore, for the given element $(y,\wec{x})\in\al B'$, 
there exists $o\in\mathcal{O}$
such that $o^{\al B'}\in(y,\wec{x})/\alpha'$, and so
$y\equiv_{\bar{\alpha}'} o^{\bar{\al B}}=o^{\al B_{\ee+2}}\equiv_{\alpha_{\ee+2}} x_{\ee+2}$.
Since $\bar{\alpha}'$ is trivial and 
$\al B_{\ee+2}=\bar{\al B}$, $\alpha_{\ee+2}=\bar{\alpha}$, we get that
$y=o^{\bar{\al B}}\equiv_{\bar{\alpha}} x_{\ee+2}\, (\in\bar{B})$.
Thus, $y=\nu(x_{\ee+2})$. 
This completes the proof of $\subseteq$ in
\eqref{eq-partial-dentails}.

For $\supseteq$, assume that
$(y,x_{\ee+2},\dots,x_n)\in \bar{B}\times B''$ is such that
$(x_{\ee+2},y)\in D'$, that is, $y=\nu(x_{\ee+2})$.
Since $(x_{\ee+2},\dots,x_n)\in B''=\proj_{[n-\ee]\setminus\{1\}}(B')$,
there exists $y'\in\bar{B}'$ such that
$(y',x_{\ee+2},\dots,x_n)\in B'$.
The inclusion $\subseteq$ in \eqref{eq-partial-dentails}
proved in the preceding paragraph implies that 
$y'=\nu(x_{\ee+2})$.
Thus $y=y'$ and $(y,x_{\ee+2},\dots,x_n)=(y',x_{\ee+2},\dots,x_n)\in B'$.
This completes the proof of \eqref{eq-partial-dentails}. 

By \eqref{eq-partial-dentails}, $B'$ is definable by 
a quantifier-free primitive positive formula, using 
$\bar{B}=B_1$, $B''$ and $D'$. Hence $\{B_1,B'',D'\}\dentails B'$.
It follows that
\[
\{B_1,\dots,B_n,B'',D,D'\}\dentails \{B_1,\dots,B_n,B',D\}.
\]
Combining this with the result of Claim~\ref{clm-entailB}, and using  
the transitivity of $\dentails$, we obtain the desired conclusion
\eqref{eq-modified-entailB}.

The proof of Lemma~\ref{lm-main} is complete.   
\end{proof}

\begin{cor}
\label{cor-main}
If some power $\al A^p$ of $\al A$ with $1\le p\le \pp$ has a subalgebra 
that is isomorphic to a section of $\al A$ with
a nontrivial abelian congruence, then
$\rel{R}_{\pp\hh(\ee+1)}(\al A)\dentails\rel{R}^*(\al A)$.
\end{cor}

\begin{proof}
As we noted in Remarks~\ref{rem-about-rhqc}%
(b),
the assumption that some $\al A^p$ with
$1\le p\le \pp$ has a subalgebra isomorphic to a section of $\al A$
with a nontrivial abelian congruence implies that 
$\hh\ge1$ and $\ee\ge2$. 

To prove the corollary we have to show that
\begin{equation}
\label{eq-ehq-B}
\rel{R}_{\pp\hh(\ee+1)}(\al A)\dentails B
\end{equation}
for every $B\in\rel{R}^*(\al A)$.
We proceed by induction on the $*$-arity $n$ of $B$.
We will use the same notation for the data 
$\al B_i$, $\alpha_i$ ($i\in[n]$) and $\alpha$ 
associated to $B$
as in (I)--(IV).
If $n\le\hh(\ee+1)$, 
then the arity of $B$ (as a compatible relation of $\al A$)
is $\sum_{i=1}^n p_i\le n\pp$, therefore $B\in\rel{R}_{\pp\hh(\ee+1)}(\al A)$ and
\eqref{eq-ehq-B} is trivial.
So assume that $n>\hh(\ee+1)$. 
The algebra $\al B$ is the intersection of a family
$\mathcal{I}$ of $\cap$-irreducible
subalgebras of $\al B[\alpha]$.
For every $\hat{\al B}\in\mathcal{I}$ we have 
$\al B\le\hat{\al B}\le\al B[\alpha]$, therefore
$\hat{\al B}$ is a subdirect product of $\al B_1,\dots,\al B_n$
and
$\hat{\al B}[\alpha]=\al B[\alpha]$. The latter implies that
the index of $\alpha_{\hat{\al B}}$ in $\hat{\al B}$ is the same as the
index of $\alpha_{\al B}$ in $\al B$, because both are equal to the index of
$\alpha_{\al B[\alpha]}$ in $\al B[\alpha]$.
This shows that for every $\hat{\al B}\in\mathcal{I}$,
$\hat{B}$ is a relation in $\rel{R}^*(\al A)$ with $*$-arity $n$.
Since $\cap$ is an $\dentails$-construct, it suffices
to prove \eqref{eq-ehq-B} for the case when
$\al B$ is a $\cap$-irreducible subalgebra of $\al B[\alpha]$.
Then, by Lemma~\ref{lm-main},
$\{B_1,\dots,B_n,B',D,D'\}\dentails B$ 
for some $B'\in\rel{R}^*(\al A)$ of $*$-arity
$\le n-\ee$ and some compatible 
relations $D$ and $D'$ 
of $\al A$
that are universes
of algebras $\al D\le\al B_i^{\ee+1}\le\al A^{p_i(\ee+1)}$ 
and $\al D'\le\al B_i^2\le\al A^{2p_i}$ 
for some $i$
($i\in[n]$).
Now the induction hypothesis implies that
$\rel{R}_{\pp\hh(\ee+1)}(\al A)\dentails B'$, because $n-\ee<n$, 
while
the inequalities $\hh\ge1$, $\ee\ge2$ from
the first paragraph of this 
proof imply that
$B_1,\dots,B_n,D,D'\in\rel{R}_{\pp\hh(\ee+1)}(\al A)$, because 
$p_i\le 2p_i\le p_i(\ee+1)\le\pp(\ee+1)\le\pp\hh(\ee+1)$.
Hence 
\[
\rel{R}_{\pp\hh(\ee+1)}(\al A)\dentails \{B_1,\dots,B_n,B',D,D'\}\dentails B,
\]
so \eqref{eq-ehq-B} 
follows by the transitivity of $\dentails$.
\end{proof}

Now we are ready to prove Theorem~\ref{thm-main}.

\begin{proof}[Proof of Theorem~\ref{thm-main}]
Assume that $\al A$ satisfies the hypotheses of Theorem~\ref{thm-main},
that is, in addition to our global assumptions in this section
that $\al A$ is a finite algebra with a $k$-parallelogram term, $\al A$ 
also satisfies the split centralizer condition. (For the definition
of the split centralizer condition, see the Introduction.)
Our aim is to prove that
\begin{equation}
\label{eq-ToProve2}
\rel{R}_\ccc(\al A)\dentails\rel{R}(\al A).
\end{equation}
In addition to the notation
$\rel{R}^*(\al A)$ introduced before Lemma~\ref{lm-main}, we will
write $\rel{R}_{\crit}(\al A)$ for the set of all critical 
relations of $\al A$. (Critical relations are defined in Section~\ref{prelim},
subsection~\ref{prelim}.2.) 

First we will argue that 
\begin{equation}
\label{eq-crit-entails}
\rel{R}_{\crit}(\al A)\dentails\rel{R}(\al A).
\end{equation}
Every compatible relation of $\al A$ is an intersection of
$\cap$-irreducible compatible relations. Furthermore, if a
$\cap$-irreducible compatible relation of $\al A$ is not critical
(i.e., not directly indecomposable), then up to a permutation
of coordinates, it has the form $\rho\times A^\ell$, so it follows
that $\rho$ is a critical relation of $\al A$; 
$A$ itself is clearly a critical relation of $\al A$.
Therefore every compatible relation of $\al A$ can be obtained from
critical relations of $\al A$ by product and intersection. 
Since product and intersection are $\dentails$-constructs, 
\eqref{eq-crit-entails} follows.

Theorem~\ref{thm-reduction} implies that
\begin{equation}
\label{eq-*-entails}
\rel{R}_{\ccc_0}(\al A)\cup\rel{R}^*(\al A)\dentails\rel{R}_{\crit}(\al A),
\end{equation}
as we will show now.
Let $C$ be a critical relation of $\al A$ of arity $n$.
It is clear that 
$\rel{R}_{\ccc_0}(\al A)\cup\rel{R}^*(\al A)\dentails C$
if $n\le \ccc_0$ 
(i.e., if $C\in\rel{R}_{\ccc_0}(\al A)$), so let us assume that 
$n>\ccc_0$. It follows from the definition of $\ccc_0$ that
$n\ge\max(3,k)$. 
Since $\al A$ has a $k$-parallelogram term
and satisfies the split centralizer condition,  
all hypotheses of Theorem~\ref{thm-reduction} are satisfied.
Hence the theorem yields the existence of a compatible relation
$B$ of $\al A$ such that $B\in\rel{R}^*(\al A)$ and
$\rel{R}_{1+\pp}(\al A)\cup\{B\}\dentails C$.
Since $1+\pp\le\ccc_0$, this 
implies that $\rel{R}_{\ccc_0}(\al A)\cup\rel{R}^*(\al A)\dentails C$,
and
proves \eqref{eq-*-entails}.

If no power $\al A^p$ ($1\le p\le \pp$) of $\al A$ has a subalgebra 
that is isomorphic to a section of $\al A$ with a nontrivial
abelian congruence, then $\rel{R}^*(\al A)=\emptyset$ and 
$\Mod(\al A,\pp)=\emptyset$. As we saw in Remarks~\ref{rem-about-rhqc}%
(b),
the latter implies that $\hh=0$ and hence $\ccc=\ccc_0$.
Thus, in this case \eqref{eq-crit-entails} and \eqref{eq-*-entails} combine to
show that
\[
\rel{R}_{\ccc}(\al A)
=\rel{R}_{\ccc_0}(\al A)
\stackrel{\eqref{eq-*-entails}}{\dentails}\rel{R}_{\crit}(\al A)
\stackrel{\eqref{eq-crit-entails}}{\dentails}\rel{R}(\al A).
\]
Hence  \eqref{eq-ToProve2}
follows by the transitivity of $\dentails$.

In the opposite case, when some $\al A^p$ ($1\le p\le\pp$)
has a subalgebra isomorphic to a section of $\al A$ with 
a nontrivial abelian congruence, 
Corollary~\ref{cor-main} shows that
\begin{equation}
\label{eq-ehq-entails}
\rel{R}_{\pp\hh(\ee+1)}(\al A)\dentails\rel{R}^*(\al A).
\end{equation}
Since $\ccc=\max(\ccc_0,\pp\hh(\ee+1))$, we have that
$\rel{R}_{\ccc_0}(\al A)\cup\rel{R}_{\pp\hh(\ee+1)}(\al A)\subseteq
\rel{R}_{\ccc}(\al A)$,
so 
\eqref{eq-crit-entails}, \eqref{eq-*-entails}, and 
\eqref{eq-ehq-entails} together imply that
\[
\rel{R}_{\ccc}(\al A)
\stackrel{\eqref{eq-ehq-entails}}{\dentails}
    \rel{R}_{\ccc}(\al A)\cup\rel{R}^*(\al A)
\stackrel{\eqref{eq-*-entails}}{\dentails} \rel{R}_{\crit}(\al A)
\stackrel{\eqref{eq-crit-entails}}{\dentails} \rel{R}(\al A).
\]
By the transitivity of $\dentails$ this shows that 
\eqref{eq-ToProve2} holds.

This proves \eqref{eq-ToProve2} in all cases.
In view of Theorem~\ref{thm-WZ}, \eqref{eq-ToProve2}
is sufficient to conclude that $\al A$ is dualizable.
\end{proof}

\section{Applications}\label{appl}

In this section we apply the main
theorem of the paper to establish dualizability within
some well known classes of algebras.
Some of these results were known before.

In the first part of the section we identify
some conditions on a variety $\mathcal V$
which guarantee that every finite member of $\mathcal V$
is dualizable. We end the section by proving that
if $\m a$ is a finite algebra with a parallelogram term
and $\Su\Pd(\m a)$ is a variety, then $\m a$ is dualizable
(although some of the other algebras in 
$\Su\Pd(\m a)$ need not be dualizable).

We will often apply Theorem~\ref{thm-main} to prove dualizability
simultaneously for all members of a class $\mathcal{C}$ of finite algebras
with parallelogram terms,
where $\mathcal{C}$ is closed under taking subalgebras.
In such cases, to get the desired conclusion,  
it will be enough to check that for every $\al A\in\mathcal{C}$,
each relevant triple $(\delta,\theta,\nu)$ of $\al A$ is split
by a triple $(\alpha,\beta,\kappa)$ relative to $\Su\Pd(\al A)$.
Indeed, if every $\al A\in\mathcal{C}$ has this property, then
for every $\al A\in\mathcal{C}$ and $\al B\le\al A$ we have that
$\al B\in\mathcal{C}$, so every relevant triple of $\al B$ is split
by a triple relative to $\Su\Pd(\al B)$, and therefore relative to
$\Su\Pd(\al A)$ as well. 
This shows that every $\al A\in\mathcal{C}$ satisfies the split centralizer
condition, which implies by Theorem~\ref{thm-main} that every 
$\al A\in\mathcal{C}$ is dualizable.

We will refer to some commutator
identities by the number assigned to them
in \cite[Chapter~8]{freese-mckenzie}:
\begin{align}
[x\wedge y,y]&= x\wedge[y,y],\tag{C1}\\
[x,y]&= x\wedge y.\tag{C3}\\
[1,x]&= x.\tag{C8}
\end{align}

\noindent
An algebra satisfying a given (Ci), $i\in \{1,3,8\}$, 
may be called a 
\emph{(Ci)-algebra}. It is known that if 
$\mathcal V$ is congruence modular, then the subclass
of (Ci)-algebras in $\mathcal V$
is closed under the formation of finite subdirect products
and quotient algebras. (See 
\cite[Chapter~8]{freese-mckenzie}.)

We have met (C1) before: any residually 
small congruence modular variety consists
of 
(C1)-algebras, 
and conversely any 
congruence modular variety
generated by a finite algebra whose subalgebras
are all (C1) is residually small.

(C3)-algebras are also called \emph{neutral}.
{}From \cite[Chapter~8]{freese-mckenzie}
we know that an algebra is neutral if and only if
it has no nontrivial abelian congruence intervals.
Therefore, we will call an interval in a congruence lattice 
\emph{neutral} if it has
no nontrivial abelian subintervals.

\begin{lm}\label{solvable-or-perfect}
Let $\mathcal V$ be a congruence modular
variety in which every finite
subdirectly irreducible algebra is 
either solvable or is a (C8)-algebra.
If $\m a\in {\mathcal V}$ is a finite algebra,
then 
\begin{enumerate}
\item[{\rm(1)}]
$\m a$ 
has a unique pair $(\sigma,\rho)$
of complementary factor congruences such that 
$\m a\cong \m a/\sigma\times \m a/\rho$, 
$\m a/\sigma$ is solvable and $\m a/\rho$ is a (C8)-algebra; moreover
\item[{\rm(2)}]
every congruence $\chi$ on $\al A$
is a product congruence relative to the factorization
$\m a\cong \m a/\sigma\times \m a/\rho$, meaning that 
$\chi = (\chi\vee\sigma)\wedge(\chi\vee\rho)$.
\end{enumerate}
\end{lm}

\begin{proof}{}
To prove (1)
recall that the classes
of solvable algebras and (C8)-algebras in $\mathcal V$
are closed under finite subdirect products and quotients.
Therefore, if
$\m a\in \mathcal V$ is any finite
algebra, then it has a least congruence
$\sigma$ such that $\m a/\sigma$ is solvable
and also a least congruence $\rho$ such that
$\m a/\rho$ is a (C8)-algebra. 
The congruence
$\sigma$ is contained in the kernel of any 
homomorphism of $\m a$ onto a solvable subdirectly irreducible
algebra and the congruence $\rho$
is contained in the kernel of any 
homomorphism of $\m a$ onto a (C8) subdirectly irreducible
algebra. Hence $\sigma\wedge \rho$ is contained in the kernel
of any homomorphism of $\m a$ onto a subdirectly irreducible algebra,
implying that $\sigma\wedge\rho = 0$. 

The algebra $\m a/(\sigma\vee\rho)$ is a quotient 
of the solvable algebra $\m a/\sigma$ and is also 
a quotient of the (C8)-algebra $\m a/\rho$,
so it is both solvable and (C8). This forces it to be trivial,
and therefore $\sigma\vee\rho = 1$.

Congruences $\sigma$ and $\rho$ permute, 
since $\sigma$ is cosolvable and cosolvable
congruences permute with all congruences 
by \cite[Theorem~6.2]{freese-mckenzie}. 
This completes the proof
that $(\sigma,\rho)$ is a pair of complementary factor
congruences, and the argument shows that 
$\m a/\sigma$ is solvable and $\m a/\rho$ is (C8).

If $(\sigma',\rho')$ were a second pair 
of complementary factor congruences such that 
$\m a\cong \m a/\sigma'\times \m a/\rho'$, 
$\m a/\sigma'$ is solvable and $\m a/\rho'$ is a (C8)-algebra,
then the solvability of $\m a/\sigma'$ implies that
$\sigma'\supseteq \sigma$,
and a similar argument shows that $\rho'\supseteq\rho$. 
But if $(\sigma,\rho)$ and $(\sigma',\rho')$
are pairs of complementary factor congruences where
$\sigma\subseteq \sigma'$ and $\rho\subseteq \rho'$,
then $\sigma=\sigma'$ and $\rho=\rho'$. 

For (2) notice that
the set of product congruences on $\m a\cong \m a/\sigma\times \m a/\rho$
is closed under meet, so to prove 
claim (2)
it suffices to show that 
the meet irreducible congruences on $\m a$ are product
congruences. Each one has been shown to be above $\sigma$ or $\rho$,
so is in fact a product congruence.
\end{proof}

Under the
assumptions of Lemma~\ref{solvable-or-perfect} a congruence 
$\chi$ of $\al A$ is central
(i.e., satisfies $[1,\chi]=0$) if and only if
the congruence
$\bar{\chi}:=(\chi\vee\sigma)/\sigma$ of $\al A/\sigma$ is central
and $\chi\le\rho$.
This can be verified as follows. 
By Lemma~\ref{solvable-or-perfect}(2),
$\chi$ is a product congruence relative to the factorization
$\al A\cong\al A/\sigma\times\al A/\rho$, therefore it follows that
$\chi$ is a central congruence of $\al A$ if and only if
$\bar{\chi}:=(\chi\vee\sigma)/\sigma$ is a central congruence of $\al A/\sigma$
and $(\chi\vee\rho)/\rho$ is a central congruence of $\al A/\rho$.
But $\al A/\rho$ is a (C8)-algebra, therefore 
$(\chi\vee\rho)/\rho$ is central if and only if it is the trivial congruence
of $\al A/\rho$, that is, $\chi\le\rho$.
 
The statement in the preceding paragraph implies that 
if $\zeta$ is the center of
$\al A$, then $\bar{\zeta}=(\zeta\vee\sigma)/\sigma$ is the center of
$\al A/\sigma$ and $\zeta\le\rho$. 
Consequently, 
if $0=:\zeta_0\le\zeta=:\zeta_1\le\zeta_2\le\dots$ is the ascending
central series of $\al A$, then
$0=\bar{\zeta}_0\le\bar{\zeta}=\bar{\zeta_1}\le\bar{\zeta_2}\le\dots$ 
is the ascending central series of $\al A/\sigma$ and $\zeta_i\le\rho$
for all $i$.

\begin{cor}
\label{nilpotent-or-perfect}
If $\al A$ is a finite algebra in a congruence modular variety $\var{V}$
such that every finite subdirectly irreducible algebra in $\var{V}$
is either nilpotent or
(C8), then in the factorization 
$\al A\cong\al A/\sigma\times\al A/\rho$ in Lemma~\ref{solvable-or-perfect}
the first factor
$\m a/\sigma$ is nilpotent,
and $\sigma$ and $\rho$ are the 
final congruences in the descending and
ascending central series of $\al A$, respectively.
\end{cor}

\begin{proof}
By the construction in Lemma~\ref{solvable-or-perfect},
$\sigma$ is the least congruence such that $\al A/\sigma$ is
solvable, so
$\al A/\sigma$ is a subdirect product of solvable subdirectly
irreducible algebras in $\var{V}$. 
Since every solvable subdirectly irreducible 
algebra in $\var{V}$ is nilpotent,
we get that $\al A/\sigma$ is nilpotent, and
$\sigma$ is the least congruence such that $\al A/\sigma$ is nilpotent. 
Hence $\sigma$ is the final congruence in the descending central series
of $\al A$.

To prove our claim on $\rho$, let 
$0=\zeta_0\le\zeta=\zeta_1\le\zeta_2\le\dots$ be the ascending
central series of $\al A$. As we saw earlier, this implies that
$0=\bar{\zeta}_0\le\bar{\zeta}=\bar{\zeta_1}\le\bar{\zeta_2}\le\dots$ 
is the ascending central series of $\al A/\sigma$ and $\zeta_i\le\rho$
for all $i$.
Since $\al A/\sigma$ is nilpotent, $\bar{\zeta_c}=1$ for some $c$,
so $\zeta_c\vee\sigma=1$ and $\zeta_c\le\rho$. 
This implies that
$\rho=1\wedge\rho=(\zeta_c\vee\sigma)\wedge\rho
\stackrel{\bf mod}{=}\zeta_c\vee(\sigma\wedge\rho)=\zeta_c\vee 0=\zeta_c$.
Moreover, $\rho=\zeta_c\le\zeta_i\le\rho$ for every $i\ge c$, therefore
$\rho=\zeta_c$ is the final congruence in the ascending central series 
of $\al A$.
\end{proof}

We will only use the following special case.

\begin{cor}
\label{abelian-or-perfect}
If $\al A$ is a finite algebra in a congruence modular variety $\var{V}$
such that every finite subdirectly irreducible algebra in $\var{V}$
is either abelian or (C8), then in the factorization 
$\al A\cong\al A/\sigma\times\al A/\rho$ in Lemma~\ref{solvable-or-perfect}
the first factor
$\m a/\sigma$ is abelian,
$\sigma=[1,1]$ is the derived congruence,
and $\rho=\zeta$ is the center.
\end{cor}

\begin{proof}
Restricting the argument in the proof of Corollary~\ref{nilpotent-or-perfect}
to the case when all finite nilpotent subdirectly irreducible algebras 
are abelian yields that $\al A/\sigma$ is abelian and
$\sigma$ is the least congruence such that $\al A$ is abelian, so
$\sigma=[1,1]$.
Moreover, the center of $\al A/\sigma$ is $1=\bar{\zeta}=\bar{\zeta_1}$,
showing that we can use $c=1$ in the argument.
Therefore $\rho=\zeta_1=\zeta$ is the center of $\al A$.
\end{proof}

We extend the ``neutral'' terminology by calling
a subdirectly irreducible algebra $\m a$ \emph{almost neutral}
if it is nonabelian and fails (C3)
($[x,y]=x\wedge y$) in exactly one way:
$[\mu,\mu]=0$, where $\mu$ is the monolith of $\m a$.
Equivalently, a subdirectly irreducible algebra
$\m a$ with monolith $\mu$ is almost neutral if it is nonabelian
and
\begin{enumerate}
\item[(i)] $(0:\mu)=\mu$, and 
\item[(ii)] the congruence interval $\interval{\mu}{1}$ is neutral.
\end{enumerate}

\begin{thm} \label{maincor}
Let $\mathcal V$ be a variety with a parallelogram term.
Assume that every finite subdirectly 
irreducible algebra
in ${\mathcal V}$ is abelian, neutral or almost neutral.
Then every finite algebra in $\mathcal V$ is dualizable.
\end{thm}

\begin{proof}
Any neutral or almost neutral subdirectly irreducible
algebra is a (C8)-algebra, hence the finite subdirectly
irreducible algebras in $\mathcal V$
are abelian or (C8). 
{}From Corollary~\ref{abelian-or-perfect}
we have that
any finite algebra $\m a\in \mathcal V$ has 
the properties that 
\begin{enumerate}
\item[(a)] 
$[1,1]$ and $\zeta$
are complementary factor congruences, 
\item[(b)] $\m a/[1,1]$ is abelian and
$\m a/\zeta$ is a (C8)-algebra, and
\item[(c)] every meet irreducible congruence
on $\m a$ lies above $[1,1]$ or $\zeta$.
\end{enumerate}
Our task is to verify
that for every
relevant triple $(\delta,\theta,\nu)$
of $\al A$, where 
$\delta$ is a meet irreducible congruence,
$\delta\prec\theta$,
$\nu = (\delta:\theta)$, and
$\theta/\delta$ is abelian,
there exists a triple $(\alpha,\beta,\kappa)$
which splits $(\delta,\theta,\nu)$ relative to
$\var{Q}:=\Su\Pd(\al A)$.
The conditions (i)--(v) that define what this means
for $(\alpha,\beta,\kappa)$ are listed
in the Introduction.

\smallskip

{\sc Case 1:} $[1,1]\leq \delta$.

We have $[1,\theta]\leq [1,1] \leq \delta$,
proving that $1 \leq (\delta:\theta)$. Hence
$1 = (\delta:\theta)=\nu$. 
If we choose $(\alpha,\beta,\kappa)=(\zeta,[1,1],0)$, 
then we have 
\begin{enumerate}
\item[(i)] 
$\kappa=0$ is a $\var{Q}$-congruence,
\item[(ii)]
$\beta=[1,1]\leq \delta$,
\item[(iii)] 
$\alpha\wedge\beta = \zeta\wedge[1,1] = 0 = \kappa$,
\item[(iv)] 
$\alpha\vee\beta = \zeta\vee[1,1] = 1 = 
\nu$, and
\item[(v)] 
$[\alpha,\alpha]=[\zeta,\zeta]=0\leq \kappa$.
\end{enumerate}
Hence 
the conditions required for $(\alpha,\beta,\kappa)$ 
are met.

\smallskip

{\sc Case 2:} $\zeta\leq \delta$.

In this case, $\m a/\delta$ is a nonabelian subdirectly
irreducible algebra. Since $(\delta,\theta,\nu)$ is relevant
it must be that the monolith $\mu=\theta/\delta$
of $\m a/\delta$ is abelian. 
$\m a/\delta$ must be almost neutral, hence
\[
\theta/\delta\leq \nu/\delta=(0:\mu) = \mu=\theta/\delta,
\]
and therefore $\nu=\theta$.

In this case it is our aim to show that there is a
congruence $\gamma$ covering $\zeta$ such that 
$\gamma\not\leq \delta$.
For such a congruence we have a perspectivity
$\interval{\zeta}{\gamma}\nearrow\interval{\delta}{\theta}$.
In this situation $(\alpha,\beta,\kappa)=(\gamma\wedge[1,1],\delta,0)$
is a splitting triple for $(\delta,\theta,\nu)$, since
\begin{enumerate}
\item[(i)] 
$\kappa=0$ is a $\var{Q}$-congruence,
\item[(ii)]
$\beta= \delta$ (we need only $\beta\leq \delta$ here),
\item[(iii)]
$\beta\wedge\alpha = \delta\wedge\gamma\wedge[1,1] = 
\zeta\wedge[1,1] = 0 = \kappa$,
\item[(iv)]
$\beta\vee\alpha = 
\delta\vee(\gamma\wedge[1,1]) = 
\delta\vee\zeta\vee(\gamma\wedge[1,1]) \stackrel{\bf mod}{=}
\delta\vee(\gamma\wedge(\zeta\vee[1,1])) = 
\delta\vee\gamma = 
\theta = \nu$, and
\item[(v)]
$[\alpha,\alpha]\leq\kappa$. 
(There are perspectivities
$\interval{\delta}{\theta}\searrow
\interval{\zeta}{\gamma}\searrow
\interval{0}{\alpha}=
\interval{\kappa}{\alpha}$
and the first is abelian, so the last is.)
\end{enumerate}

To reiterate the conclusion just drawn, it suffices
to show that $\m a$ has a congruence $\gamma$ covering
$\zeta$ such that $\gamma\not\leq \delta$.
Since $\delta$ is also above $\zeta$, we may work
modulo $\zeta$ and henceforth assume that 
$\m a$ is a (C8)-algebra.
In this situation, every subdirectly irreducible quotient
of $\m a$ is neutral or almost neutral.
So, we will be done of we prove the following claim.

\begin{clm}
\label{clm-atomic}
If $\m a$ is a finite (C8)-algebra in $\mathcal V$,
and $(\delta,\theta,\nu)$ is a relevant triple of $\m a$,
then $\al A$ has
an atomic congruence $\gamma$ 
such that $\gamma\not\leq \delta$.
\end{clm}

\noindent
{\it Proof of Claim~\ref{clm-atomic}.}
Let $\delta_i$, $i=1,\dots,n$, be the set of meet irreducible
congruences of $\m a$, and for each $i$ let $\theta_i$
be the upper cover of $\delta_i$. For each $i$ choose
$\nabla_i\in \{\delta_i,\theta_i\}$ according to
\[
\nabla_i = 
\begin{cases}
\delta_i & \textrm{if $\m a/\delta_i$ is neutral};\\
\theta_i & \textrm{if $\m a/\delta_i$ is almost neutral.}
\end{cases}
\]
Equivalently,  
$\nabla_i\in \{\delta_i,\theta_i\}$ is chosen 
as small as possible so that
the interval 
$\interval{\nabla_i}{1}$ 
is neutral.

Let $\nabla = \bigcap_{i=1}^n \nabla_i$. Since each $\m a/\nabla_i$
is neutral and the class
of neutral algebras in $\mathcal V$ 
is closed under finite subdirect products
we get that the interval 
$\interval{\nabla}{1}$ 
is neutral.
$\m a$ itself is not a neutral algebra, since it
has a relevant triple $(\delta,\theta,\nu)$ (and therefore
an abelian congruence interval 
$\interval{\delta}{\theta}$), 
so $0<\nabla$.

For each $i$, define $\Delta_i = \delta_i\wedge \nabla$.
If $\delta_i=\nabla_i\;(\geq \nabla)$ for some $i$, 
then $\Delta_i=\nabla$.
In the alternative case where 
$\theta_i=\nabla_i$, the interval
$\interval{\Delta_i}{1}$ 
contains the abelian prime quotient
$\interval{\delta_i}{\theta_i}$, and $\interval{\nabla}{1}$
contains no such prime quotient, so $\Delta_i\neq\nabla$. 
Hence 
$\nabla\not\le\delta_i$, 
so
there is a perspectivity 
$\interval{\delta_i}{\theta_i}\searrow\interval{\Delta_i}{\nabla}$.
Thus, for each $i$, $\Delta_i=\nabla$ 
if $\al A/\delta_i$ is neutral, and
$\Delta_i$
is a lower cover of $\nabla$ for which the interval 
$\interval{\Delta_i}{\nabla}$ 
is abelian
if $\al A/\delta_i$ is almost neutral.

The intersection $\bigcap_{i=1}^n \Delta_i$ equals 
$\bigcap_{i=1}^n \delta_i = 0$. Since the former is a meet
of lower covers of $\nabla$, and the interval
$\interval{0}{\nabla}$ 
is modular, this interval is complemented.
Thus, for every $i$ where $\Delta_i\prec\nabla$
there is an atom $\gamma_i$ below $\nabla$ such that 
$\interval{0}{\gamma_i}\nearrow\interval{\Delta_i}{\nabla}$. 
Choose the
value of $i$ where $\delta_i$ is the first congruence
in the relevant triple $(\delta,\theta,\nu)$. 
Since
$\theta/\delta$ is abelian, 
$\al A/\delta_i=\al A/\delta$ is almost neutral. 
Therefore
$\interval{0}{\gamma_i}\nearrow
\interval{\Delta_i}{\nabla}\nearrow
\interval{\delta_i}{\theta_i}=
\interval{\delta}{\theta}$. 
For this $i$
we have found an atom $\gamma = \gamma_i$
satisfying $\gamma\not\leq \delta$.
\hfill$\diamond$

\medskip

This completes the proof of the theorem.
\end{proof}

\begin{cor}\label{abelian-or-neutral}
Let $\mathcal V$ be a variety with a parallelogram term.
If every finite subdirectly irreducible algebra in
$\mathcal V$ is abelian or neutral, then 
every finite algebra in $\mathcal V$ is dualizable.
\end{cor}

\begin{proof}
This is Theorem~\ref{maincor} restricted 
to the situation where $\mathcal V$ has no 
finite almost
neutral subdirectly irreducible algebras.
\end{proof}

\begin{cor}[See \cite{davey-werner}]
Any finite algebra with a near unanimity term is dualizable.
\end{cor}

\begin{proof}
A near unanimity term is a parallelogram term,
and any algebra having such a term is neutral.
Hence this corollary is a further 
restriction of Theorem~\ref{maincor}.
(In fact, it is exactly the
 restriction of 
Theorem~\ref{maincor} 
to the situation where $\mathcal V$ has no 
finite abelian or almost
neutral subdirectly irreducible algebras.)
\end{proof}

\begin{cor}
\label{cor-dirrep}
Any finite algebra in a directly representable variety
is dualizable.
\end{cor}

\begin{proof}
This is a corollary to Corollary~\ref{abelian-or-neutral}.
To see this, recall that (i)~a finite algebra
in a directly representable variety has a Maltsev term
(see \cite[Theorem~5.11]{mckenzie_narrowness}),
(ii)~a Maltsev term is a parallelogram term, and
(iii)~the finite nonabelian subdirectly irreducible algebras
in a directly representable variety are simple, hence neutral
(see \cite[Theorem~5.11]{mckenzie_narrowness}).
\end{proof}

\begin{cor}[Cf.~\cite{davey-quackenbush}]
Any finite algebra in a variety generated by a
paraprimal algebra is dualizable.
\end{cor}

\begin{proof}
This follows from Corollary~\ref{cor-dirrep} and
the fact that a variety generated by a paraprimal algebra
is directly representable
(see \cite[Theorem~1.6]{clark-krauss-paraprimal}).
\end{proof}

\begin{cor}
\label{cor-affine}
Any finite affine algebra is dualizable.
In particular, any finite module is dualizable.
\end{cor}

\begin{proof}
This is another restriction of Theorem~\ref{maincor}.
This time we are restricting to the case where all
finite subdirectly irreducible algebras in $\mathcal V$ are abelian.
\end{proof}

Modules and affine algebras for which dualizability was known before
include finite abelian groups (using the restriction of Pontryagin duality, 
see \cite[Chapter~4]{clark-davey}),
finite affine spaces (see \cite{pszczola}), and
finite algebras in a variety generated by a finite simple affine algebra
(see \cite{davey-quackenbush}). All these special cases of 
Corollary~\ref{cor-affine}
are covered also by Corollary~\ref{cor-dirrep}.
However, Corollary~\ref{cor-affine} is not a consequence of
Corollary~\ref{cor-dirrep}, because not every finite module lies in
a directly representable variety.
(A finite faithful $R$-module generates a directly representable variety
if and only if $R$ is of finite representation type.)

\begin{cor}
\label{cor-K-alg}
Let $\mathbb K$ be a commutative unital ring.
Let $\mathcal V$ be a residually small variety
of $\mathbb K$-algebras
(commutative or not, unital or not).
Any finite algebra in $\mathcal V$ is dualizable.
\end{cor}

\begin{proof}
It follows from Theorems~3.1 and 3.2 in \cite{mckenzie_kalg} that 
for every  nonabelian subdirectly irreducible $\mathbb{K}$-algebra 
$S$ in 
a residually small variety the ring reduct of $S$ is 
also subdirectly irreducible and lies in a residually small variety.
The 
possible structure of a subdirectly irreducible ring in a
residually small variety is described in Section~7 of
\cite{mckenzie_kalg}. 
These results imply that
every nonabelian
subdirectly irreducible algebra $S\in \mathcal V$ 
is either simple or else (i) has abelian monolith 
equal to the radical, $J = \textrm{rad}(S)$, and 
(ii) has $S/J$ isomorphic to a field or a product of two fields.
This is enough to guarantee that 
every subdirectly irreducible
algebra in $\mathcal V$ is abelian, neutral or 
almost neutral.
So, our statement follows from Theorem~\ref{maincor}.
\end{proof}

\begin{cor}
\label{cor-ring}
Any finite ring 
(commutative or not, unital or not)
that generates a residually small variety 
is dualizable.
\end{cor}

\begin{proof}
This is Corollary~\ref{cor-K-alg} restricted to the case when
$\mathbb K = \mathbb Z$.
\end{proof}

The special case of Corollary~\ref{cor-ring} when the 
ring is assumed to be 
commutative and unital was proved in \cite{commutative_rings}.
For a commutative ring $R$ with unit, the condition that $R$ generates
a residually small variety is equivalent to the condition that
its radical, $J=\text{\rm rad}(R)$, is abelian (i.e., $J^2=0$).%

\medskip

For a finite group $G$, the condition that 
$G$ generates a residually small variety is equivalent 
to the condition that the Sylow subgroups of $G$ are abelian.
Now we explain how to derive from Theorem~\ref{thm-main} the 
result that any finite group with abelian Sylow subgroups is dualizable.
(This does not follow from Theorem~\ref{maincor}.)

The crucial result is the following.

\begin{thm}[From {\cite[Chapter~2]{nickodemus}}] 
\label{matt-nick}
If $G$ is a finite group with abelian Sylow subgroups
and $(\delta,\theta,\nu)$ is a relevant triple of $G$,
then there is an endomorphism $\varepsilon\colon G\to G$
such that 
\[
[\nu,\nu]\leq \ker(\varepsilon)\leq \delta. 
\]
\end{thm}

\begin{cor}[From {\cite{nickodemus}}]  
\label{groups}
Any finite group with abelian Sylow subgroups is dualizable.
\end{cor}

\begin{proof}
Let $G$ be a finite group with abelian Sylow subgroups.
We have to show that every relevant triple $(\delta,\theta,\nu)$
of $G$ is split by a triple $(\alpha,\beta,\kappa)$
relative to ${\mathcal Q}=\Su\Pd(G)$.
Choose $(\alpha,\beta,\kappa) = (\nu,\ker(\varepsilon),\ker(\varepsilon))$
where $\epsilon$ is the endomorphism from Theorem~\ref{matt-nick}.
Then 
\begin{enumerate}
\item[(i)]
$\kappa$ is a $\mathcal Q$-congruence, since
$\kappa$ is the kernel of an endomorphism of $G$. 
\item[(ii)] 
$\beta=\ker(\varepsilon)\leq \delta$,
\item[(iii)] 
$\alpha\wedge\beta = \nu\wedge\ker(\varepsilon)=
\ker(\varepsilon) = \kappa$,
\item[(iv)] 
$\alpha\vee\beta = \nu\vee\ker(\varepsilon)=\nu$, and
\item[(v)] 
$[\alpha,\alpha]=[\nu,\nu]\leq \ker(\varepsilon)=\kappa$.
\end{enumerate}
Hence the
conditions required for $(\alpha,\beta,\kappa)$
are met.
\end{proof}

\begin{exmp}
Pawe{\l} Idziak proved in \cite{idziak}
that the algebra $\m a$ obtained from the 
six-element symmetric group $S_3$ by adjoining
all six nullary operations is not a dualizable algebra.
This is in contrast to Corollary~\ref{groups},
which establishes that $S_3$ without the additional
constants \emph{is} dualizable. It is worth pointing
out how Idziak's example fails the conditions
of Theorem~\ref{thm-main}. The algebra $\m a$ has
three congruences, $0 < \delta < 1$, where
$\delta$ is the congruence on $S_3$ corresponding
to the alternating group. The triple $(\delta,\theta,\nu)=(\delta,1,1)$
is relevant. For $\mathcal Q =
\Su\Pd(\m a)$, the only $\mathcal Q$-congruences on $\m a$
are $0$ and $1$, and $1\not\leq \delta$,
so if $(\alpha,\beta,\kappa)$ is a splitting triple
for $(\delta,\theta,\nu)$
then 
\begin{enumerate}
\item[(i)]
$\kappa$ must equal $0$. 
\end{enumerate}
To establish dualizability our theorem 
also requires that
\begin{enumerate}
\item[(ii)] 
$\beta\leq \delta$,
\item[(iii)] 
$\alpha\wedge\beta = \kappa \;(= 0)$,
\item[(iv)] 
$\alpha\vee\beta = \nu \;(= 1)$, and
\item[(v)] 
$[\alpha,\alpha]\leq\kappa \;(= 0)$.
\end{enumerate}
But $\delta$ is the largest abelian congruence on $\m a$,
so $\alpha\leq \delta$. This and property 
(ii) 
yield
that $\alpha\vee\beta\leq\delta<1$, contrary to 
(iv).

If we had not added the additional nullary operations,
then the congruence
$\delta$ would be another $\mathcal Q$-congruence and in this situation
$(\alpha,\beta,\kappa)=(1,\delta,\delta)$ would be a splitting
triple for $(\delta,1,1)$.
\end{exmp}

We close with a different kind of application of Theorem~\ref{thm-main}.

\begin{thm}\label{quasivariety}
If $\m a$ is a finite algebra with a parallelogram term
and $\Su\Pd(\m a)$ is a variety, then $\m a$ is dualizable.
\end{thm}

\begin{proof}
If $\mathcal Q:=\Su\Pd(\m a)$ is a variety, then 
it must be residually small, hence satisfies (C1)
(by Theorems~\ref{thm-rs} and \ref{thm-parterm-cm}).
Every 
congruence on every algebra 
in $\var{Q}$
is a $\mathcal Q$-congruence, 
so 
if 
$(\delta,\theta,\nu)$ is relevant, then
it is split by $(\nu,\delta,\delta)$,
because (C1) implies that
$[\nu,\nu]\le\delta$.
\end{proof}

Here it is worth pointing out that if 
$\mathcal V$ is a variety with only finitely
many subdirectly irreducible algebras,
each 
finite and
having a 1-element subalgebra, then the product
$\m a$ of 
all these subdirectly irreducible algebras 
is a finite algebra with the property 
that $\mathcal V = \Su\Pd(\m a)$.
If a variety $\mathcal V$, like this, has a parallelogram term,
then Theorem~\ref{quasivariety},
combined with the main result of \cite{davey-willard},
proves that each quasivariety generator for $\mathcal V$
is dualizable. 

\bigskip
\noindent
{\bf Acknowledgment.}
The research reported in this paper was carried out while visiting 
La Trobe University as a guest of Brian Davey. 
We gratefully acknowledge 
the support and encouragement provided by our host.
We also thank Ross Willard for reading a preliminary draft of the paper and
recommending improvements.

\bibliographystyle{plain}

\end{document}